\documentclass[pdflatex,sn-mathphys-num]{sn-jnl}% Math and Physical Sciences Numbered Reference Style
\usepackage{graphicx}%
\usepackage{multirow}%
\usepackage{amsmath,amssymb,amsfonts,dsfont}%
\usepackage{amsthm}%
\usepackage{mathrsfs}%
\usepackage{mathtools}
\usepackage[title]{appendix}%
\usepackage{xcolor}%
\usepackage{textcomp}%
\usepackage{manyfoot}%
\usepackage{booktabs}%
\usepackage{algorithm}%
\usepackage{algorithmicx}%
\usepackage{algpseudocode}%
\usepackage{listings}%
\usepackage{subcaption} % 
\usepackage{enumitem}
\usepackage{hyperref}
\usepackage{cleveref} % optional but gives nice references
\theoremstyle{thmstyleone}%
\newtheorem{Theorem}{Theorem}%  meant for continuous numbers
\newtheorem{Proposition}[Theorem]{Proposition}
\newtheorem{Corollary}[Theorem]{Corollary}
\newtheorem{Lemma}[Theorem]{Lemma}
\newtheorem{Definition}[Theorem]{Definition}
\newtheorem{Remark}[Theorem]{Remark}
\newtheorem{Example}[Theorem]{Example}

\newtheorem{Algorithm}[Theorem]{Algorithm}

\usepackage{tikz}
\usetikzlibrary{shapes,fit}
\usetikzlibrary{shapes.multipart}
\usetikzlibrary{positioning}
\usetikzlibrary{shapes, arrows, backgrounds, fit, positioning}

\DeclareMathOperator*{\grad}{grad}

\DeclareMathOperator*{\amin}{argmin}
\DeclareMathOperator*{\dom}{dom}
\DeclareMathOperator*{\intr}{int}
\DeclareMathOperator*{\prox}{Prox}
\DeclareMathOperator*{\fix}{Fix}

\newcommand{\tr}[1]{\textcolor{black}{#1}}

\newcommand{\R}{\mathbb{R}}
\newcommand{\M}{\mathcal{M}}
\newcommand{\N}{\mathbb{N}}
\newcommand{\T}{\mathcal{T}}
\newcommand{\hide}[1]{}

\begin{document}

\title[Efficient Douglas-Rachford Methods on Hadamard Manifolds with Applications to the Heron Problems]{Efficient Douglas-Rachford Methods on Hadamard Manifolds with Applications to the Heron Problems}


\author[1]{\fnm{D. R.} \sur{Sahu}}\email{drsahudr@gmail.com}

\author*[1,2]{\fnm{Shikher} \sur{Sharma}}\email{shikhers043@gmail.com}
\equalcont{These authors contributed equally to this work.}

\author[3]{\fnm{Pankaj} \sur{Gautam}}\email{pgautam908@gmail.com}
\equalcont{These authors contributed equally to this work.}

\affil*[1]{\orgdiv{Department of Mathematics}, \orgname{Banaras Hindu University}, \orgaddress{\city{Varanasi}, \country{India}}}

\affil[2]{\orgdiv{Department of Mathematics}, \orgname{The Technion -- Israel Institute of Technology}, \orgaddress{ \city{Haifa}, \postcode{32000}, \country{Israel}}}

\affil[3]{\orgdiv{Department of Applied Mathematics and Scientific Computing}, \orgname{IIT Roorkee}, \orgaddress{\city{Saharanpur}, \country{India}}}


\abstract{		Our interest lies in developing some efficient methods for minimizing the sum of two geodesically convex functions on Hadamard manifolds, with the aim \tr{of improving the convergence} of the Douglas-Rachford algorithm in Hadamard manifolds. Specifically, we propose two types of algorithms: inertial and non-inertial algorithms. The convergence analysis of both algorithms is provided under suitable assumptions on algorithmic parameters and the geodesic convexity of the objective functions. This convergence analysis is based on fixed-point theory for nonexpansive operators. We also study the convergence rates of these two methods.

	Additionally, we introduce parallel Douglas-Rachford type algorithms for minimizing functionals containing multiple summands with applications to the generalized Heron problem on Hadamard manifolds. To demonstrate the effectiveness of the proposed algorithms, we present some numerical experiments for the generalized Heron problems.}

\keywords{Inertial methods, Douglas-Rachford methods, Fixed point methods, Minimization problems, Generalized Heron problems}



\maketitle

\section{Introduction}

The Douglas-Rachford algorithm is among the most celebrated splitting algorithms in convex analysis and operator theory \cite{Combettes2008Inv,Bot2015,Dixit2022,Bergmann2016,ShikherOpt2024,Hieu2018COA, Barshad2023}, originally introduced by Douglas and Rachford \cite{Douglas1956} for the numerical solution of partial differential equations. It was later generalized by Lions and Mercier \cite{Lions1967} and independently by Passty \cite{Passty1979} to solve the more general problem of finding a zero of the sum of two maximally monotone operators in Hilbert spaces. Eckstein and Bertsekas \cite{Eckstein1992} provided further insights into its convergence under relaxation schemes and inexact computations.

Beyond operator inclusions, the algorithm also finds wide application in convex optimization, particularly for problems involving the sum of two proper, lower semicontinuous, and convex functions; see \cite{Combettes2008Inv}.
Given a Hilbert space $X$ and two proper, convex, lower semicontinuous functions $\phi, \psi: X \to (-\infty, +\infty]$, the Douglas--Rachford algorithm aims to solve the following convex minimization problem:
\begin{equation}\label{HilbertMin}
	\min_{x \in X} \phi(x) + \psi(x).
\end{equation}
The strength of the Douglas-Rachford algorithm lies in its splitting structure, which allows one to compute the proximal mappings of $\phi$ and $\psi$ separately. This structure makes the algorithm especially useful in large-scale and structured optimization problems \cite{Combettes2008Inv}.

Recently, these algorithms were successfully applied in image processing mainly for two reasons: the functional to minimize allows simple proximal mappings within the method and it turned out that the algorithms are highly parallelizable; see \cite{Combettes2008Inv}. Therefore, Douglas-Rachford methods became some of the most popular variational methods for image processing. Continued interest in Douglas-Rachford iteration methods is also due to its excellent performance on various nonconvex problems; see \cite{Bauschke2002, Borwein2011}.

Motivated by applications where data naturally reside on nonlinear spaces such as manifolds, common in computer vision, machine learning, and information geometry, there has been growing interest in extending convex optimization methods to manifold settings. Hadamard manifolds, in particular, offer a natural framework for such generalizations due to their geometric properties, including unique geodesics between points \cite{BacakNAA2014}.
In the past two decades, several results from linear spaces have been extended to manifold settings because of their applications in many areas of science, engineering, and management (see \cite{Martin2, Martin3, Ferreira2002, Fereira2005}).

In this context, a natural generalization of classical convex optimization problems to Hadamard manifolds involves minimizing the sum of two geodesically convex functions.
Let $\M$ be a Hadamard manifold and let $\phi,\psi\colon \M\to (-\infty,+\infty]$ be proper, geodesically convex and  lower semicontinuous functions. Consider the following minimization  problem:
\begin{equation}\label{Pbm1}
	\min_{x \in \M} \phi(x) + \psi(x).
\end{equation}
To compute the solution to problem \eqref{Pbm1}, Bergmann et al. \cite{Bergmann2016} proposed the following Douglas-Rachford algorithm on Hadamard manifold: 
\begin{equation}\label{Bergmannalg}
	x_{n+1}=\gamma(x_n,R_{\lambda \phi}R_{\lambda \psi}(x_n); \alpha_n),
\end{equation}
\tr{where $\lambda>0$, $\{\alpha_n\}$ is a sequence in $[0,1]$ with $\sum_{n=1}^\infty \alpha_n (1-\alpha_n)=\infty$ and  \(\gamma(x,y;\cdot)\) is the unique geodesic joining $x$ to $y$ in $\M$. The operators \(R_{\lambda\phi}\) and \(R_{\lambda\psi}\) denote the reflection of $\phi$ and $\psi$, respectively. Geometrically, \(R_{\lambda\phi}(x)\) is obtained by reflecting the point \(x\) with respect to its proximal point \(\prox_{\lambda\phi}(x)\) along the unique geodesic joining them.}
Bergmann et al. \cite{Bergmann2016} showed that if $R_{\lambda \phi}R_{\lambda \psi}$ is nonexpansive, then the sequence generated by method  \eqref{Bergmannalg} converges to an element $v$ of $\fix(R_{\lambda \phi}R_{\lambda \psi})$ such that $u=\prox_{\lambda \psi}(v)$ is solution of problem \eqref{Pbm1}.

It is worth noting that by setting $T={R}_{\lambda \phi} {R}_{\lambda \psi}$, the  Douglas-Rachford algorithm \eqref{Bergmannalg} can be seen as a Riemannian Krasnosel'ski\u{\i}--Mann iteration of the form
\[
x_{n+1} = \gamma(x_n, T x_n; \alpha_n),
\]
\tr{which is the natural extension of the classical Krasnosel'ski\u{\i}--Mann scheme from Hilbert spaces to Hadamard manifolds; see \cite{Bergmann2016}.}  This observation connects optimization and fixed point theory and hence problem \eqref{Pbm1} can now be viewed as a fixed point problem.

This naturally leads to the framework of fixed point theory.
Various problems across diverse domains, such as image reconstruction \cite{Byrne2004}, signal processing \cite{Byrne2004}, variational inequality problems \cite{Mercier1980}, and convex feasibility problems \cite{Bauschke1996}  can be formulated as fixed point problems. Let $C$ be a nonempty, closed and convex subset of a real Hilbert space $X$ and $T: C \to C$ be a nonlinear operator. Then, the fixed point problem can be stated as:
\begin{equation}\label{FPP}
	\text{ find } x\in C \text{ such that } Tx=x.
\end{equation}

To solve such problems, numerous iterative techniques have been developed. One of the most widely used is the Mann iterative method, introduced by Mann \cite{Mann1953} in 1953. 
The Mann iteration method has its limitations, as it may not effectively approximate the fixed points of pseudocontractive mappings (see \cite{Chidume2001}).
To tackle this challenge, Ishikawa \cite{Ishikawa1974} proposed an iterative approach, extensively investigated by multiple authors across different spaces to compute the fixed point of pseudocontractive and nonexpansive mappings (see  \cite{Takahashi2008,Acedo2007,Dotson1970}).

Building on this, Agarwal et al. \cite{Agarwal2007} introduced the S-iteration method  as follows:
\begin{equation}\label{Sitermethod}
	\begin{cases}
		y_n=(1-\beta_n) x_n+\beta_n Tx_n,\\
		x_{n+1}=(1-\alpha_n)Tx_n+\alpha_n Ty_n  \text{ for all }  n \in \N,
	\end{cases}
\end{equation}
where $\{\alpha_n\}$ and $\{\beta_n\}$ are sequences in $(0,1)$ with $\sum_{n=1}^{\infty}\alpha_n\beta_n(1-\beta_n)=\infty$. 
In 2011, Sahu \cite{Sahu2011} proposed a modified version of the S-iteration method \eqref{Sitermethod}, called the normal S-iteration method, also known as the hybrid Picard-Mann iteration method, which is defined as follows:
\begin{equation}\label{NSmethod}
	x_{n+1}=T((1-\alpha_n)x_n+\alpha_n Tx_n)  \text{ for all }  n \in \N,
\end{equation}
where $\{\alpha_n\}$ is a sequence in $(0,1)$.
Continuing in this direction, Sahu \cite{Sahusoft2020} observed that applying the operator multiple times in each iteration can significantly enhance the convergence.
This led to the $p$-accelerated normal S-iteration method \cite{Sahusoft2020}, defined as follows:
\begin{equation}\label{paccmethod}
	x_{n+1}=T^p((1-\alpha_n)x_n+\alpha_n Tx_n)  \text{ for all }  n \in \N,
\end{equation}
where $\{\alpha_n\}$ is a sequence  in $(0,1)$ and $p$ is the power of acceleration. For $p=1$, the  $p$-accelerated normal S-iteration method reduces to the normal S-iteration method defined by \eqref{NSmethod}. It is observed in \cite{Sahusoft2020} that the $p$-accelerated normal S-iteration method has better convergence performance than the Mann iterative method and normal S-iterative method for the fixed point of nonexpansive operators with convex domain in infinite dimensional Hilbert space. In fact in most cases $p$-accelerated normal S-iteration method works better than the inertial Mann iteration method due to the parameter $p$.

Inertial-type extrapolation techniques play a critical role in accelerating iterative methods.	
The foundation of such methods can be traced back to Polyak \cite{Polyak1964}, who introduced the heavy ball method.  Alvarez and Attouch \cite{Alvarez2001} extended the concept to improve the performance of the proximal point algorithm, proposing the inertial proximal point algorithm for solving inclusion problems.

In the existing literature, the utilization of inertial extrapolation-based methods has been thoroughly explored and implemented in various contexts 
(see \cite{Ochs2014,Bot2016,Chen2014,Samir2023, Attouch2019} and references therein). Recently, researchers have developed several iterative algorithms by incorporating inertial extrapolation techniques, such as  the fast shrinkage thresholding algorithm \cite{Beck2009(1)}, inertial forward-backward algorithm \cite{Lorenz2015} and inertial Douglas-Rachford splitting algorithm  \cite{Bot2015}.

Building upon the development of inertial and accelerated iterative methods, Mainge \cite{Mainge2008} proposed a novel approach in 2008 by combining inertial extrapolation with the Mann method, aiming to compute the  fixed point of nonexpansive mappings on a real Hilbert space. The inertial Mann algorithm is defined as follows:
\begin{equation}\label{Mainge}
	\begin{cases}
		y_n=x_n+\theta_n(x_n-x_{n-1}),\\
		x_{n+1}=(1-\alpha_n)y_n+\alpha_n Ty_n  \text{ for all }  n \in \N,
	\end{cases}
\end{equation}
where $\{\theta_n\}$ and $\{\alpha_n\}$ are sequences in $[0,1)$.
Mainge \cite{Mainge2008} studied the weak convergence of inertial Mann algorithm \eqref{Mainge} under the conditions that 
$\theta_n\in [0,\theta]$ for all $n\in \N$, where $\theta\in [0,1)$, 
$\sum_{n=1}^\infty \theta_n\|x_n-x_{n-1}\|^2<\infty$, and  $\inf_{n \in \N} \alpha_n >0$ and $\sup_{n\in \N} \alpha_n<1$.

The second condition that $\sum_{n=1}^\infty \theta_n \|x_n-x_{n-1}\|^2<\infty$, taken by Mainge \cite{Mainge2008} is often difficult to verify in practice. Bo\c{t} et al. \cite{Bot2015}  proposed a modification in algorithm \eqref{Mainge} 
by replacing this condition with an alternative criterion
that the sequence $\{\theta_n\}$ is nondecreasing in $[0,\theta]$ with $\theta_1=0$ and $\theta\in [0,1)$  such that 
\begin{align}\label{Botassump}
	&\delta >\frac{\alpha^2(1+\alpha)+\alpha \sigma}{1-\alpha^2} \text{ and }
	0<\lambda\leq \lambda_n \leq \frac{\delta-\alpha[\alpha(1+\alpha)+\alpha\delta+\sigma]}{\delta[1+\alpha(1+\alpha)+\alpha\delta+\sigma]}  \text{ for all }  n \in \N,
\end{align}
where $\alpha,\delta,\sigma>0$ are some suitable parameters.

Recently, Dixit et al. \cite{Dixit2020} proposed the inertial normal S-iteration method for the  solution of problem \eqref{FPP} for a nonexpansive operator, which in practice converges faster than the Mann iteration method, inertial Mann iteration method and normal S-iteration method. 

The verification of condition \eqref{Botassump}, as considered by \cite{Bot2015}, is a challenging task. To address this issue, Sahu \cite{Sahusoft2020} proposed a new convergence criterion for the inertial Mann iteration. This criterion assumes that the sequences \( \{\alpha_n\} \) and \( \{\theta_n\} \) satisfy the following conditions:  there exist constants $a,b>0$ such that \( 0 < a \leq \alpha_n \leq b < 1 \) for all \( n \in \N \), the sequence \( \{\theta_n\}\)  is nondecreasing in $[0,1)$, and there exists \( \theta \in (0,1) \) such that $0\leq \theta_n \leq \theta < \frac{\epsilon}{1+\epsilon+\max\{1,\epsilon\}}$ for all $n \in \N$, where $\epsilon=\frac{1-b}{b}$. These conditions are easier to verify in practice.

Motivated by the work of Sahu \cite{Sahusoft2020} and  Bergmann et al. \cite{Bergmann2016}, the main purpose of this paper is to introduce an inertial Douglas-Rachford method and $p$-accelerated normal S-Douglas-Rachford method for solving problem \eqref{Pbm1} and study its convergence analysis. Our proposed methods are novel and of interest in the following sense:
\begin{itemize}
	\item[(i)] 
\tr{Compared with the classical Douglas--Rachford method on Hadamard manifolds Bergmann et al.~\cite{Bergmann2016}, we propose two accelerated variants on Hadamard manifolds, namely an inertial Douglas-Rachford method and $p$-accelerated normal S-Douglas-Rachford method. To the best of our knowledge, such inertial and $p$-accelerated Douglas--Rachford schemes have not been previously analyzed in the Riemannian setting.}
	
	\item[(ii)]
	\tr{Unlike Sahu~\cite{Sahusoft2020} and other works on $p$-accelerated normal S-iteration in Hilbert spaces, we extend the $p$-acceleration idea to Douglas--Rachford splitting methods on Hadamard manifolds. Although some of the conditions are similar to the Hilbert space case, this extension is not straightforward because Hadamard manifolds do not have a linear structure. It requires careful use of geodesic convexity, exponential maps, and comparison geometry.}
	
	\item[(iii)]
\tr{Compared with the inertial Douglas--Rachford and inertial Mann schemes studied in Hilbert spaces by Bo\c{t} et al.~\cite{Bot2015} and Mainge~\cite{Mainge2008}, our analysis is carried out entirely in the nonlinear setting of Hadamard manifolds. Moreover, we use different parameter conditions, which are easier to verify in practice than those in Bo\c{t} et al.~\cite{Bot2015} and Mainge~\cite{Mainge2008}.}

\item[(iv)] \tr{
We analyze the convergence rates of our proposed algorithms. Specifically, the $p$-accelerated normal S-iteration method converges at a rate of order $o(1/\sqrt{n})$, while the inertial Mann iteration method converges in the minimal residual sense, also at a rate of order $o(1/\sqrt{n})$.}
\end{itemize}

We give \tr{an} application of these algorithms to minimizing functionals containing multiple summands by introducing a parallel inertial  Douglas-Rachford method and parallel $p$-accelerated Douglas-Rachford method.  We also apply our parallel Douglas-Rachford methods to solve the generalized Heron problems in Hadamard manifolds and provide numerical experiments.

	\section{Preliminaries}

Let $C$ be a nonempty subset of a metric space $(X,d)$ and let $T\colon  C \to C$ be an operator. Then $T$ is said to be nonexpansive if 
$d(Tx,Ty)\leq d(x,y)  \text{ for all }  x,y \in C.$
We denote the set of fixed points of an operator $T\colon C \to C$ by $\fix(T)$, thus, 
$\fix(T)=\{x\in C\colon Tx=x\}.$

For a sequence $\{x_n\}$ in $C$, define $R_{T,\{x_n\}}:\N \to [0,\infty)$ by 
\begin{equation*}
	R_{T,\{x_n\}}(n)=\min_{1\leq i\leq n}d(x_i,Tx_i) \text{ for all } n \in \N.
\end{equation*}
Observe that $R_{T,\{x_n\}}(n)\leq d(x_i,Tx_i)$ for all $i=1,2,\ldots,n$ and $\{R_{T,\{x_n\}}(n)\}$ is decreasing.
We use the notation \(o\) \tr{to mean that} \(s_n = o(1/t_n)\) if and only if 
\(
\lim_{n\to\infty} s_n t_n = 0.
\)
\begin{Lemma}[\cite{Dong2015comments}]\label{DongLemma}
	Let \(\{a_n\}\) and \(\{b_n\}\) be sequences of positive real numbers such that  
	\(\sum_{n=1}^{\infty} a_n b_n < \infty.\)
	Suppose that \(\sum_{n=1}^{\infty} a_n  = \infty\) and that the sequence \(\{b_n\}\) is decreasing. Then
	\(
	b_n = o\!\left(\frac{1}{\sum_{i=1}^{n} a_i}\right).
	\)
\end{Lemma}

\begin{Definition}[\cite{SahuNFAO}]
	Let $C$ be a nonempty subset of a metric space $(X,d)$ and let  $T\colon  C \to C$ be an operator with $\fix(T)\neq \emptyset$. Let $\{x_n\}$ be a sequence in $C$. We say that $\{x_n\}$ has
	\begin{enumerate}
		\item[(a)]   the limit existence property (for short, the LE property) for $T$ if $\lim_{n \to \infty}d(x_n,v)$ exists for all $v \in \fix(T)$;
		\item[(b)]    the approximate fixed point property (for short, the AF point property) for $T$ if $$\lim_{n \to \infty}d(x_n, Tx_n)=0;$$
		\item[(c)]  the LEAF point property if $\{x_n\}$ satisfies both LE property and AF property.
	\end{enumerate}
\end{Definition}

\begin{Lemma}[\cite{Sahusoft2020}] \label{Lemma1}
	Let $X$ be a Hilbert space and $\alpha\geq 0$. Then,
	$$\|x-\alpha y\|^2\geq (1-\alpha)\|x\|^2-\alpha(1-\alpha)\|y\|^2  \text{ for all }  x,y\in X.$$
\end{Lemma}

\begin{Lemma}[{\cite{Alvarez2004}}]  \label{LeSC3}
	Let $\{a_n\},\{b_n\}$ and $\{\theta_n\}$ be  sequences in $[0,\infty)$ such that 
	\[a_{n+1}\leq a_n+\theta_n(a_n-a_{n-1})+b_n  \text{ for all }  n \in \N,\]
	$\sum_{n=1}^\infty b_n <\infty$ and there exists a real number $\theta$ with $0\leq \theta_n \leq \theta<1$ for all $n\in \N$. Then, the following hold:
	\begin{enumerate}
		\item[(i)] $\sum_{n=1}^\infty [a_n-a_{n-1}]_+<\infty$, where $[t]_+=\max\{t,0\}$.
		\item[(ii)] There exists $a^*\in [0,\infty)$ such that $a_n\to a^*$.
	\end{enumerate}
\end{Lemma}

For the basic definitions and results regarding the geometry of Riemannian manifolds, we refer to \cite{Docarmo,Sakai,Udriste}. 

Let $\M$ be an $m$-dimensional differentiable manifold and let  $x \in \M$. The set of all tangent vectors at the point $x$, denoted by $\T_x\M$, is called the tangent space of $\M$ at $x$.  The tangent bundle of $\M$ is defined as the disjoint union $\T\M = \bigcup_{x \in \M} \T_x\M$.
A differentiable manifold $\M$ with a Riemannian metric $\langle \cdot , \cdot \rangle$ is said to be a Riemannian manifold.
The induced norm on $\T_x\M$ is denoted by $\|\cdot \|_x$.  The Riemannian distance between $x$ and $y$ in $\M$ is designated by $d(x,y)$.
A geodesic joining $x$ to $y$ in the Riemannian manifold $\M$ is said to be a minimal geodesic if its length is equal to $d(x,y)$. We denote geodesic joining $x$ and $y$ by $\gamma(x,y; \cdot)$, i.e.,
$\gamma(x,y; \cdot)\colon[a,b] \to \M$ is such that, for $a,b \in \R$, we have $\gamma(x,y;a)=x$ and $\gamma(x,y;b)=y$.

Let $\nabla $ be the Riemannian connection associated with the Riemannian manifold $\M$. The parallel transport on the tangent bundle $\T\M$ along the geodesic $\gamma(x,y;\cdot)$ with respect to Riemannian connection $\nabla$ is defined as $P_{\gamma(x,y;b),\gamma(x,y;a)} \colon \T_{\gamma(x,y;a)}\M \to \T_{\gamma(x,y;b)}\M$ such that 
$P_{\gamma(x,y;b),\gamma(x,y;a)}(v)=V(\gamma(x,y;b))  \text{ for all }  a,b \in \R,  v \in \T_x\M,$
where $V$ is the unique vector field satisfying $\nabla_{\gamma'(x,y;t)}V =0$ for all $t \in [a,b]$ and $V(\gamma(x,y;a))=v$.
A Riemannian manifold $\M$ is said to be complete if for any $x\in\M$ all geodesic emanating from $x$ are defined for all $t\in\R$. By Hopf-Rinow Theorem \cite{Sakai}, if $\M$ is complete, then any pair of points in $\M$ can be joined by a minimal geodesic. Moreover, $(\M,d)$ is a complete metric space. \tr{A complete simply connected Riemannian manifold of nonpositive sectional curvature is called a Hadamard manifold.  Since in a Hadamard manifold $\M$, any two points can be joined by a unique geodesic, we denote the parallel transport by  $P_{y,x}$ instead of $P_{\gamma(x,y;b),\gamma(x,y;a)}$.}
If $\M$ is a Hadamard manifold, then the exponential map $\exp_x \colon \T_x\M
\to \M$ at $x\in \M$ is defined by $	\exp_xv=\gamma_{x,v}(1)$ for all  $v\in \T_x\M$,
where $\gamma_{x,v} \colon  \R \to \M$ is a unique geodesic
starting from $x$ with velocity $v$, i.e., $\gamma_{x,v}(0)=x$ and $
\gamma'_{x,v}(0)=v$. It is known that $\exp_x(tv)=\gamma_{x,v}(t)$ for
each real number $t$. For the exponential map, the following properties hold:
\begin{enumerate}
	\item[(i)]  $\exp_x\mathbf{0}=\gamma_{x,\mathbf{0}}(1)=\gamma_{x,v}(0)=x$, where $\mathbf{0}$ is the zero tangent vector
	\item[(ii)] Inverse of $\exp_x$,  $\exp_x^{-1}\colon \M \to \T_x\M$ is given by 
	$\exp_x^{-1}y=v,$
	where $\gamma_{x,v}(1)=\exp_xv=y$.
	\item[(iii)] $\exp_x$ is a diffeomorphism on $\T_x\M$ and  for any $x,y \in
	\M$, we have $d(x,y)=\|\exp_x^{-1}y\|$. 
\end{enumerate} 

The following proposition establishes the relationship between geodesics and the exponential map.
\begin{Proposition}[\cite{Sakai}]
	\label{gtoe}  Let $\M$ be a
	Hadamard manifold and $x\in \M$. Then $\exp_x\colon  \T_x\M \to \M$ is a diffeomorphism and for any two points $x,y \in \M$, there exists a unique unit-speed geodesic joining $x$ to $y$, which is in fact a minimal geodesic. Let $\gamma(x,y; \cdot) \colon  [0,1] \to \M$ be the desired geodesic joining $x$ to $y$, i.e., $x=\gamma(x,y;0)$ and $y=\gamma(x,y;1)$. Then
	\begin{equation*}
		\gamma(x,y;t) = \exp_x(t\exp_x^{-1} y)  \text{ for all }  t \in [0,1].
	\end{equation*}
\end{Proposition}

\begin{Remark} [\cite{Martin2}] \label{re101} Let $\M$ be a Hadamard manifold.
	If $x,y \in \M$ and $v \in \T_y\M$, then 
	\begin{equation*}
		\langle v, -\exp_y^{-1}x \rangle = \langle v, P_{x,y}\exp_x^{-1}y\rangle =\langle P_{x,y} v, \exp_x^{-1}y \rangle.
	\end{equation*}
\end{Remark}

Hadamard manifolds have the nice feature that they resemble convexity properties from Euclidean spaces. A subset $C$ of a Hadamard manifold $\M$ is said to be geodesically convex if, for any two points $x$ and $y$ in $C$, the geodesic  joining $x$ to $y$ is contained in $C$. A function $f\colon \M \to \R$ is said to be  geodesically convex if, for any geodesic $\gamma(x,y; \cdot) \colon [a,b] \to \M$, the composition function $ f\circ \gamma(x,y;\cdot) \colon [a,b] \to \R$ is convex, i.e.,
		\begin{equation*}
				f\circ \gamma(x,y;(1-t)a+tb) \leq (1-t)(f\circ \gamma(x,y;a))+ t(f \circ
				\gamma(x,y;b)),
			\end{equation*}
		for any  $a,b \in \R$, $x,y \in \M$ and $t\in [0,1]$.

The following proposition highlights the convexity of the distance function.
\begin{Proposition}[\cite{Sakai}]  \label{disconv2}
	Let $\M$ be a Hadamard manifold and  $d\colon \M\times \M \to \R $ be the metric function. Then $d$ is a geodesically convex function with respect to product Riemannian metric, i.e., given any pair of geodesics $ \gamma(x_1,y_1;\cdot)\colon [0,1] \to \M$ and $\gamma(x_2,y_2;\cdot)\colon [0,1] \to \M$, for some $x_1,y_1,x_2,y_2\in \M$, the following hold:
	\begin{enumerate}
		\item[(i)] $d(\gamma(x_1,y_1;t), \gamma(x_2,y_2;t)) 
		\leq (1-t)d(x_1,x_2)+td(y_1,y_2)  \text{ for all }  t \in [0,\,1].$
		\item[(ii)] $d(\gamma(x_1,y_1;t), z) \leq (1-t)d(x_1,z)+ td(y_1,z)  \text{ for all }   t\in [0,1] \text{ and } z\in \M.$
	\end{enumerate}
\end{Proposition}

\begin{Proposition} [\cite{ShikherCNSNS2024}] \label{Newpr1}
	Let $\M$ be a Hadamard manifold and $x,y,z \in \M$. Let $u\in \M$ be a point on the geodesic joining $x,y$, i.e., $u=\gamma(x,y;t)$, for some $t\in [0,1]$, then 
	\begin{equation*}
		d^2(u,z) \leq (1-t)d^2(x,z)+td^2(y,z)-t(1-t)d^2(x,y).
	\end{equation*}
\end{Proposition}

\begin{Proposition}\label{Newpr2}
	Let $\M$ be a Hadamard manifold and $x,y \in \M$. Then, for $z\in \M$ and $\theta \in [0,1]$, the following inequality holds:
	\begin{equation*}
		d^2(\exp_x(-\theta \exp_x^{-1}y),z) \leq (1+\theta)d^2(x,z)-\theta d^2(y,z)+\theta(1+\theta) d^2(x,y).
	\end{equation*}
\end{Proposition}
\noindent
(The proof is given in Appendix~\ref{Ap1}.)

Let $\M$ be a Hadamard manifold and let $N\in \N$. Then, the product space of manifold $\M$ is defined by 
$\M^N=\{x=(x_1,\ldots,x_N)\colon x_i\in \M \text{ for } i=1,2,\ldots,N\}.$
\tr{Note that product space $\M^N$ is again a Hadamard manifold. The tangent space of $\M^N$ at $x=(x_1,\ldots,x_N)$ is $\T_x\M^N=\T_{x_1}\M\times \cdots \times \T_{x_N}\M$.
The distance between two points $x=(x_1,\ldots,x_N)$ and $y=(y_1,\ldots,y_N)$ in $\M^N$ is defined by
\[
\rho(x,y)=\left(\sum_{i=1}^N d^2(x_i,y_i)\right)^{1/2}.
\]
Moreover, the  geodesic in $\M^N$ joining $x=(x_1,\ldots,x_N)$ to $y=(y_1,\ldots,y_N)$ is given componentwise by
\(\gamma(x,y;t)=\bigl(\gamma(x_i,y_i;t),\ldots,\gamma(x_N,y_N;t)\bigr) \text{ for all } t\in[0,1]\).}

We now recall an important property related to the convergence of sequences satisfying the LEAF point (limit existence and approximate fixed point)  property.
\begin{Proposition}[\cite{ShikherCNSNS2024}] \label{LEAFpr1}
	Let $C$ be a nonempty and closed subset of a Hadamard manifold  $\M$ and let $T\colon C \to C$ be a nonexpansive operator such that $\fix(T)\neq \emptyset$. Let $\{x_n\}$ be a sequence in $C$ satisfying the LEAF  point property for $T$.
	Then, $\{x_n\}$ converges to a fixed point of $T$.
\end{Proposition}

	Let  $f\colon \M \to (-\infty,\infty]$ be a proper, geodesically convex and lower semicontinuous function. For $\lambda>0$, the proximal operator of $f$ \tr{is well defined  and given by}
	$$\operatorname{Prox}_{\!\lambda f}(x)=\amin_{y\in \M}\left(f(y)+\frac{1}{2\lambda} d^2(x,y)\right)  \text{ for all }  x\in \M.$$
	The proximal operator of a  proper, geodesically convex and lower semicontinuous function is firmly nonexpansive operator (see \cite{Jost1995, Martin3}), i.e., for any $x,y \in \M$ and $t \in [0,1]$
	\begin{equation}\label{FNEeq}
		d(\operatorname{Prox}_{\!\lambda f}(x),\operatorname{Prox}_{\!\lambda f}(y)) \leq d(\gamma(x,\operatorname{Prox}_{\!\lambda f}(x);t),\gamma(y,\operatorname{Prox}_{\!\lambda f}(y);t)).
	\end{equation}

	%The proximal operator $\prox_{\lambda f}$ is well defined for all $\lambda >0$.
	\tr{The reflection of a proper, geodesically convex and lower semicontinuous function $f\colon \M \to (-\infty,\infty]$ is given by $R_{\lambda f}:\M \to \M$ with 
	$$R_{\lambda f}(x)=\exp_{\prox_{\lambda f}(x)}(-\exp_{\prox_{\lambda f}(x)}^{-1}(x))  \text{ for all }  x\in \M.$$
	That is, $R_{\lambda f}(x)$ is the reflection of $x$ with respect to $\prox_{\lambda f}(x)$ along the unique geodesic joining $x$ to $\prox_{\lambda f}(x)$.}

\begin{Lemma} [\cite{Bergmann2016}] \label{Lem1}
	Let $\M$ be an Hadamard manifold, $\lambda > 0$ and $c \in \M$. Define $f\colon \M \to \R$ by 
	$f(x)=d(c,x)  \text{ for all }  x\in \M$.
	Then,  
	$\prox_{\lambda f}(x) = \gamma(x,c,t)$, where 	$t =\min\big\{ \tfrac{\lambda}{d(x,c)} ,1 \big\}$.
\end{Lemma}

We denote the domain of a function $f\colon \M\to (-\infty,\infty]$ by $\dom f$ and interior of a set $C$ by $\intr(C)$.

\begin{Theorem} [\cite{Bergmann2016}]  \label{thmequiv}
	Let $\M$ be a Hadamard manifold and
	$\phi,\psi\colon\M \rightarrow (-\infty,\infty]$ be proper, geodesically convex and lower semicontinuous functions
	such that 
	$\intr(\dom \phi) \cap \dom \psi\neq \emptyset$ and solution set of problem \eqref{Pbm1} is nonempty. Let $\eta >0$.
	Then, for each solution $u$ of problem \eqref{Pbm1} there exists
	a fixed point $v$ of $R_{\eta \phi}R_{\eta \psi}$
	such that 
	$u = \prox_{\eta \psi}v$.
	Conversely, if a fixed point $v$
	of ${R}_{\eta \phi} {R}_{\eta \psi}$ exists,
	then $u=\prox_{\eta \psi}v$ is a solution of problem \eqref{Pbm1}.
\end{Theorem}

Let $C$ be a nonempty subset of a Hadamard manifold $\M$. The indicator function $\iota_C\colon \M \to (-\infty, \infty]$ is defined by 
$$\iota_C=\begin{cases}
	0, & \text{ if } x\in C,\\
	\infty, & \text{otherwise}.
\end{cases}$$

Next, we present some results for functions with nonexpansive reflections.
\begin{Proposition} [\cite{Leon2013,Bergmann2016}] \label{kappapr2}
	Let  $C$ be a nonempty, closed and geodesically convex subset of Hadamard manifold $\M$ with constant  sectional curvature. Then
	$$d(R_{\iota_C}(x),R_{\iota _C}(y)) \leq d(x,y)  \text{ for all }  x,y \in C.$$
\end{Proposition}

\begin{Theorem}  [\cite{Bergmann2016}] \label{distNE}
	Let $\M$ be a Hadamard Manifold, $\lambda > 0$ and $c \in \M$. Let
	$g\colon \M \to \R$ be defined by
	$g(x)=d(c,x)  \text{ for all }  x\in \M$. Then, the reflection $R_{\lambda g}$ is nonexpansive.
\end{Theorem}

\begin{Proposition}\label{Newpr3}
	\tr{Let $\M$ be a Hadamard manifold with zero sectional curvature. Then the reflection of a proper, geodesically convex and lower semicontinuous function $f\colon \M \to (-\infty,\infty]$  is nonexpansive.}
\end{Proposition}
\noindent
(The proof is given in Appendix~\ref{Ap3}.)

	\section{Efficient Douglas-Rachford Algorithms on Hadamard Manifolds}

In this section, we introduce two \hide{accelerated} algorithms for finding fixed points of nonexpansive operators on Hadamard manifolds. We also discuss the consequences of these algorithms for solving the minimization problem \eqref{Pbm1}.

First, we introduce a method that incorporates an inertial extrapolation acceleration term, given as follows:
\begin{Algorithm}\label{ALG1'}
	Let  $\M$ be a Hadamard manifold and $T \colon  \M \to \M$ be an operator.\\
	\textbf{Initialization:} Choose $x_0,x_1\in \M$ arbitrary.\\
	\textbf{Iterative step:} For the current $n^{th}$ iterate, compute the $(n+1)^{th}$ iteration as follows:
	\begin{equation}\label{ALGeq1'}
		\begin{cases}
			y_n= \exp_{x_n}(-\theta_n \exp_{x_n}^{-1} x_{n-1}),\\
			x_{n+1} = \gamma(y_n,Ty_n;\alpha_n)  \text{ for all }  n \in \N,
		\end{cases}
	\end{equation}
	where $\{\alpha_n\}$ is a sequence in $(0,1)$and $\{\theta_n\}$ is a sequence in $[0,1)$.
\end{Algorithm}

Note that method \eqref{ALGeq1'} is an extension of the inertial Mann iteration method \eqref{Mainge} from Hilbert spaces to the setting of Hadamard manifolds. Therefore, we also refer to it here as the inertial Mann iteration method.

For the convergence analysis of inertial Mann method \eqref{ALGeq1'}, we assume that the sequences $\{\alpha_n\}$ and $\{\theta_n\}$ satisfy the following conditions:
\begin{enumerate}[label=(C\arabic*), ref=C\arabic*]
	\item \label{C_1} $0 < a \leq \alpha_n \leq b < 1$ for all $n \in \mathbb{N}$ for some $a,b>0$,
	\item \label{C_2} $\{\theta_n\}$ is a nondecreasing sequence in $[0,1)$,
	\item \label{C_3} There exists $\theta \in (0,1)$ such that $0 \leq \theta_n \leq \theta < \frac{\epsilon}{1+\epsilon+ \max\{1, \epsilon\}}$ for all $n \in \mathbb{N}$, where $\epsilon = \frac{1-b}{b}$.
\end{enumerate}

\begin{Remark}\label{ParameterRem}
		\tr{The upper bound for the inertial parameter $\{\theta_n\}$ in the assumption of Bo\c{t} et al.~\cite{Bot2015} (see \eqref{Botassump}) depends on several auxiliary parameters $\alpha$, $\delta$, and $\sigma$. Consequently, determining an admissible range for $\{\theta_n\}$ requires careful tuning of these coupled parameters and is not straightforward. In contrast, under our condition \eqref{C_3}, an explicit upper bound for $\{\theta_n\}$ can be directly obtained once an upper bound for the sequence $\{\alpha_n\}$ is fixed. Intuitively, the parameter $\theta_n$ controls the amount of inertia introduced at each iteration, while the bound involving $\epsilon = (1-b)/b$ ensures that this inertial effect remains sufficiently small relative to the relaxation parameter $\alpha_n$. This  allows the parameters to be chosen independently, making the conditions easier to verify and more convenient to implement in practice than those in~\cite{Bot2015}.}
		\end{Remark}

\begin{Theorem} \label{InerThm}
	Let $\M$ be a Hadamard manifold and let $T\colon  \M \to \M$ be a nonexpansive operator such that $\fix(T)\neq \emptyset$.  Let $\{x_n\}$ be a sequence  generated by   inertial Mann method \eqref{ALGeq1'}, where $\{\alpha_n\}$ and $\{\theta_n\}$ are sequences satisfying the conditions \eqref{C_1}, \eqref{C_2} and \eqref{C_3}.  \tr{For $v\in \fix(T)$, define $\phi_n= d^2(x_n,v), \psi_n= \phi_n-\theta_n \phi_{n-1}+ K_n d^2(x_n,x_{n-1})$, where $K_n=  \theta_n (1+\theta_n+ \epsilon (1-\theta_n))$. Assume that $\psi_1>0$. Then, we have the following:
	\begin{itemize}
		\item[(i)] $\{\psi_n\}$ is nonincreasing.
		\item[(ii)]  $\sum_{n=1}^\infty d^2(x_{n+1},x_n) \leq \frac{\psi_1}{K(1-\theta)},$ where
		$K= \epsilon- \theta(1+\epsilon + \max\{1,\epsilon\}) >0$.
		\item[(iii)] $\lim_{n\to \infty} d(x_n,v)$ and $\lim_{n\to \infty} d(x_n,v)$  exists and $\lim_{n\to \infty} d(y_n,Ty_n)=0$. 
		\item[(iv)] $\{x_n\}$ converges to an element of $\fix(T)$.
	\end{itemize}}
\end{Theorem}
\begin{proof}
(i)	Let $v\in \fix(T)$. From $\eqref{ALGeq1'}$ and Proposition $\ref{Newpr2}$, we have
	\begin{equation} \label{IMpfeq2}
		d^2(y_n,v) \leq (1+\theta_n)d^2(x_n,v) - \theta_nd^2(x_{n-1},v) + \theta_n(1+\theta_n) d^2(x_n,x_{n-1}).
	\end{equation}
	From $\eqref{ALGeq1'}$ and Proposition $\ref{Newpr1}$, we have 
	\begin{align}
		d^2(x_{n+1},v) &= (1-\alpha_n) d^2(y_n,v)+ \alpha_n d^2(Ty_n,v)-\alpha_n(1-\alpha_n)d^2(y_n,Ty_n)\nonumber\\
		&\leq (1-\alpha_n) d^2(y_n,v)+ \alpha_n d^2(y_n,v)-\alpha_n(1-\alpha_n)d^2(y_n,Ty_n)\nonumber\\
		& = d^2(y_n,v)-\alpha_n (1-\alpha_n) d^2(y_n,Ty_n).\label{IMeq5}
	\end{align}
	Consider the geodesic triangle $\triangle(x_n,x_{n-1},v)$ and its comparison triangle $\triangle(x_n',x_{n-1}',v')$. Then 
	\begin{equation}\label{IMeq3}
		d(x_n,x_{n-1})=\|x_n'-x_{n-1}'\|, d( x_{n-1},v)=\|x_{n-1}'-v'\| \text{ and } d(x_n,v)=\|x_n'-v'\|.
	\end{equation}
	Since $y_n= \exp_{x_n}(-\theta_n \exp_{x_n}^{-1}x_{n-1})$, its comparison point will be 
	\begin{equation}\label{IMpfeq1}
		y_n'= (1+\theta_n)x_n'+(-\theta_n) x_{n-1}'.
	\end{equation}
	Now consider the geodesic triangle $\triangle(y_n,{Ty_n},v)$.  Without loss of generality, we can choose its comparison triangle to be $\triangle(y_n', Ty_n',\overline{v})$, where $y_n'$ is the same as defined in \eqref{IMpfeq1}. Since $x_{n+1}= \exp_{y_n} (\alpha_n \exp_{y_n}^{-1}Ty_n)$, its comparison point will be
	\begin{equation}\label{comparisonpoint}
		x_{n+1}'= (1-\alpha_n)y_n'+ \alpha_n Ty_n'.
	\end{equation}
	Without loss of generality, we can also choose $Ty_n'$  such that 
	\begin{equation}\label{Mainineq}
		d(x_n,x_{n+1})=\|x_n'-x_{n+1}'\|
	\end{equation}
	as shown in Figure \ref{figure1}. Then 
	\begin{equation}\label{IMeq4}
		d(y_n,Ty_n)=\|y_n'-Ty_n'\|, d(Ty_n,v)=\|Ty_n'-\overline{v}\| \text{ and } d(y_n,v)=\|y_n'-\overline{v}\|.
	\end{equation}
	From \eqref{comparisonpoint}, we have $x_{n+1}'  = \alpha_n( Ty_n'- y_n')+y_n'$,
	which implies that 
	\begin{equation}\label{IMeq4'}
		Ty_n'-y_n'= \frac{1}{\alpha_n} (x_{n+1}'-y_n').
	\end{equation}
	Since $x_{n+1}'$ is comparison point of $x_{n+1}$, we have $d(x_{n+1},y_n)=\|x_{n+1}'-y_n'\|$. It follows 
	from \eqref{IMeq4} and  \eqref{IMeq4'} that 
	\begin{equation*}
		d^2(Ty_n,y_n)=\|Ty_n'-y_n'\|^2= \frac{1}{\alpha_n^2} \|x_{n+1}'-y_n'\|^2 = \frac{1}{\alpha_n^2} d^2(x_{n+1},y_n).
	\end{equation*}
	Thus, from \eqref{IMeq5}, we have 
	\begin{align*}
		d^2(x_{n+1},v)&\leq d^2(y_n,v)- \frac{1-\alpha_n}{\alpha_n}d^2(x_{n+1},y_n) \leq d^2(y_n,v)- \frac{1-b}{b}d^2(x_{n+1},y_n).
	\end{align*}
	It follows from \eqref{IMpfeq2} that
	\begin{equation}\label{IMpfeq8}
		d^2(x_{n+1},v) \leq (1+\theta_n)d^2(x_n,v)-\theta_n d^2(x_{n-1},v)+ \theta_n(1+\theta_n) d^2(x_n,x_{n-1})- \epsilon d^2(x_{n+1},y_n).
	\end{equation}
	By $\eqref{IMpfeq1}$ and Lemma \ref{Lemma1}, we have  
	\begin{align}
		d^2(x_{n+1},y_n)&= \|x_{n+1}'-y_n'\|^2= \|x_{n+1}'-((1+\theta_n)x_n'+ (-\theta_n)x_{n-1}')\|^2\nonumber\\%\label{IMpfeq4}
		& = \|(x_{n+1}'-x_n')-\theta_n(x_n' -x_{n-1}')\|^2\nonumber\\
		& \geq (1-\theta_n) \|x_{n+1}'-x_n'\|^2 -\theta_n (1-\theta_n) \|x_n'-x_{n-1}'\|^2.\nonumber
	\end{align}
	using \eqref{IMeq3} and \eqref{Mainineq}, we have $d(x_n,x_{n-1})=\|x_n'-x_{n-1}'\|$ and $d(x_n,x_{n+1})=\|x_n'-x_{n+1}'\|$. Thus
	\begin{equation}\label{IMpfeq5}
		d^2(x_{n+1},y_n) \geq (1-\theta_n) d^2(x_{n+1},x_n) - \theta_n (1-\theta_n) d^2(x_n,x_{n-1}).
	\end{equation}
	\begin{figure}
		\centering
		\begin{tikzpicture}
			\draw[-,ultra thick] (-8,0) -- (2,0) node[below] {};
			\draw[-,ultra thick] (-8,0) -- (-4,5) node[below] {};
			\draw[-,ultra thick] (-4,5) -- (0,0) node[below] {};
			\draw[-,ultra thick] (2,0) -- (1,5) node[below] {};
			\draw[-,ultra thick] (1,5) -- (4,2) node[below] {};
			\draw[-,ultra thick] (4,2) -- (2,0) node[below] {};
			\draw[](-8,0)node[below]{$x_{n-1}'$};
			\draw[](-4,5)node[above]{$v'$};
			\draw[](0,0)node[below]{$x_n'$};
			\draw[](2,0)node[below]{$y_n'$};
			\draw[](1,5)node[above]{${Ty_n'}$};
			\draw[](4,2)node[right]{$\bar{v}$};
			\draw[](1.5,2.5)node[right]{$ x_{n+1}'$};
			\draw[dashed](0,0)--(1.5,2.5)node[left=0.5, rotate=60]{$d(x_n,x_{n+1})$};
		\end{tikzpicture}
		\caption{Selection of comparison points}
		\label{figure1}
	\end{figure}
	From \eqref{IMpfeq8} and \eqref{IMpfeq5}, we get
	\begin{align}
		d^2(x_{n+1},v) \leq& (1+\theta_n) d^2(x_n,v)-\theta_n d^2(x_{n-1},v)+ \theta_n(1+\theta_n) d^2(x_n,x_{n-1})\nonumber\\
		&-\epsilon [(1-\theta_n)d^2(x_n,x_{n+1})-\theta_n(1-\theta_n)d^2(x_n,x_{n-1})]\nonumber\\
		= &(1+\theta_n) d^2(x_n,v)-\theta_n d^2(x_{n-1},v) -\epsilon (1-\theta_n) d^2(x_{n+1},x_{n})\nonumber\\
		&+ \theta_n (1+\theta_n+ \epsilon (1-\theta_n)) d^2(x_n,x_{n-1}).\label{IMpfeq9}
	\end{align}
	Take $\phi_n= d^2(x_n,v), \psi_n= \phi_n-\theta_n \phi_{n-1}+ K_n d^2(x_n,x_{n-1})$, where $K_n=  \theta_n (1+\theta_n+ \epsilon (1-\theta_n))$. Then, from \eqref{IMpfeq9}, we have 
	\begin{equation*}
		\phi_{n+1} \leq (1+\theta_n) \phi_n - \theta_n \phi_{n-1} - \epsilon (1-\theta_n) d^2(x_{n+1},x_ n)+ K_n d^2(x_n,x_{n-1})  \text{ for all }  n \in \N.
	\end{equation*}
	By the definition of $\psi_n$, we obtain
	\begin{align}
		&\psi_{n+1}- \psi_n \nonumber\\
		&= \phi_{n+1} - \theta_{n+1} \phi_n + K_{n+1} d^2(x_{n+1},x_n)- \phi_n + \theta_n \phi_{n-1}- K_n d^2(x_n,x_{n-1})\nonumber\\
		& = \phi_{n+1} - (1+\theta_{n+1})\phi_n+ \theta_n \phi_{n-1} + K_{n+1} d^2(x_{n+1},x_n)- K_nd^2(x_n,x_{n-1})\nonumber\\
		& \leq \phi_{n+1}- (1+\theta_n) \phi_n+ \theta_n \phi_{n-1} + K_{n+1} d^2(x_{n+1},x_n)- K_nd^2(x_n,x_{n-1})\nonumber\\
		& \leq - \epsilon (1-\theta_n) d^2(x_{n+1},x_n)+ K_n d^2(x_n,x_{n-1})+ K_{n+1} d^2(x_{n+1},x_n)- K_nd^2(x_n,x_{n-1}) \nonumber\\
		&= -(\epsilon(1-\theta_n)-K_{n+1}) d^2(x_{n+1},x_n)  \text{ for all }  n \in \N.\label{IMpfeq10}
	\end{align}
	Note that 
	\begin{align*}
		K_n=  \theta_n (1+\theta_n+ \epsilon (1-\theta_n)) \leq \theta_n (1+\max \{1,\epsilon\})  \text{ for all }  n \in \N
	\end{align*}
	and hence 
	\begin{align*}
		-[\epsilon (1-\theta_n)-K_{n+1}] &\leq -\epsilon + \epsilon \theta_n + \theta_{n+1}(1+\max\{1,\epsilon\}) \leq -\epsilon + \theta (1+\epsilon + \max\{1,\epsilon\}).
	\end{align*}
	From condition \eqref{C_3}, we see that $K= \epsilon- \theta(1+\epsilon + \max\{1,\epsilon\}) >0$. Using \eqref{IMpfeq10}, we obtain
	\begin{equation}\label{IMpfeq11}
		\psi_{n+1}-\psi_n \leq -K d^2(x_{n+1},x_n)  \text{ for all }  n \in \N.
	\end{equation}
	Therefore $\{\psi_n\}$ is nonincreasing.\\
\tr{(ii)}	Note $\psi_1 = \phi_1-\theta_1 \phi_0+ K_1d^2(x_1,x_0) >0$. Since $\{\theta_n\}$ is bounded above by $\theta$, we get
	\begin{align*}
		\psi_n = \phi_n-\theta_n \phi_{n-1} + K_nd^2(x_n,x_{n-1})
		\geq \phi_n- \theta_n \phi_{n-1}
		\geq \phi_n -\theta \phi_{n-1}.
	\end{align*}
	Since $\{\psi_n\}$ is nonincreasing, we get 
	\begin{equation}\label{IMpfeq12}
		\phi_n-\theta \phi_{n-1} \leq \psi_n \leq \psi_1  \text{ for all }  n \in \N.
	\end{equation}
	Thus,
	\begin{align*}
		\phi_n  \leq \theta \phi_{n-1} + \psi_1
		\leq \theta (\theta \phi_{n-2}+\psi_1)+ \psi_1 \leq 
		\cdots
		&\leq \theta^n \phi_0 + \psi_1 \sum_{k=0}^{n-1} \theta^k\\
		&
		\leq \theta^n \phi_0 + \frac{\psi_1}{1-\theta}  \text{ for all }  n \in \N.
	\end{align*}
	It follows from \eqref{IMpfeq11} and $\eqref{IMpfeq12}$ that
	\begin{align*}
		K\sum_{k=1}^n d^2(x_{k+1},x_k) \leq \psi_1-\psi_{n+1}
		\leq\psi_1+ \theta \phi_n
		\leq \psi_1+ \theta\left(\phi_0 \theta^n + \frac{\psi_1}{1-\theta}\right)=
		\theta^{n+1} \phi_0 + \frac{\psi_1}{1-\theta}.
	\end{align*}
	Since $\theta^n \to 0$ as $n \to \infty$, we have 
	\begin{equation}\label{IMpfeq13}
		\sum_{n=1}^\infty d^2(x_{n+1},x_n) \leq \frac{\psi_1}{K(1-\theta)}.
	\end{equation}
\tr{(iii)} It follows from \eqref{IMpfeq13} that
	\begin{equation}\label{IMpfeqq1}
		\lim_{n \to \infty} d(x_n,x_{n+1}) =0.
	\end{equation}
	Using \eqref{IMeq3}, \eqref{Mainineq} and \eqref{IMpfeqq1}, we deduce that 
	\begin{align}
		d(x_{n+1},y_n)&=\|x_{n+1}'-y_n'\|= \|x_{n+1}'-[(1+\theta_n)x_n'+(-\theta_n) {x_{n-1}'}]\| \nonumber\\
		&= \|x_{n+1}'-x_n'+\theta_n (x_{n-1}'-x_n')\| \leq \|x_{n+1}'-x_n'\|+ \theta_n \|{x_{n-1}'}-x_n'\| \nonumber\\
		& = d(x_{n+1},x_n) + \theta_n d(x_{n-1},x_n) \to 0 \text{ as } n \to \infty. \label{IMpfeqq2}
	\end{align}
	Observe from \eqref{IMpfeqq1} and \eqref{IMpfeqq2} that 
	\begin{equation}\label{IMpfeqq3}
		d(x_n,y_n)\leq d(x_n,x_{n+1})+d(x_{n+1},y_n) \to 0 \text{ as } n \to \infty.
	\end{equation}
	From \eqref{IMpfeq9} and \eqref{IMpfeq11}, we have
	$$\phi_{n+1} \leq (1+\theta_n)\phi_n -\theta_n \phi_{n-1} + \theta (1+\max \{1,\epsilon\}) d^2(x_n,x_{n-1})  \text{ for all }  n \in \N.$$
	From \eqref{IMpfeq13} and Lemma \ref{LeSC3}, we obtain that $\lim_{n \to \infty} d(x_n,v)$ exists.
	We may assume that $\lim_{n \to \infty} d(x_n,v)=l >0$. It follows from \eqref{IMpfeqq3} that $\lim_{n \to \infty} d(y_n,v) =l.$ Hence from \eqref{IMeq5}, we obtain that 
	$\lim_{n \to \infty} d(y_n,Ty_n)=0.$\\
\tr{(iv)} Note that $\lim_{n \to \infty} d(y_n,v) =l$ exists and $\lim_{n \to \infty} d(y_n,Ty_n)=0$.	Hence $\{y_n\}$ satisfies the LEAF point property, therefore by Proposition \ref{LEAFpr1}, $\{y_n\}$ converges  to an element of $\fix(T)$. Since $\lim_{n\to \infty} d(x_n,y_n)=0$, we have $\{x_n\}$ converges to an element  of $\fix(T)$.
\end{proof}

\begin{Remark}
	Theorem \ref{InerThm} is the extension of  Theorem 3.3 of Sahu \cite{Sahusoft2020} from Hilbert spaces to the setting of Hadamard manifolds.
\end{Remark}

\begin{Remark}
	Note that the sequence $\{d(x_n, Tx_n)\}$ may not be decreasing for the inertial Mann iteration method \eqref{ALGeq1'}, even when $T$ is a nonexpansive operator. In such a situation, the accuracy of the sequence $\{x_n\}$ in approximating a fixed point of $T$ cannot be measured by $\{d(x_n, Tx_n)\}$. Therefore, we use $R_{T,\{x_n\}}(n)$ to evaluate the accuracy of $\{x_n\}$ in approximating a fixed point of $T$ when $\{d(x_n, Tx_n)\}$ is not essentially decreasing.
\end{Remark}

In the next result, we study the rate of convergence of the inertial Mann iteration method \eqref{ALGeq1'} in terms of $R_{T,\{x_n\}}(n)$.

\begin{Theorem}
Let $\M$ be a Hadamard manifold and let $T\colon  \M \to \M$ be a nonexpansive operator such that $\fix(T)\neq \emptyset$.  Let $\{x_n\}$ be a sequence  generated by   inertial Mann method \eqref{ALGeq1'} with \tr{$x_0=x_1\in \M$}, where $\{\alpha_n\}$ and $\{\theta_n\}$ are sequences satisfying the conditions \eqref{C_1}, \eqref{C_2} and \eqref{C_3}. Then $\{x_n\}$ converges to an element of $\fix(T)$ with the following rate of convergence:
\[R_{T,\{x_n\}}(n)= o\left(\frac{1}{\sqrt{n}}\right).\]
In fact for any $v\in \fix(T)$
	\[(R_{T,\{x_n\}}(n))^2\leq \frac{5}{a(1-b)}\left(1+\frac{\theta^2(1+\theta)}{K(1-\theta)^2}+\frac{\theta(1+\theta)}{K(1-\theta)}\right)\frac{d^2(x_0,v)}{n},\]
	where $K= \epsilon- \theta(1+\epsilon + \max\{1,\epsilon\})$.
\end{Theorem}

\begin{proof}
	Let $v\in \fix(T)$. From \eqref{IMpfeq2} and \eqref{IMeq5}, we have
		\begin{align*}
				d^2(x_{n+1},v)&\leq 	(1+\theta_n)d^2(x_n,v) - \theta_nd^2(x_{n-1},v) + \theta_n(1+\theta_n) d^2(x_n,x_{n-1})\\
				&\quad -\alpha_n (1-\alpha_n) d^2(y_n,Ty_n).
			\end{align*}
			It follows  from condition \eqref{C_1} that
				\begin{align*}
				&d^2(x_{n+1},v)-(1+\theta_n)d^2(x_n,v) + \theta_nd^2(x_{n-1},v) \nonumber\\
				& \leq  -\alpha_n (1-\alpha_n) d^2(y_n,Ty_n)+\theta_n(1+\theta_n) d^2(x_n,x_{n-1})\nonumber\\
				&\leq  -a(1-b)d^2(y_n,Ty_n)+\theta_n(1+\theta_n) d^2(x_n,x_{n-1}),
			\end{align*}
			which implies that
				\begin{align}
				&(d^2(x_{n+1},v)-d^2(x_n,v)) - \theta_n(d^2(x_n,v)-d^2(x_{n-1},v)) \nonumber\\
				&\leq  -a(1-b)d^2(y_n,Ty_n)+\theta_n(1+\theta_n) d^2(x_n,x_{n-1}).\label{RoCeq6}
			\end{align}
			Let $\delta_n=\theta_n(1+\theta_n)d^2(x_n,x_{n-1})$, $\phi_n=d^2(x_n,v)$,  $V_n=\phi_n-\phi_{n-1}$ and $[V_n]_+=\max\{V_n,0\}$ for all $n\in \N$.  From \eqref{RoCeq6}, we have
			\begin{align}
				a(1-b)d^2(y_n,Ty_n) &\leq \phi_n-\phi_{n+1} +\theta_n (\phi_n-\phi_{n-1})+\delta_n\nonumber\\
				&=  \phi_n-\phi_{n+1}+\theta_n V_n+\delta_n.\nonumber\\
				&\leq \phi_n-\phi_{n+1}+\theta_n [V_n]_++\delta_n.\label{RoCeq7}
			\end{align}
			From \eqref{IMpfeq13}, we have 
				\begin{equation}\label{RoCeq8}
				\sum_{n=1}^\infty d^2(x_{n+1},x_n) \leq \frac{\psi_1}{K(1-\theta)}.
			\end{equation}
			Note that $x_0=x_1$.  Then $\psi_1=d^2(x_1,v)-\theta_1d^2(x_0,v)+K_1d^2(x_0,x_1) \leq d^2(x_0,v)$. From \eqref{RoCeq8}, we have 
			\begin{equation*}
				\sum_{n=1}^\infty d^2(x_{n+1},x_n) \leq \frac{d^2(x_0,v)}{K(1-\theta)}.
			\end{equation*}
			It follows from condition \eqref{C_3} that 
			\begin{align}
				\sum_{n=1}^{\infty}\delta_n=\sum_{n=1}^{\infty} \theta_n(1+\theta_n)d^2(x_n,x_{n-1})
				\leq \sum_{n=1}^{\infty} \theta(1+\theta)d^2(x_n,x_{n-1})
				 \leq \frac{ \theta(1+\theta)d^2(x_0,v)}{K(1-\theta)}. \label{RoCeq14}
			\end{align}
			From \eqref{RoCeq6}, we have
			\begin{align*}
		 V_{n+1}\leq \theta_n V_n+\delta_n \leq \theta [V_n]_++\delta_n.
			\end{align*}
			Therefore
			\begin{align*}
				[V_{n+1}]_+\leq \theta [V_n]_++\delta_n.
			\end{align*}
			Applying the recurrence repeatedly, we get 
				\begin{align}\label{RoCeq9}
				[V_{n+1}]_+\leq \theta^n [V_1]_+ +\sum_{j=1}^{n}\theta^{j-1}\delta_{n+1-j}.
			\end{align}
			Since $x_0=x_1$, we get $V_1=[V_1]_+=0$ and $\delta_1=0.$
			From \eqref{RoCeq9}, we have
			\begin{align}
				\sum_{n=2}^{\infty}[V_n]_+\leq \frac{1}{1-\theta} \sum_{n=1}^{\infty}\delta_n=\frac{1}{1-\theta} \sum_{n=2}^{\infty}\delta_n\leq\frac{ \theta(1+\theta)d^2(x_0,v)}{K(1-\theta)^2}.\label{RoCeq13}
			\end{align}
			Since $\phi_1=\phi_0=d^2(x_0,v)$. From   \eqref{RoCeq7}, \eqref{RoCeq14} and \eqref{RoCeq13},  we get 
			\begin{align*}
				a(1-b)\sum_{i=1}^nd^2(y_n,Ty_n)&\leq \phi_1-\phi_n+\theta\sum_{i=1}^n[V]_+ +\sum_{i=2}^n\delta_i \\
				%& \leq \phi_1+\theta\sum_{i=1}^\infty[V]_+\sum_{i=2}^\infty\delta_i \nonumber\\
				&\leq d^2(x_0,v)+\frac{ \theta^2(1+\theta)d^2(x_0,v)}{K(1-\theta)^2}+\frac{ \theta(1+\theta)d^2(x_0,v)}{K(1-\theta)}\\
				 & \leq \left(1+\frac{\theta^2(1+\theta)}{K(1-\theta)^2}+\frac{\theta(1+\theta)}{K(1-\theta)}\right)d^2(x_0,v).
			\end{align*}
			This implies that 
		\begin{align}
				\sum_{i=1}^nd^2(y_n,Ty_n) \leq \frac{1}{a(1-b)}\left(1+\frac{\theta^2(1+\theta)}{K(1-\theta)^2}+\frac{\theta(1+\theta)}{K(1-\theta)}\right)d^2(x_0,v).\label{RoCeq11}
		\end{align}
				By using the nonexpansivity of $T$, $\eqref{ALGeq1'}$ and Proposition $\ref{Newpr1}$,  we deduce that 
			\begin{align*}
				d^2(x_{n+1},Tx_{n+1})&=d^2(\gamma(y_n,Ty_n;\alpha_n),Tx_{n+1})\\
				&\leq (1-\alpha_n)d^2(y_n,Tx_{n+1})+\alpha_n d^2(Ty_n,Tx_{n+1})-\alpha_n(1-\alpha_n)d^2(y_n,Ty_n)\\
				&\leq (d(y_n,Ty_n)+d(Ty_n,Tx_{n+1}))^2+d^2(Ty_n,Tx_{n+1})\\
				& \leq 2 (d^2(y_n,Ty_n)+d^2(Ty_n,Tx_{n+1}))+d^2(Ty_n,Tx_{n+1})\\
				& \leq 2d^2(y_n,Ty_n)+3d^2(y_n,x_{n+1})\\
				& \leq 2d^2(y_n,Ty_n)+3d^2(y_n,\gamma(y_n,Ty_n;\alpha_n))\\
				& \leq 2d^2(y_n,Ty_n)+3\alpha_n^2d^2(y_n,Ty_n)\\
				&\leq 5 d^2(y_n,Ty_n).
			\end{align*}
			From \eqref{RoCeq11}, we have
			\begin{align}\label{RoCeq15}
				\sum_{i=1}^n d^2(x_{i+1},Tx_{i+1}) \leq 5	\sum_{i=1}^n d^2(y_i,Ty_i) \leq  \frac{5}{a(1-b)}\left(1+\frac{\theta^2(1+\theta)}{K(1-\theta)^2}+\frac{\theta(1+\theta)}{K(1-\theta)}\right)d^2(x_0,v),
			\end{align}
		\tr{	which implies that  $\sum_{i=1}^n (R_{T,\{x_n\}}(n))^2 <\infty$.
%			\begin{align*}
%				\sum_{i=1}^n (R_{T,\{x_n\}}(n))^2 <\infty.
%			\end{align*}
It follows from Lemma \ref{DongLemma} that \[R_{T,\{x_n\}}(n)= o\left(\frac{1}{\sqrt{n}}\right).\]} 
		\tr{	Moreover, by \eqref{RoCeq15}, we obtain
			\[	(R_{T,\{x_n\}}(n))^2=\min_{1\leq i\leq n}d^2(x_{n+1},Tx_{n+1}) \leq \frac{5}{a(1-b)}\left(1+\frac{\theta^2(1+\theta)}{K(1-\theta)^2}+\frac{\theta(1+\theta)}{K(1-\theta)}\right)\frac{d^2(x_0,v)}{n}.\]}
\end{proof}

Next, we introduce another accelerated method that does not involve an inertial term. Unlike the inertial Mann iteration method, which requires working in the entire space, the following method can be applied on a convex domain.

\begin{Algorithm}\label{paccalg}
	Let $C$ be a nonempty, closed and geodesically convex subset of a Hadamard manifold $\M$, $p\in \N$ and let $T\colon  C\to C$ be an operator. \\
	\textbf{Initialization:} Choose $x_1\in C$ arbitrary.\\
	\textbf{Iterative step:} For the current $n^{th}$ iterate, compute the $(n+1)^{th}$ iteration as follows:
	\begin{equation}\label{accpeq}
		\begin{cases}
			y_n=\gamma(x_n,Tx_n;\alpha_n),\\
			x_{n+1}=T^p(y_n)  \text{ for all }  n\in \N,
		\end{cases}
	\end{equation}
	where $\{\alpha_n\}$ is a sequence in $(0,1)$.
\end{Algorithm}

Note that this method is inspired from the $p$-accelerated normal S-iteration method \eqref{paccmethod}; accordingly, we refer to it by the same name.
We now study the convergence of the $p$-accelerated normal S-iteration method \eqref{accpeq}.

\begin{Theorem} \label{paccthm}
	Let $C$ be a nonempty, closed and geodesically convex subset of a Hadamard manifold $\M$ and let $T\colon  C\to C$ be a nonexpansive operator such that  $\fix(T)\neq \emptyset$. Let $\{x_n\}$  be the sequence generated by $p$-accelerated normal S-iteration method  \eqref{accpeq}, where $\{\alpha_n\}$ is  a sequence in $(0,1)$ satisfying the condition $\sum_{n=1}^{\infty}\alpha_n (1-\alpha_n)=\infty$. \tr{Then, we have the following:
	\begin{enumerate}
		\item[(i)] $\lim_{n\to \infty} d(x_n,v)$ exists for any $v\in \fix(T)$.
		\item[(ii)] $\sum_{n=1}^{\infty}\alpha_n(1-\alpha_n) d^2(x_n,Tx_n)<\infty$.
		\item[(iii)] $\{d(x_n,Tx_n)\}$ is decreasing.
		%\item[(iv)] $\lim_{n \to \infty}d(x_n,Tx_n)=0$.
		\item[(iv)] $\{x_n\}$ converges to an element of $\fix(T)$.
	\end{enumerate}}
%	  $\{x_n\}$ converges to an element of $\fix(T)$ with the convergence rate gives as follows:
%	\[d(x_n,Tx_n)=o\left(\frac{1}{\sqrt{\sum_{i=1}^{n}\alpha_i(1-\alpha_i)}}\right).\]
\end{Theorem}
\begin{proof}
(i)	Let $v\in \fix(T)$. From \eqref{accpeq} and Proposition \ref{disconv2}, we have
	\begin{align}
		d(y_n,v)&=d(\gamma(x_n,Tx_n;\alpha_n),v)\nonumber\\
		& \leq (1-\alpha_n)d(x_n,v)+\alpha_n d(Tx_n,v)\nonumber\\
		&\leq (1-\alpha_n)d(x_n,v)+\alpha_n d(x_n,v)=d(x_n,v). \label{pacceq1}
	\end{align}
	From \eqref{accpeq}, \eqref{pacceq1} and nonexpansivity of $T$,  we have 
	\begin{equation*}
		d(x_{n+1},v)=d(T^p(y_n),v)\leq d(T^{p-1}y_n,v)\leq \cdots\leq d(y_n,v) \leq d(x_n,v).
	\end{equation*}
%	Hence, $\lim_{n \to \infty}d(x_n,v)$ exists and $\{x_n\}$ is bounded.\\
	%Next we show that $\lim_{n \to \infty}d(x_n,Tx_n)=0$.
(ii)	From Proposition \ref{Newpr1}, we have
	\begin{align}
		d^2(x_{n+1},v)&= d^2(T^p(y_n),v) \nonumber\\
		& \leq d^2(y_n,v)\nonumber\\
		& =d^2(\gamma(x_n,Tx_n;\alpha_n),v)\nonumber\\
		& \leq (1-\alpha_n)d^2(x_n,v)+\alpha_n d^2(Tx_n,v)-\alpha_n(1-\alpha_n)d^2(x_n,Tx_n)\nonumber\\
		& =d^2(x_n,v)-\alpha_n(1-\alpha_n)d^2(x_n,Tx_n). \nonumber
	\end{align}
	It follows that 
	\begin{equation*}
		\alpha_n(1-\alpha_n)d^2(x_n,Tx_n) \leq d^2(x_n,v)-d^2(x_{n+1},v).
	\end{equation*}
	%Since $\lim_{n \to \infty}d(x_n,v)$ exists, from \eqref{pacceq2} we have 
	Summing up the above inequality from $n=1$ to $k$, we have
	\begin{equation}\label{pacceq2}
			\sum_{n=1}^k	\alpha_n(1-\alpha_n)d^2(x_n,Tx_n) \leq \sum_{n=1}^k (d^2(x_n,v)-d^2(x_{n+1},v))=d^2(x_0,v),
	\end{equation}
	which implies that 
	\begin{equation}\label{pacceq3}
		\sum_{n=1}^\infty	\alpha_n(1-\alpha_n)d^2(x_n,Tx_n) < \infty.
	\end{equation}
(iii) Using the nonexpansivity of $T$ and Proposition \ref{disconv2}, we have 
\begin{align}
	d(x_{n+1},Tx_{n+1})& = d(T^p(y_n),T^{p+1}(y_n))\nonumber\\
	& \leq d(y_n,Ty_n)\nonumber\\
	& \leq d(y_n,Tx_n)+d(Tx_n,Ty_n)\nonumber\\
	& \leq d(y_n,Tx_n)+d(x_n,y_n)\nonumber\\
	& = d(\gamma(x_n,Tx_n;\alpha_n), Tx_n)+d(\gamma(x_n,Tx_n;\alpha_n), x_n)\nonumber\\
	& = (1-\alpha_n)d(x_n,Tx_n)+\alpha_nd(Tx_n,Tx_n)+ (1-\alpha_n)d(x_n,x_n)+\alpha_nd(Tx_n,x_n)\nonumber\\
	& = (1-\alpha_n)d(x_n,Tx_n)+\alpha_nd(Tx_n,x_n)=d(x_n,Tx_n). \nonumber
\end{align}	
(iv)	Using \eqref{pacceq3}, we have
	\begin{equation}\label{pacceq4}
		\liminf_{n\to \infty}d(x_n,Tx_n)=0,
	\end{equation}
	because otherwise $d(x_n,Tx_n)\geq K$  for all $n \in \N$ for some $K>0$ and then
	$$\sum_{n=1}^\infty \alpha_n (1-\alpha_n)d^2(x_n,Tx_n)\geq K \sum_{n=1}^\infty \alpha_n (1-\alpha_n)=\infty,$$
	which is a contradiction to the condition \eqref{pacceq3}.
%	On the other hand, using the nonexpansivity of $T$ and Proposition \ref{disconv2}, we have 
%	\begin{align}
%		d(x_{n+1},Tx_{n+1})& = d(T^p(y_n),T^{p+1}(y_n))\\
%		& \leq d(y_n,Ty_n)\nonumber\\
%		& \leq d(y_n,Tx_n)+d(Tx_n,Ty_n)\nonumber\\
%		& \leq d(y_n,Tx_n)+d(x_n,y_n)\nonumber\\
%		& = d(\gamma(x_n,Tx_n;\alpha_n), Tx_n)+d(\gamma(x_n,Tx_n;\alpha_n), x_n)\nonumber\\
%		& = (1-\alpha_n)d(x_n,Tx_n)+\alpha_nd(Tx_n,Tx_n)+ (1-\alpha_n)d(x_n,x_n)+\alpha_nd(Tx_n,x_n)\nonumber\\
%		& = (1-\alpha_n)d(x_n,Tx_n)+\alpha_nd(Tx_n,x_n)=d(x_n,Tx_n). \label{RoCeeq1}
%	\end{align}
%	This means that $\{d(x_n,Tx_n)\}$ is a monotone sequence. 
From (iii) sequence $\{d(x_n,Tx_n)\}$ is decreasing.	Combining this fact with \eqref{pacceq4} gives $\lim_{n \to \infty}d(x_n,Tx_n)=0$. Hence, $\{x_n\}$ satisfies the  LEAF point property. Therefore, from Proposition \ref{LEAFpr1},  $\{x_n\}$ converges to an element of $\fix(T)$.
\end{proof}

\begin{Theorem}
	\tr{Let $C$ be a nonempty, closed and geodesically convex subset of a Hadamard manifold $\M$ and let $T\colon  C\to C$ be a nonexpansive operator such that  $\fix(T)\neq \emptyset$. Let $\{x_n\}$  be the sequence generated by $p$-accelerated normal S-iteration method  \eqref{accpeq}, where $\{\alpha_n\}$ is  a sequence in $(0,1)$ satisfying the condition $\sum_{n=1}^{\infty}\alpha_n (1-\alpha_n)=\infty$. Then, we have the following:}
	\begin{enumerate}
		\item[(i)] $d(x_n,Tx_n) \leq \frac{d(x_0,v)}{\sqrt{\sum_{i=1}^n	\alpha_i(1-\alpha_i)}}$ and  $d(x_n,Tx_n)=o\left(\frac{1}{\sqrt{\sum_{i=1}^{n}\alpha_i(1-\alpha_i)}}\right)$.
		\item[(ii)] If in addition $0< a\leq \alpha_n \leq b<1$, then $d(x_n,Tx_n)=o\left(\frac{1}{\sqrt{n}}\right)$.
	\end{enumerate}
\end{Theorem}
\begin{proof}
(i)	Using \eqref{pacceq2}, we have
	\begin{align}
		\sum_{i=1}^n	\alpha_i(1-\alpha_i)d^2(x_i,Tx_i) \leq \sum_{i=1}^n (d^2(x_i,v)-d^2(x_{i+1},v))=d^2(x_0,v). \label{RoCeq2}
	\end{align}
	From Theorem \ref{paccthm}(iii), sequence $\{d(x_n,Tx_n)\}$ is decreasing. It follows from \eqref{RoCeq2} that 
	\begin{align*}
		d^2(x_n,Tx_n) \sum_{i=1}^n	\alpha_i(1-\alpha_i)\leq d^2(x_0,v).
	\end{align*}
	This implies that 
	\begin{align*}
		d(x_n,Tx_n) \leq \frac{d(x_0,v)}{\sqrt{\sum_{i=1}^n	\alpha_i(1-\alpha_i)}}. 
	\end{align*}
	From Theorem \ref{paccthm}(ii), we have $\sum_{n=1}^{\infty}\alpha_n(1-\alpha_n) d^2(x_n,Tx_n)<\infty$. Since $\sum_{n=1}^{\infty}\alpha_n (1-\alpha_n)=\infty$, it follows from Lemma \ref{DongLemma} that 
    \[d(x_n,Tx_n)=o\left(\frac{1}{\sqrt{\sum_{i=1}^{n}\alpha_i(1-\alpha_i)}}\right).\]
    (ii) \tr{Since $0< a\leq \alpha_n \leq b<1$, from Theorem \ref{paccthm}(ii), we have 
        $\sum_{n=1}^{\infty}a(1-b) d^2(x_n,Tx_n)<\infty$ or  $\sum_{n=1}^{\infty} d^2(x_n,Tx_n)<\infty$.
        Again  it follows from Lemma \ref{DongLemma} that 
    	\[d(x_n,Tx_n) =o\left(\frac{1}{\sqrt{n}}\right).\]}
\end{proof}

\begin{Remark}
	While both  Inertial Mann method \eqref{ALGeq1'} and $p$-accelerated normal S-iteration method \eqref{accpeq} aim to compute a fixed point of the nonexpansive mapping $T$, their convergence behavior differs significantly in practice.
	The $p$-accelerated normal S-iteration method \eqref{accpeq} generates a relatively stable sequence, and the fixed-point residual $d(Tx_n, x_n)$ typically decreases in a monotonic fashion. As such, the value of $d(Tx_n, x_n)$ itself serves as a reliable indicator of the convergence rate.
	
	On the other hand, the inertial Mann iteration method \eqref{ALGeq1'}, which incorporates momentum via previous iterates, produce more oscillatory trajectories due to the presence of momentum terms. As a result, the standard residual $\{d(Tx_n, x_n)\}$ may not monotonically decrease, and may not reflect the actual rate at which the iterates approach the fixed point set. To better assess convergence behavior, we therefore consider the best iterate criterion 
	 $R_{T,\{x_n\}}(n)=\underset{{1 \leq i \leq n} }{\min} d(Tx_i, x_i)$, which is more appropriate for methods with inertial components.
\end{Remark}

Now we introduce two accelerated Douglas-Rachford-type algorithms for solving the minimization problem \eqref{Pbm1} on Hadamard manifolds.
By defining \( T = R_{\lambda \phi} R_{\lambda \psi} \), the classical Douglas-Rachford method can be seen as a special case of the Krasnoselskii-Mann iteration method. The proposed algorithms extend this idea to Hadamard manifolds and incorporate acceleration techniques. Specifically, both algorithms are constructed using the composition of the reflected resolvents \( R_{\lambda \phi} \) and \( R_{\lambda \psi} \),
following the structure of the inertial Mann iteration method \eqref{ALGeq1'} and the $p$-accelerated normal S-iteration method \eqref{accpeq}.

\begin{Algorithm}\label{ALG1DRS}
	Let  $\M$ be a Hadamard manifold and $\phi, \psi \colon  \M \to (-\infty,\infty]$ be proper, geodesically convex and  lower semicontinuous functions.\\
	\textbf{Initialization:} Choose $x_0,x_1\in \M$ arbitrary.\\
	\textbf{Iterative step:} For the current $n^{th}$ iterate, compute the $(n+1)^{th}$ iteration as follows:
	\begin{equation}\label{ALGeq1DRS}
		\begin{cases}
			y_n= \exp_{x_n}(-\theta_n \exp_{x_n}^{-1} x_{n-1}),\\
			x_{n+1} = \gamma(y_n,{R}_{\lambda \phi} {R}_{\lambda \psi}(y_n);\alpha_n)  \text{ for all }  n \in \N,
		\end{cases}
	\end{equation}
	where $\{\alpha_n\}$ and $\{\theta_n\}$ are sequences in $(0,1)$ and $\lambda >0$.
\end{Algorithm}

\begin{Algorithm}\label{paccalgDRS}
	Let  $\M$ be a Hadamard manifold and $\phi, \psi \colon  \M \to (-\infty,\infty]$ be proper, geodesically convex and lower semicontinuous functions. Let $p\in \N$.\\
	\textbf{Initialization:} Choose $x_1\in \M$ arbitrary.\\
	\textbf{Iterative step:} For the current $n^{th}$ iterate, compute the $(n+1)^{th}$ iteration as follows:
	\begin{equation}\label{accpeqDRS}
		\begin{cases}
			y_n=\gamma(x_n,{R}_{\lambda \phi} {R}_{\lambda \psi}(x_n);\alpha_n),\\
			x_{n+1}=({R}_{\lambda \phi} {R}_{\lambda \psi})^p(y_n)  \text{ for all }  n\in \N,
		\end{cases}
	\end{equation}
	where $\{\alpha_n\}$ is a sequence in $(0,1)$ and $\lambda >0$.
\end{Algorithm}

We refer to the method defined by \eqref{ALGeq1DRS} as the inertial Douglas-Rachford method, and the method defined by \eqref{accpeqDRS} as the $p$-accelerated normal S-Douglas-Rachford method.

\begin{Remark}
	\tr{If $\M=\R^m$, the Euclidean space, and we choose $A=\partial \phi$ and $B=\partial \psi$ in the inertial Douglas-Rachford algorithm of Bo\c{t} et al. \cite[Theorem 8]{Bot2015}, then their method is equivalent to Algorithm \ref{ALG1DRS}. Note that we use parameter conditions different from those in Bo\c{t} et al. \cite{Bot2015}; see Remark \ref{ParameterRem}.}
\end{Remark}

We now study the convergence analysis of the inertial Douglas-Rachford method \eqref{ALGeq1DRS} and  the $p$-accelerated normal S-Douglas-Rachford method \eqref{accpeqDRS}, utilizing the above results established for nonexpansive operators.

\begin{Theorem} \label{Corp-Acc}
	Let $\M$ be a Hadamard manifold and 
	let $\phi,\psi\colon \M \rightarrow (-\infty,+\infty]$ be proper, geodesically convex and  lower semicontinuous functions such that
	$\intr(\dom \phi)\cap \dom\psi \neq \emptyset$.  Let $\lambda >0$ and let ${R}_{\lambda \phi}{R}_{\lambda \psi}$ be nonexpansive such that $\fix({R}_{\lambda \phi}{R}_{\lambda \psi})\neq \emptyset$. Let $\{x_n\}$ be the sequence generated by  inertial Douglas-Rachford method
	\eqref{ALGeq1DRS},
	where $\{\alpha_n\}$ and $\{\theta_n\}$ are sequences  satisfying the conditions \eqref{C_1}, \eqref{C_2} and \eqref{C_3}. Then $\{x_n\}$ converges to an element  $v\in \fix({R}_{\lambda \phi} {R}_{\lambda \psi})$ such that  $\prox_{\lambda\psi}(v)$  is solution of problem \eqref{Pbm1}.
\end{Theorem}
\begin{proof}
	Since $T={R}_{\lambda \phi}{R}_{\lambda \psi}$ is nonexpansive. From Theorem \ref{InerThm}, sequence $\{x_n\}$ generated by  \eqref{ALGeq1DRS} converges to an element $v\in \fix({R}_{\lambda \phi}{R}_{\lambda \psi})$. From Theorem \ref{thmequiv}, $u=\prox_{\lambda \psi}(v)$ is solution of problem \eqref{Pbm1}.
\end{proof}

\begin{Theorem} \label{CorIn-M}
	Let $\M$ be a Hadamard manifold and 
	let $\phi,\psi\colon \M \rightarrow (-\infty,+\infty]$ be proper, geodesically convex and lower semicontinuous functions such that $\intr(\dom \phi)\cap \dom\psi \neq \emptyset$. Let $\lambda >0$ and let ${R}_{\lambda \phi}{R}_{\lambda \psi}$ be nonexpansive such that $\fix({R}_{\lambda \phi}{R}_{\lambda \psi})\neq \emptyset$. Let $\{x_n\}$ be the sequence generated by $p$-accelerated normal S-Douglas-Rachford method \eqref{accpeqDRS},
	where $\{\alpha_n\}$ is a sequences in $(0,1)$ satisfying the condition $\sum_{n=1}^\infty \alpha_n(1-\alpha_n)=\infty$. Then $\{x_n\}$ converges to an element  $v\in \fix({R}_{\lambda \phi} {R}_{\lambda \psi})$ such that  $\prox_{\lambda\psi}(v)$  is solution of problem \eqref{Pbm1}.
\end{Theorem}
\begin{proof}
	Since $T={R}_{\lambda \phi}{R}_{\lambda \psi}$ is nonexpansive. From Theorem \ref{paccthm}, sequence $\{x_n\}$ generated by  \eqref{accpeqDRS} converges to an element $v\in \fix({R}_{\lambda \phi}{R}_{\lambda \psi})$. From Theorem \ref{thmequiv}, $u=\prox_{\lambda \psi}(v)$ is solution of problem \eqref{Pbm1}.
\end{proof}

\section{Numerical Experiments}

In this section, we perform numerical experiments on a well-known nonconvex optimization problem, the Rosenbrock problem, defined as follows:
\begin{align}\label{Rosenbrock}
	\min_{x\in \R^2} a(x_1^2-x_2)^2+(x_1-b)^2,
\end{align}
where $a,b\in (0,\infty)$ (see \cite{Rosenbrock}). 
In order to handle the nonconvex problem \eqref{Rosenbrock} in  $\M=\R^2$, consider
\begin{equation*}
	G_x=\begin{pmatrix}
		1+4x_1^2 & -2x_1\\
		-2x_1&1
	\end{pmatrix}  \text{ for all }  x=(x_1,x_2)^T\in \R^2 ,
\end{equation*}
which has the inverse matrix
$$G_x^{-1}=\begin{pmatrix}
	1&2x_1\\
	2x_1&1+4x_1^2
\end{pmatrix}  \text{ for all }  x=(x_1,x_2)^T\in \R^2.$$
Consider the inner product on $\T_x\M=\R^2$ defined by
\begin{equation}\label{Rosenmetric}
	\langle u,v\rangle_x = u^TG_xv  \text{ for all }  u,v \in \T_x\M.
\end{equation}
Then $\M$ is a Hadamard manifold of zero sectional curvature with respect to metric \eqref{Rosenmetric} (see \cite{DaCruz}) and  the following  statements
hold:
\begin{enumerate}
	\item[(a)] The Riemannian distance between any $x=(x_1,x_2)^T$ and $y=(y_1,y_2)^T$ in $\M$ is  given by
	$$d(x,y)= \left((x_1-y_1)^2+ (x_1^2-y_1^2-x_2+y_2)^2\right)^{1/2}.$$
	\item[(b)] The geodesic $\gamma(x,y;\cdot)\colon[0,1]\to \M$ joining the points $x=(x_1,x_2)^T$ and $y=(y_1,y_2)^T$  is
	\begin{eqnarray*}
		\gamma(x,y;t)=\begin{pmatrix}
			x_1+t(y_1-x_1)\\
			(x_2+t((y_2-x_2)-(y_1-x_1^2)+
			t^2(y_1-x_1)^2)
		\end{pmatrix}.
	\end{eqnarray*}
	\item[(c)]  For $x=(x_1,x_2)^T$, the exponential map $\exp_x\colon \T_x\M\to \M$ is given by
	$$\exp_x(u)=\begin{pmatrix}
		x_1+u_1\\
		x_2+u_2+u_1^2
	\end{pmatrix}  \text{ for all }  u=(u_1,u_2)^T\in \T_x\M,$$
	which has the inverse  $\exp_x^{-1}\colon\M \to \T_x\M$ given by
	$$\exp_x^{-1}(y)=\begin{pmatrix}
		y_1-x_1\\
		y_2-x_2-(y_1-x_1)^2
	\end{pmatrix}  \text{ for all }  y=(y_1,y_2)^T\in \M.$$
\end{enumerate}

\begin{Remark}\label{Isomrem}
	\begin{enumerate}[label=(\roman*)]
		\item If $f\colon \M \to \R$ is  a real-valued function, then, the Riemannian gradient of $f$, $\grad f\colon \M \to T\M$, is given by 
		$$\grad f(x)=G_x^{-1}\nabla f(x)  \text{ for all }  x\in \M.$$
		\item If $\Phi\colon\R^2 \to \M$ is an operator defined by 
		$$\Phi(x)=(x_1,x_1^2-x_2)  \text{ for all }  x=(x_1,x_2)^T\in \R^2,$$
		then $\Phi$ is an isometry between the Euclidean space $\R^2$ and $\M$.
	\end{enumerate}
\end{Remark}
For $a,b\in (0,\infty)$, consider functions $\phi,\psi\colon \M\to \R$ defined by 
\begin{equation} \label{phi+psi}
	\phi(x)=a(x_1^2-x_2)^2 \text{ and } \psi(x)=(x_1-b)^2  \text{ for all }  x=(x_1,x_2)^T\in \M.
\end{equation}
Using the isometry $\Phi$, we get 
$$\phi\circ \Phi(x)=a(x_1^2-x_1^2+x_2)^2=ax_2^2  \text{ for all }  x=(x_1,x_2)^T\in \M$$
and 
$$\psi\circ \Phi(x)=(x_1-b)^2  \text{ for all }  x=(x_1,x_2)^T\in \M.$$
Thus $\phi$ and $\psi$ are geodesically convex on $\M$. For more details of this splitting; see \cite{BergmannJOTA2024}.

We observe that 
\begin{enumerate}
	\item[(i)]  
	The  Euclidean gradients of $\phi$ and $\psi$ are given by
	$$\nabla \phi(x)=\begin{pmatrix}
		4ax_1(x_1^2-x_2)\\
		-2a(x_1^2-x_2)
	\end{pmatrix} \text{ and } \nabla \psi(x)=\begin{pmatrix}
		2(x_1-b)\\
		0
	\end{pmatrix}  \text{ for all }  x=(x_1,x_2)^T\in \R^2.$$
	\item[(ii)]  In view of Remark \ref{Isomrem}(i), the gradient of $\phi$ and $\psi$ are given by
	$$\grad \phi(x)=\begin{pmatrix}
		0\\
		-2a(x_1^2-x_2)
	\end{pmatrix} \text{ and } \grad \psi(x)=\begin{pmatrix}
		2(x_1-b)\\
		4x_1(x_1-b)
	\end{pmatrix}  \text{ for all }  x=(x_1,x_2)^T\in \M.$$
	\item[(iii)]For $\lambda >0$,  the proximal operators $\prox_{\lambda \phi}$ and $\prox_{\lambda \psi}$ are defined by 
	$$\operatorname{Prox}_{\lambda \phi}(x)=\begin{pmatrix}
		x_1\\
		\frac{x_2+2a\lambda x_1^2}{1+2a\lambda}
	\end{pmatrix}  \text{ for all }  x=(x_1,x_2)^T\in \M \text{ and }$$
	$$\operatorname{Prox}_{\lambda \psi}(x)=\begin{pmatrix}
		\frac{x_1+2\lambda b}{1+2\lambda}\\
		x_2- \frac{4\lambda(x_1+2\lambda b)(x_1-b)+4\lambda^2(x_1-b)^2}{(1+2\lambda)^2}
	\end{pmatrix}  \text{ for all }  x=(x_1,x_2)^T\in \M.$$
	For $\lambda >0$, the reflection  $R_{\lambda \phi}$ of  $\prox_{\lambda \phi}$ and the reflection $R_{\lambda \psi}$ of $\prox_{\lambda \psi}$ are defined by
	\begin{equation}\label{Ref1}
		R_{\lambda \phi}(x)=\begin{pmatrix}
			x_1\\
			\frac{2(x_2+2a\lambda x_1^2)}{1+2a\lambda}-x_2
		\end{pmatrix}  \text{ for all }  x=(x_1,x_2)^T\in \M \text{ and }
	\end{equation}
	\begin{equation}\label{Ref2}
		R_{\lambda \psi}(x)=\begin{pmatrix}
			\frac{(1-2\lambda)x_1+4\lambda b}{1+2\lambda}\\
			x_2-\frac{8\lambda(x_1+2\lambda b)(x_1-b)}{(1+2\lambda)^2}
		\end{pmatrix}  \text{ for all }  x=(x_1,x_2)^T\in \M.
	\end{equation}
\end{enumerate}
%\begin{Remark} \label{reflectionrem}
	By Proposition \ref{Newpr3}, the reflections $R_{\lambda \phi}$ and $R_{\lambda \psi}$  defined by \eqref{Ref1} and \eqref{Ref2}, respectively, are nonexpansive.
	
%\end{Remark}

Note that, the Rosenbrock problem \eqref{Rosenbrock} can be transformed into the following minimization problem:
\begin{equation}\label{RosenManif}
	\min_{x\in \M} \phi(x)+\psi(x),
\end{equation}
where $\phi$ and $\psi$ are defined by \eqref{phi+psi}.

For the numerical experiment, we compare three different methods: the Douglas-Rachford method [Alg(DR)] \eqref{Bergmannalg}, the inertial Douglas-Rachford method [Alg(InDR)] \eqref{ALGeq1DRS} and the $p$-accelerated normal S-Douglas-Rachford method [Alg($p$-AccDR)] \eqref{accpeqDRS}. All numerical tests are carried out using MATLAB R2024b on an Apple Macbook Air Laptop with a 1.8 GHz Dual-Core Intel Core i5 processor and 8 GB RAM. All algorithms stops when they meet the stopping criterion $E_n=d(s_n,s_{n-1})<10^{-14}$, where $s_n=\prox_{\lambda \psi}(x_n)$.

Numerical experiments are performed on the Rosenbrock  function \eqref{Rosenbrock} with parameters
\((a,b)=(1,2)\) and \((3,7)\).
The initial point is chosen as \(x_0=(1,2)^{\top}\) for all methods, while
Method~(InDR) uses \(x_1=(1,3)^{\top}\) in both cases.

%We conduct numerical for these two examples $a=1$, $b=2$ and $\lambda=0.5$ and $a=3$, $b=7$ and $\lambda=1$. 
%We take the initial point $x_0=(1,2)^T$ for all algorithms and $x_{1}=(1,3)^T$ for Alg(InDR) for both examples. 

We set $\lambda=1$ for all three algorithms and iteration parameters  are chosen as follows:

\begin{table*}[h]
	\centering
	%\caption{Parameter settings for the algorithms}
	\begin{tabular}{lllll}
		Algorithm & $(a,b)=(1,2)$ & $(a,b)=(1,2)$  & $(a,b)=(3,7)$& $(a,b)=(3,7)$ \\
		\hline
		Alg(DR) & $\alpha_n = 0.1$ & $\alpha_n = 0.3$ & $\alpha_n = 0.2$ & $\alpha_n = 0.4$ \\
		Alg(InDR) & $\alpha_n = 0.1,$ & $\alpha_n = 0.3,$ & $\alpha_n = 0.2,$ & $\alpha_n = 0.4,$ \\
		& $ \theta_n = 0.42$ & $ \theta_n = 0.35$ & $\theta_n = 0.4$ & $ \theta_n = 0.33$ \\
		Alg($p$-AccDR) & $\alpha_n = 0.1$ & $\alpha_n = 0.3$ & $\alpha_n = 0.2$ & $\alpha_n = 0.4$ \\
	\end{tabular}
\end{table*}
\tr{Note that in our parameter choice, we have taken the same $\alpha_n$ (in the Mann-type step) for all methods. Corresponding to this choice of $\alpha_n$, we selected the parameter $\theta_n$ according to condition \eqref{C_3} for Alg(InDR), that is, $\epsilon = (1-\alpha_n)/\alpha_n$, $\epsilon_1 = 1 + \epsilon + \max\{1,\epsilon\}$, and $\theta_n \approx 0.9(\epsilon/\epsilon_1)$. We chose $p = 2$, which is the smallest value that yielded the best performance in our experiments. These parameters were chosen empirically by testing several reasonable values.}

%\begin{eqnarray*}
%	\text{[Alg(DR)]} &:& \alpha_n=0.5;\\
%	\text{[Alg(InDR)]} &:& \alpha_n=0.5, \theta_{n}=0.3;\\
%	\text{[Alg($p$-AccDR)]} &:& \alpha_n=0.5, \beta_{n}=0.3.	
%\end{eqnarray*}
%We have chosen the same 

The runtime and number of iterations are presented in Table \ref{TableIner1} for $(a,b)=(1,2)$ and in Table \ref{TableIner2} for $(a,b)=(3,7)$. The graph of $E_n$ values with respect to the number of iterations $n$ is shown in Figure \ref{FigureIner1} for both cases. From Tables \ref{TableIner1}, \ref{TableIner2}, and Figure \ref{FigureIner1}, it is evident that Alg($p$-AccDR) is faster than Alg(InDR) and Alg(DR) in terms of both runtime and number of iterations.

\begin{table}[h]
	\centering
	\caption{Numerical performance of the runtime and number of iterations for Rosenbrock problem with $(a,b)=(1,2)$}
	\begin{tabular}{lllll}
		\hline
		 & $(\alpha_n, \theta_n)=(0.1,0.42)$&  & $(\alpha_n, \theta_n)=(0.3,0.35)$ &\\
		\hline
		Algorithm & Runtime(sec) & Iterations & Runtime(sec) & Iterations\\
		\hline
		Alg(DR)& 0.000800 &213 & 0.000468 &63\\
		Alg(InDR)& 0.000461 &102 & 0.000287&43\\
		Alg($p$-AccDR) & 0.000088 &15 & 0.000099 &13\\
		\hline
	\end{tabular}
	\label{TableIner1}
\end{table}

\begin{table}[h]
	\centering
	\caption{Numerical performance of the runtime and number of iterations for Rosenbrock problem with $(a,b)=(3,7)$}
	\begin{tabular}{lllll}
		\hline
		& $(\alpha_n, \theta_n)=(0.2,0.4)$&  & $(\alpha_n, \theta_n)=(0.4,0.33)$ &\\
		\hline
		Algorithm & Runtime(sec) & Iterations & Runtime(sec) & Iterations\\
		\hline
		Alg(DR)& 0.001003 &106 & 0.000623 &44\\
		Alg(InDR)& 0.000385 &55 & 0.000368 &37\\
		Alg($p$-AccDR) & 0.000150 &30 & 0.000130 &19\\
		\hline
	\end{tabular}
	\label{TableIner2}
\end{table}

\begin{figure}[htbp]
	\centering
	\begin{subfigure}[b]{0.48\textwidth}
		\includegraphics[width=\textwidth]{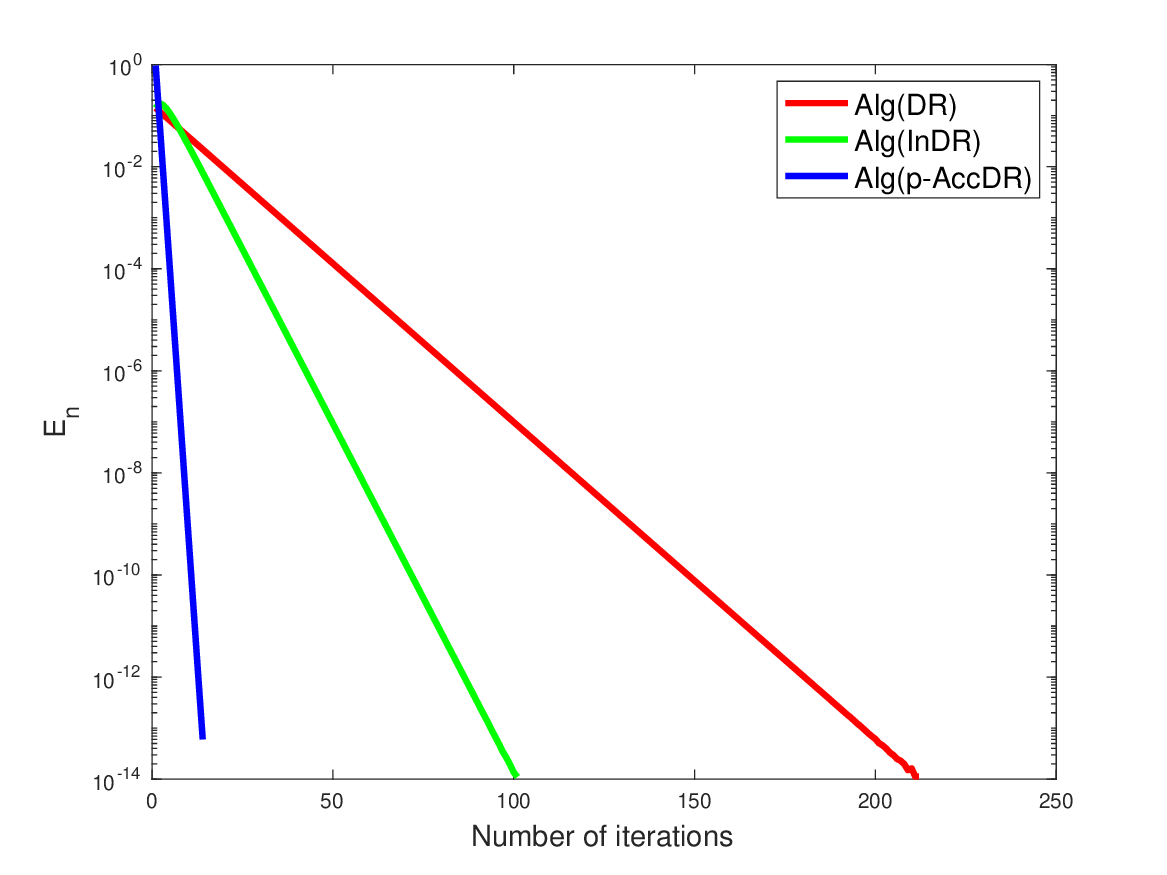}
		\caption{$(a,b)=(1,2), (\alpha_n,\theta_n)=(0.1,0.42)$}
		\label{}
	\end{subfigure}
	\hfill
	\begin{subfigure}[b]{0.48\textwidth}
		\includegraphics[width=\textwidth]{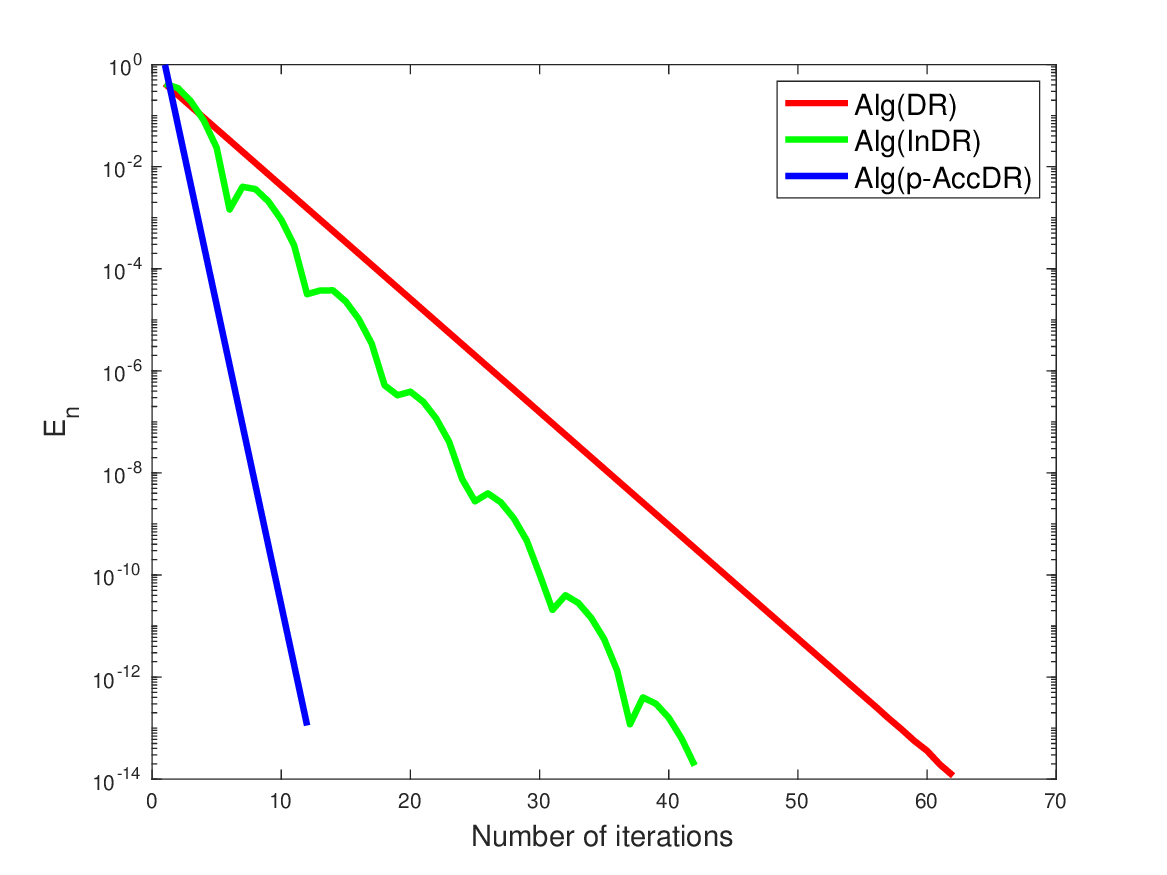}
		\caption{$(a,b)=(1,2), (\alpha_n,\theta_n)=(0.3,0.35)$}
		\label{}
	\end{subfigure}
	\hfill
	\begin{subfigure}[b]{0.48\textwidth}
		\includegraphics[width=\textwidth]{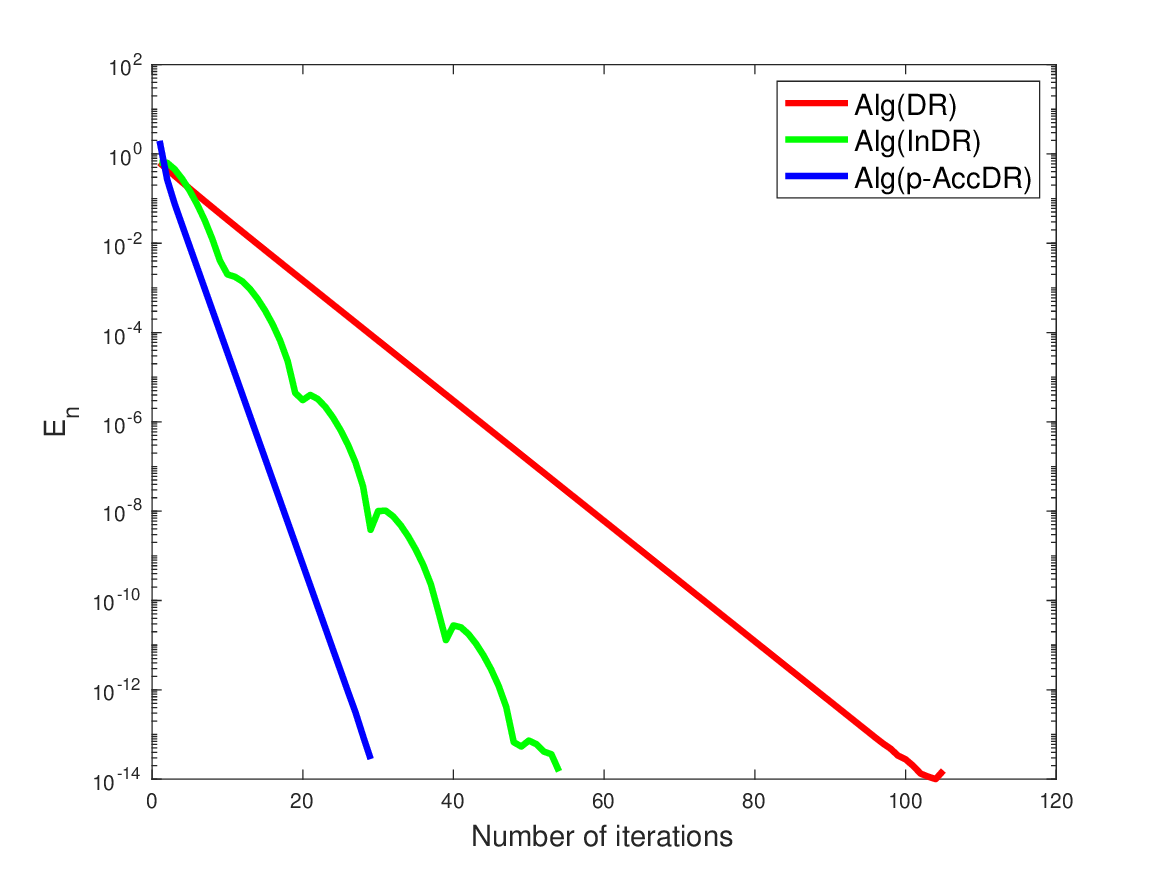}
		\caption{$(a,b)=(3,7), (\alpha_n,\theta_n)=(0.2,0.4)$}
		\label{}
	\end{subfigure}
	\hfill
	\begin{subfigure}[b]{0.48\textwidth}
		\includegraphics[width=\textwidth]{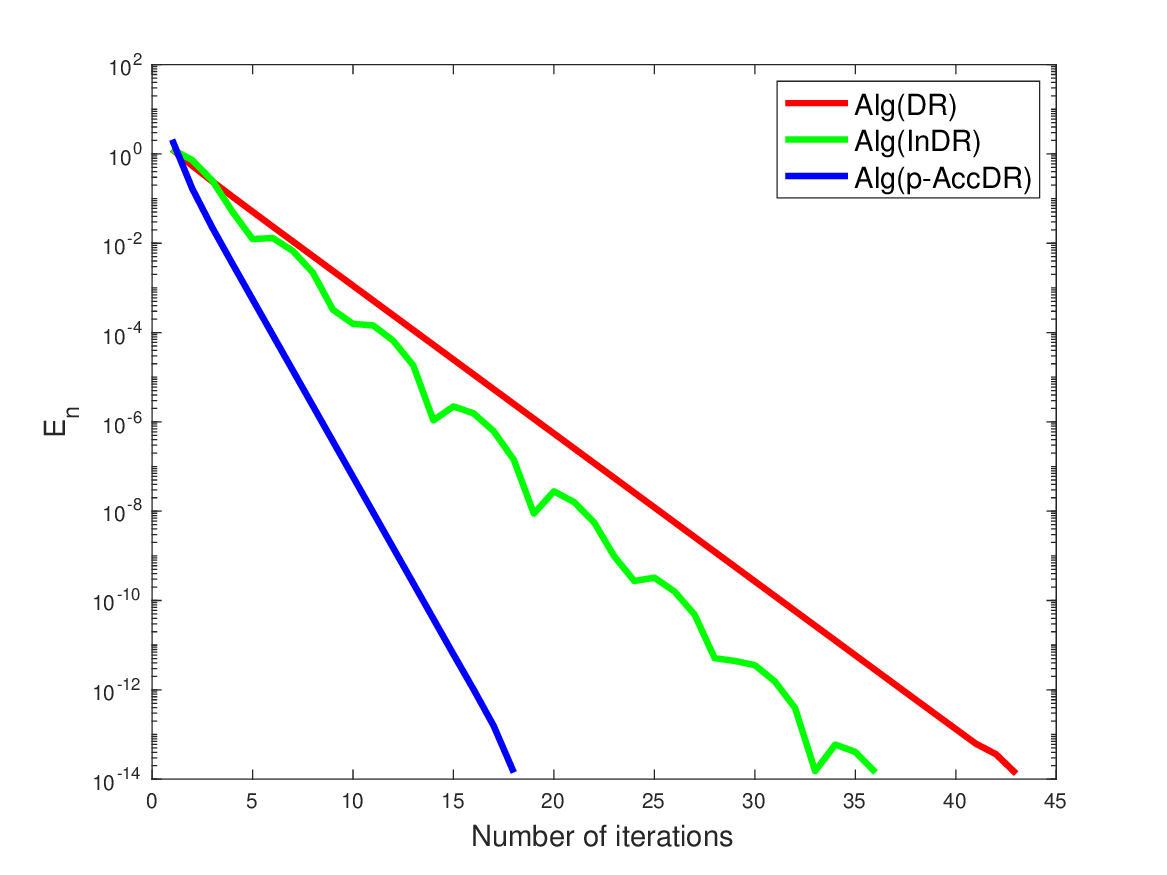}
		\caption{$(a,b)=(3,7), (\alpha_n,\theta_n)=(0.4,0.33)$}
		\label{}
	\end{subfigure}
	\caption{Comparison of $E_n$ with respect to number of iterations for the Rosenbrock Problem}
	\label{FigureIner1}
\end{figure}

\section{Applications}

In this section, we present applications of the proposed algorithms for solving the problem of minimizing a functional involving multiple summands. To this end, we develop parallel versions of the proposed algorithms and apply them to find the solution of generalized Heron problem on Hadamard manifolds.

\subsection{Minimization of functional containing multiple summands}

Let $\M$ be a Hadamard manifold.  Let $\varLambda_N=\{1,2,\ldots,N\}$ and  let $f_k\colon  \M \rightarrow (-\infty,+\infty]$, $k\in \varLambda_N$ be proper, geodesically convex and  lower semicontinuous functions.
Consider the following problem of  minimization of functionals containing multiple summands:
\begin{align} \label{Pbm2}
	\min_{x\in \M}  \sum_{k\in \varLambda_N} f_k(x).
\end{align}
Now, consider the product manifold $\M^{N}= \{\mathbf{x}=(x_1,\ldots,x_N)\colon  x_k\in \M, k\in \varLambda_N\}$. Define 
\begin{equation}\label{D}
	\mathbf{D}=\{\mathbf{x}\in\M^{N} \colon x_1=\dots=x_N\in\M\}.
\end{equation}
Then $\mathbf{D}$ is a nonempty, closed and geodesically convex set. Hence, $\iota_{\mathbf{D}}$ is a proper, geodesically convex and  lower semicontinuous function; see \cite{BacakNAA2014}. For 
$\mathbf{x}= (x_1,\ldots,x_N) \in \M^{N}$
\begin{equation*}
	\operatorname{Prox}_{\lambda\iota_{\mathbf{D}}}(\mathbf{x})
	= P_{\mathbf{D}}(\mathbf{x})
	= \Bigl(
	\amin_{x \in \M} \sum_{k\in \varLambda_N} d^2(x_k,x), \ldots, \amin_{x \in \M} \sum_{k\in \varLambda_N} d^2(x_k,x)\Bigl)
	\in\M^{N}.
\end{equation*}
Define $F\colon \M^{N}\to (-\infty,\infty]$ by 
\begin{equation}\label{F}
	F (\mathbf{x})=\sum_{k\in \varLambda_N} f_k(x_k)  \text{ for all }  \mathbf{x}=(x_1,\ldots,x_N)\in \M^{N}.
\end{equation}
We are interested in minimization of $F$ over $\mathbf{D}$. For this, we consider the following unconstrained minimization problem:
\begin{equation}\label{Pbm3}
	\min_{\mathbf{x}\in\M^{N}} F(\mathbf{x})+\iota_{\mathbf{D}}(\mathbf{x}).
\end{equation}
A detailed explanation of the above formulation is given in \cite{Bergmann2016}. Note that for $\lambda >0$
\begin{align*}
	\operatorname{Prox}_{\lambda F}(\mathbf{x})
	&=\left(\prox_{\lambda f_1}(x_1), \ldots, \prox_{\lambda f_N}(x_N) \right)  \text{ for all }  \mathbf{x}\in \M^{N}.
\end{align*}
It follows that the reflection of $F$ is just given by the componentwise reflections of $\prox_{\lambda f_k}$, i.e.,  
\begin{equation}\label{CompRefl}
	R_{\lambda F}(\mathbf{x})=\left(R_{\lambda f_1}(x_1),\ldots,R_{\lambda f_N}(x_N)\right)  \text{ for all }  \mathbf{x}\in \M^{N}.
\end{equation}

We now present two methods based on the inertial Douglas-Rachford method \eqref{ALGeq1DRS} and the $p$-accelerated normal S-Douglas-Rachford method \eqref{accpeqDRS} for finding a solution to problem \eqref{Pbm2}.

\begin{Corollary} 
	Let $\M$ be a Hadamard manifold and 
	let $f_k\colon \M \rightarrow (-\infty,+\infty]$, $k\in \varLambda_N$ be proper, geodesically convex and lower semicontinuous functions. Let $\mathbf{D}$ be the set defined by \eqref{D} and $F\colon \M^N \to (-\infty,+\infty]$ be the function defined by \eqref{F} such that $\intr(\dom F)\cap \dom\iota_{\mathbf{D}} \neq \emptyset$ and  ${R}_{\lambda F} {R}_{\lambda\iota_{\mathbf{D}}}$ is nonexpansive such that $\fix({R}_{\lambda F} {R}_{\lambda\iota_{\mathbf{D}}})\neq \emptyset$. Let $\{\mathbf{x}_n\}$ be the sequence generated by 
	\begin{equation}\label{ALGeq2'}
		\begin{cases}
			\mathbf{y}_n= \exp_{\mathbf{x}_n}(-\theta_n \exp_{\mathbf{x}_n}^{-1} \mathbf{x}_{n-1}),\\
			\mathbf{x}_{n+1} = \gamma(\mathbf{y}_n,{R}_{\lambda F} {R}_{\lambda\iota_{\mathbf{D}}}(\mathbf{y}_n);\alpha_n)  \text{ for all }  n \in \N,
		\end{cases}
	\end{equation}
	where $\{\alpha_n\}$ and $\{\theta_n\}$ are sequences satisfying the conditions \eqref{C_1}, \eqref{C_2} and \eqref{C_3}. Then $\{\mathbf{x}_n\}$ converges to an element  $\mathbf{v}\in \fix({R}_{\lambda F} {R}_{\lambda\iota_{\mathbf{D}}})$ such that  $\prox_{\lambda\iota_{\mathbf{D}}}(\mathbf{v})$ is solution of problem \eqref{Pbm3}.
\end{Corollary}
\begin{proof}
	Since $T=R_{\lambda F}R_{\lambda\iota_{\mathbf{D}}}$ is nonexpansive. From Theorem \ref{Corp-Acc}, sequence $\{\mathbf{x}_n\}$ generated by \eqref{ALGeq2'} converges to an element $\mathbf{v}\in \fix({R}_{\lambda F} {R}_{\lambda\iota_{\mathbf{D}}})$. From Theorem \ref{thmequiv}, we have $\mathbf{u}=\prox_{\lambda\iota_{\mathbf{D}}}(\mathbf{v})$ is solution of problem \eqref{Pbm3}.
\end{proof}

\begin{Corollary} 
	Let $\M$ be a Hadamard manifold and 
	let $f_k\colon \M \rightarrow (-\infty,+\infty]$, $k\in \varLambda_N$ be proper, geodesically convex and lower semicontinuous functions. Let $\mathbf{D}$ be the set defined by \eqref{D} and $F\colon \M^N \to (-\infty,+\infty]$ be the function defined by \eqref{F} such that $\intr(\dom F)\cap \dom\iota_{\mathbf{D}} \neq \emptyset$ and  ${R}_{\lambda F} {R}_{\lambda\iota_{\mathbf{D}}}$ is nonexpansive such that $\fix({R}_{\lambda F} {R}_{\lambda\iota_{\mathbf{D}}})\neq \emptyset$. Let $\{\mathbf{x}_n\}$ be the sequence generated by 
	\begin{equation}\label{ALGeq22}
		\begin{cases}
			\mathbf{y}_n= \gamma(\mathbf{x}_n,{R}_{\lambda F} {R}_{\lambda\iota_{\mathbf{D}}}(\mathbf{x}_n);\alpha_n),\\
			\mathbf{x}_{n+1} =({R}_{\lambda F} {R}_{\lambda\iota_{\mathbf{D}}})^p(\mathbf{y}_n)   \text{ for all }  n \in \N,
		\end{cases}
	\end{equation}
	where $\{\alpha_n\}$ is a sequence in $(0,1)$ satisfying the condition $\sum_{n=1}^\infty \alpha_n(1-\alpha_n)=\infty$. Then $\{\mathbf{x}_n\}$ converges to an element  $\mathbf{v}\in \fix({R}_{\lambda F} {R}_{\lambda\iota_{\mathbf{D}}})$ such that  $\prox_{\lambda\iota_{\mathbf{D}}}(\mathbf{v})$  is solution of problem \eqref{Pbm3}.
\end{Corollary}
\begin{proof}
	Since $T=R_{\lambda F}R_{\lambda\iota_{\mathbf{D}}}$ is nonexpansive. From Theorem \ref{CorIn-M}, sequence $\{\mathbf{x}_n\}$ generated by \eqref{ALGeq22} converges to an element $\mathbf{v}\in \fix({R}_{\lambda F} {R}_{\lambda\iota_{\mathbf{D}}})$. From Theorem \ref{thmequiv}, we have $\mathbf{u}=\prox_{\lambda\iota_{\mathbf{D}}}(\mathbf{v})$ is solution of problem \eqref{Pbm3}.
\end{proof}

We refer to the method defined by \eqref{ALGeq2'} as the inertial parallel Douglas-Rachford method, and the method defined by \eqref{ALGeq22} as the $p$-accelerated normal S-parallel Douglas-Rachford method.

\subsection{The generalized Heron problem in Hadamard manifolds}

Heron from Alexandria (10-75 AD) was ``a Greek geometer and inventor whose writings preserved for posterity a knowledge of the mathematics and engineering of Babylonia, ancient Egypt, and the Greco-Roman world'' (from the Encyclopedia Britannica). One of the geometric problems he proposed in his Catroptica was as follows: find a point on a straight line in the plane such that the sum of the distances from it to two given points is minimal.

In 2012, Mordukhovich et al. \cite{Mordukhovich2012}  generalized the Heron problem.
In the examination of the generalized Heron problem, as discussed in reference \cite{Mordukhovich2012, Mordukhovich2012(2), Bot2015, Dixit2022, ShikherOpt2024, ReichFT2023}, the conventional notions of points and straight lines are substituted with nonempty, closed  and convex sets. Within this context, the objective is to locate a point within a predetermined nonempty, closed and convex set $C\subset \R^m$ in such a way that the combined distance to specified nonempty, closed  and convex sets $C_k\subset \R^m, k\in \varLambda_{N}$ is minimized. Mathematically, we can write
\begin{equation}\label{Heronpbm}
	\text{minimize } \sum_{k\in \varLambda_{N}}d(x,C_k) \text{ subject to } x\in C \subset \R^m.
\end{equation}

Now, we consider the generalized Heron problem in the setting of Hadamard manifold $\M$, where the objective is to locate a point within a predetermined nonempty, closed and geodesically convex set $C\subset \M$ in such a way that the combined distance to specified nonempty, closed and geodesically convex sets $C_k\subset \M, k\in \varLambda_{N}$ is minimized. Mathematically, the generalized Heron problem can be formulated as:
\begin{equation}\label{HeronpbmManifold}
	\text{minimize } \sum_{k\in \varLambda_{N}}d(x,C_k) \text{ subject to } x\in C\subset \M.
\end{equation}	
The constrained problem \eqref{Heronpbm} can be formulated into the following unconstrained problem:
\begin{equation}\label{Heronunpbm}
	\min_{x\in \M} \sum_{k\in \varLambda_{N}}\phi_k(x)+\iota_{C}(x),
\end{equation}
where $\phi_k(x)=d(x,C_k)$ for all $k\in \varLambda_{N}$.
In the view of \eqref{Pbm3}, problem \eqref{Heronunpbm} can be written as
\begin{equation} \label{Pbm4}
	\min_{\mathbf{x}\in\M^{N+1}} F(\mathbf{x})+\iota_{\mathbf{D}}(\mathbf{x}) ,
\end{equation}
where $F(\mathbf{x})=\sum_{k\in \varLambda_N} d(x_k,C_k)+\iota_C(x_{N+1})$ and 
$\mathbf{D}$=$\{\mathbf{x}\in\M^{N+1} \colon x_1=\dots=x_{N+1}\}$.

\begin{Remark}
\tr{By \eqref{CompRefl}, the reflection of $F$ defined in \eqref{Pbm4} is given by
\[
R_{\lambda F}(\mathbf{x})
=\bigl(R_{\lambda d(x_1,C_1)}(x_1),\ldots,
R_{\lambda d(x_N,C_N)}(x_N),
R_{\lambda \iota_C}(x_{N+1})\bigr)
\quad \text{for all } \mathbf{x}\in \mathcal{M}^{N+1}.
\]
Since $\mathcal{M}^{N+1}$ has constant sectional curvature, it follows from
Proposition~\ref{kappapr2} and Theorem~\ref{distNE} that the reflection
$R_{\lambda F}$ is nonexpansive. Similarly, by Proposition~\ref{kappapr2},
the reflection $R_{\lambda \iota_{\mathbf{D}}}$ is nonexpansive, since
$\mathbf{D}$ is a nonempty, closed, and geodesically convex set.}
\end{Remark}

\begin{Example} 
	Let $\M=\R^m_{++}=\{x=(x_1,\ldots , x_m)\in \R^m\colon  x_i>0, i=1,\ldots,m\}$. 
	For $x\in \M$, let 
	\begin{equation*} %\label{HG}
		G(x)=\operatorname{diag}(1/x_1^2, \ldots,1/x_m^2)
		\text{ and } H(x)=\operatorname{diag}(x_1, \ldots,x_m)
	\end{equation*} 
	be two diagonal matrices.  Then  $\M$ is a Riemannian manifold with the
	Riemannian metric  defined by 
	\begin{equation}\label{Riemmetric}
		\langle u,v \rangle_x =\langle u,G(x)v\rangle \text{ for } x\in \M \text{ and } u,v \in \T_x\M=\R^m.
	\end{equation}
	The Riemannian distance $d\colon  \M\times \M\to [0,\infty)$ is given by 
	$$d(x,y)= \left(\sum_{i=1}^m\left(\ln \left(\frac{x_i}{y_i}\right)\right)^2\right)^{1/2}  \text{ for all }   x=(x_1, \ldots, x_m), y=(y_1, \ldots, y_m) \in  \M.$$
	$\M$ is a Hadamard manifold with zero sectional curvature (see \cite{DaCruz}). For $x=(x_1, \ldots, x_m),y=(y_1, \ldots, y_m) \in \M$, the geodesic joining $x$ to $y$ is given by $$\gamma(x,y;t)=(x_1^{1-t}y_1^t, \ldots, x_m^{1-t}y_m^t)  \text{ for all }  t\in [0,1];$$
	the exponential map  $\exp_x \colon \T_x\M\to \M$ is given by 
	$$\exp_x (v) =\left(x_1 e^{(v_1/x_1)},\ldots, x_m e^{(v_m/x_m)}\right)  \text{ for all }  v=(v_1,\ldots,v_m) \in \T_x\M;$$
	the  inverse of the exponential map $\exp_x^{-1}\colon \M\to \T_x\M$ is given by 
	\begin{equation*} %\label{InvExp}
		\exp^{-1}_x y = \left(x_1 \ln \left({y_1}/{x_1}\right), \ldots, x_m \ln \left({y_m}/{x_m}\right)\right)  \text{ for all }  y=(y_1, \ldots, y_m) \in \M.
	\end{equation*}
\end{Example}

We denote the closed ball centered at $c\in \M$ with radius $r$ by $B_r[c]$, i.e.,
$B_r[c]=\{x\in \M: d(x,c)\leq r\}.$

\begin{Remark} \label{RemarkProxb}
	For the Hadamard manifold $\M=\R_{++}^m$ with the Riemannian metric defined by \eqref{Riemmetric}, we have the following:
	\begin{enumerate}
		\item[(a)] $\prox_{\lambda\iota_{\mathbf{D}}}(\mathbf{x})=P_{\mathbf{D}}(\mathbf{x})=\left(b,\ldots,b\right)\in \M^{N}$ for all $\mathbf{x}=(x_1,\ldots,x_{N})\in \M^N$, where $x_k=(x_{k,1},x_{k,2},\ldots,x_{k,m})$, $k\in \varLambda_{N}$ and $b=\left(\left(\prod_{k\in \varLambda_{N}}x_{k,1}\right)^\frac{1}{N},\ldots,\left(\prod_{k\in \varLambda_{N}}x_{k,m}\right)^\frac{1}{N}\right)$.
		\item[(b)] $P_{B_r[c]}(x)=\begin{cases}
			\gamma\left(c,x;\frac{r}{d(c,x)}\right), & \text{ if } d(c,x)>r,\\
			x, & \text{ if } d(c,x)\leq r.
		\end{cases}$
		\item[(c)] $\prox_{\lambda d(x,c)}(x)=\begin{cases}
			\gamma(x,c;\theta), &\text{ if } d(x,c)>\lambda,\\
			c, &\text{ if } d(x,c)\leq \lambda,
		\end{cases}$  where $\theta=\frac{\lambda}{d(x,c)}$.
		\item[(d)] $\prox_{\lambda d(x,B_r[c])}(x)=\begin{cases}
			\gamma(x,P_{B_r[c]}(x);\theta), &\text{ if } d(x,B_r[c])>\lambda,\\
			P_{B_r[c]}(x), &\text{ if } d(x,B_r[c])\leq \lambda,
		\end{cases}$  where $\theta=\frac{\lambda}{d(x,B_r[c])}$.
	\end{enumerate} 
\end{Remark}

For the numerical experiment, we first consider the generalized Heron problem for a ball with point constraints in $\M=\R_{++}^m$ with the Riemannian metric defined by \eqref{Riemmetric}, i.e.,
\begin{equation}\label{DiluteHeron}
	\text{minimize } \sum_{k\in \varLambda_{N}}d(x,a_k) \text{ subject to } x\in C\subset \M,
\end{equation}
where $a_k\in \M, k\in \varLambda_{N}$. 

We compute the solution of the generalized Heron problem using parallel Douglas-Rachford method \cite[Algorithm 4]{Bergmann2016}, inertial parallel Douglas-Rachford method \eqref{ALGeq22} and $p$-accelerated normal S-parallel Douglas-Rachford  method \eqref{ALGeq2'} and compare based on their efficiency. 

By using Remark \ref{RemarkProxb}(a), we may assume that $\prox_{\lambda\iota_{\mathbf{D}}}(\mathbf{x}_n)=(t_n,\ldots,t_n)\in \M^{N+1}$. We then define the error term by \(Er(n)=d(t_n,t_{n-1})\). In all experiments, the algorithm is terminated when the stopping criterion \(Er(n)\leq 10^{-10}\) is satisfied.

We take $\lambda=1$ and iteration  parameters for these algorithms  are listed below:
\begin{eqnarray*}
	\text{\cite[Algorithm 4]{Bergmann2016}} \text{[Alg(PDRA)]} &:& \alpha_n=0.4;\\
	\text{Algorithm \eqref{ALGeq22}} \text{[Alg(In-M)]} &:& \alpha_n=0.4, \theta_{n}=0.08;\\
	\text{Algorithm \eqref{ALGeq2'}} \text{[Alg($p$-Acc)]} &:& \alpha_n=0.4.	
\end{eqnarray*}
\tr{Note that we use the same parameter $\alpha_n$ for all the algorithms. However, the inertial parameter $\theta_n$ is chosen according to condition (C3), that is, $\epsilon = (1-\alpha_n)/\alpha_n$, $\epsilon_1 = 1 + \epsilon + \max\{1,\epsilon\}$, and $\theta_n \approx 0.25(\epsilon/\epsilon_1)$. We chose $p = 2$, which is the smallest value that yielded the best performance in our experiments. The parameters were chosen by testing several reasonable values, and those giving the best convergence are reported.}

\begin{Example} \label{DiscConsEx}
	Consider the generalized Heron problem \eqref{DiluteHeron}, where $m=2,N=2$,  the constraint set $C=B_{r}[c]$ is a ball in $\M$ centered at $c=(35,35)$ with radius $r=0.4$. Consider two cases:\\
	\textbf{Case 1:} Let $a_1=(15,70)$ and $a_2=(70,15)$.\\
	\textbf{Case 2:} Let $a_1=(5,50)$ and $a_2=(50,5)$.\\
	Note that in both cases, $C$ does not contain $a_1$ and $a_2$. Now we have two observations:
	
	\begin{enumerate}
		\item[(i)] In Case 1, the geodesic segment $\gamma(a_1,a_2; \cdot)$ intersects the ball $C$; see Figure \ref{fig:intersect}. Then any point in the intersection is a solution of problem \eqref{DiluteHeron}. In this case, problem \eqref{DiluteHeron} has infinitely many solutions.
		
		\item[(ii)] In Case 2, the geodesic segment $\gamma(a_1, a_2; \cdot)$ does not intersect the ball $C$; see Figure \ref{fig:notintersect}. Then, there is a unique point on the boundary of the ball $C$.
	\end{enumerate}

\end{Example}

\begin{figure}[htbp]
	\centering
	\begin{subfigure}[b]{0.45\textwidth}
		\includegraphics[width=\textwidth]{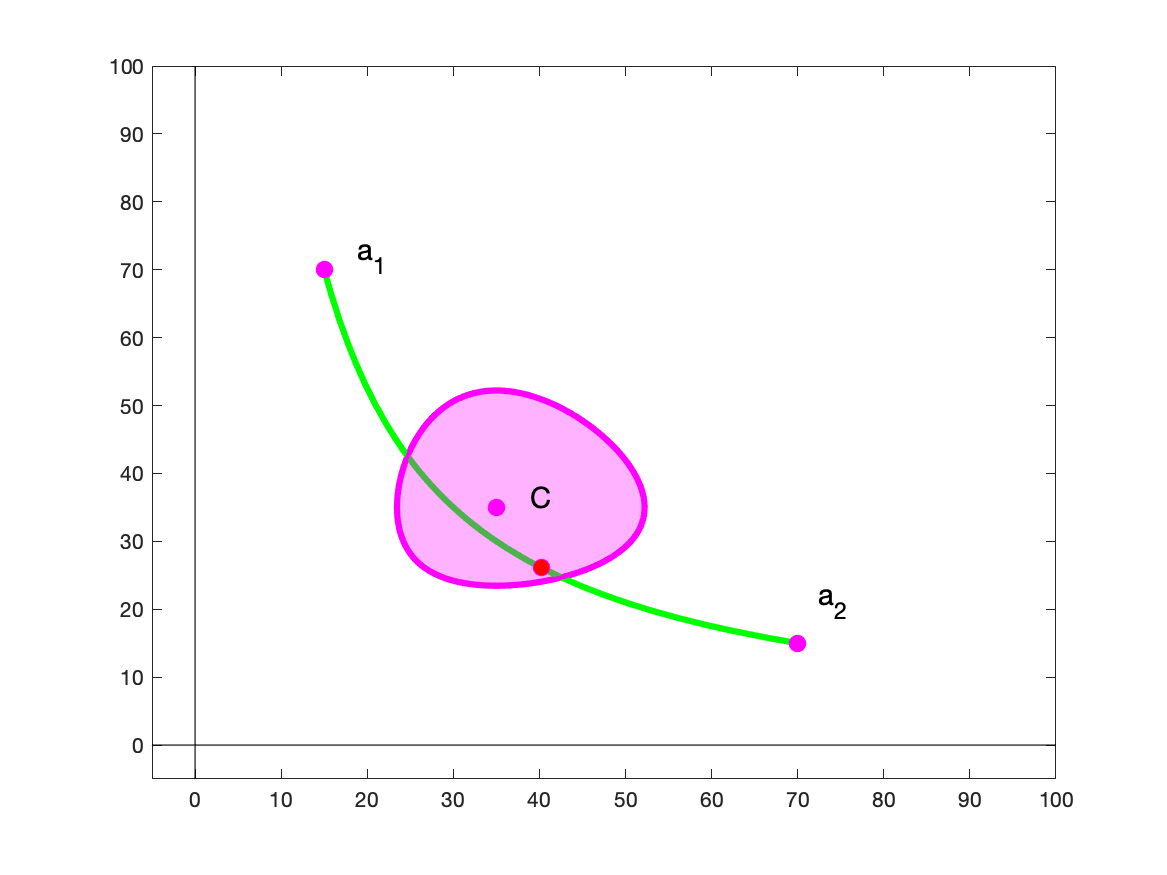}
		\caption{Intersects}
		\label{fig:intersect}
	\end{subfigure}
	\hfill
	\begin{subfigure}[b]{0.45\textwidth}
		\includegraphics[width=\textwidth]{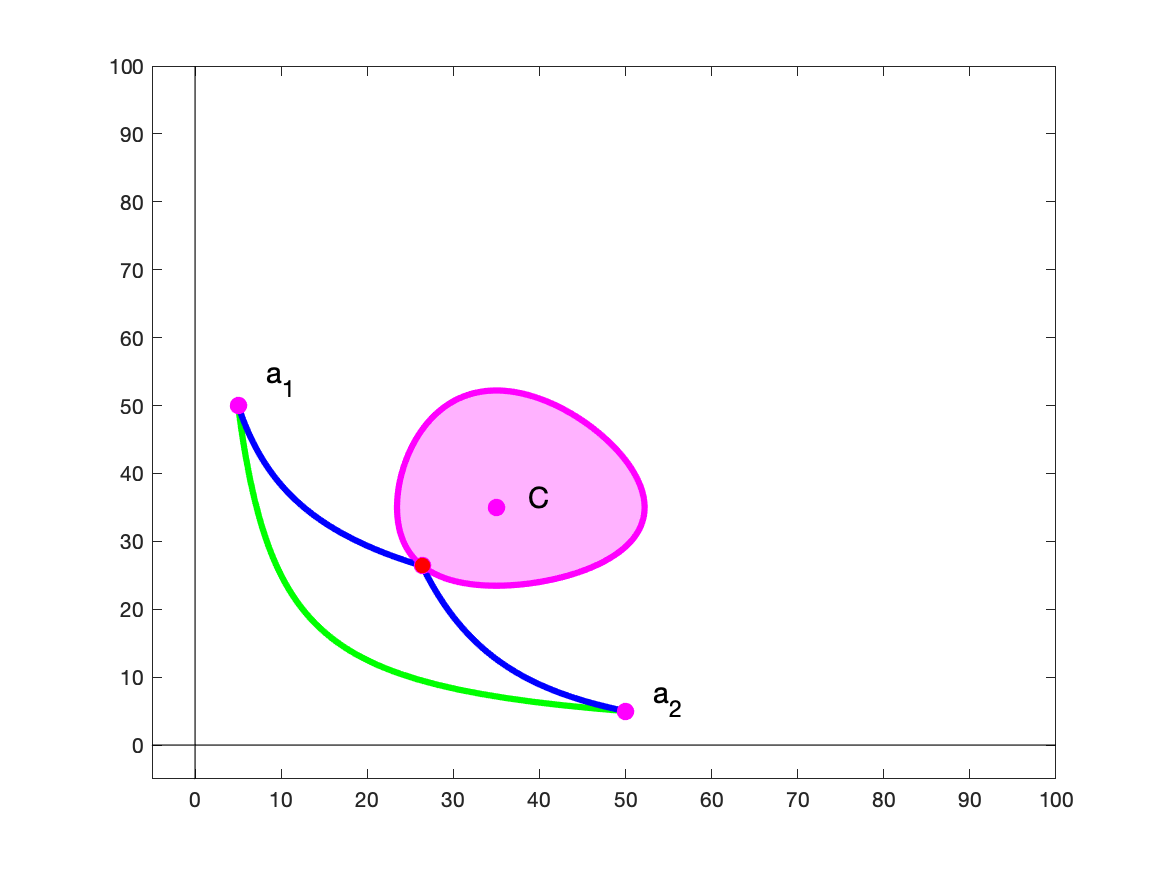}
		\caption{Does not intersect}
		\label{fig:notintersect}
	\end{subfigure}
	\caption{Generalized Heron problems for two points with a ball constraint.}
	\label{fig:discCons}
\end{figure}

We take the initial point $\mathbf{x}_0=ones(1,6)$ for all the algorithms and $\mathbf{x}_1=2*ones(1,6)$ for Alg(In-M) for both cases (Case 1 and Case 2) in Example \ref{DiscConsEx}. The numerical performance of algorithms Alg(PDRA), Alg(In-M) and Alg($p$-Acc) are depicted in Table \ref{compareDisctab}. We plot the relationship between the number of iterations and \(Er(n)\); see Figure~\ref{compareDisc}. From Table~\ref{compareDisctab} and Figure~\ref{compareDisc}, we observe that Alg(\(p\)-Acc) requires less computational time and fewer iterations to achieve the desired stopping criterion in both cases.

\begin{table}[htbp]
	\begin{center}
		\begin{minipage}{\textwidth}
			\caption{Numerical performance in terms of CPU time and number of iterations for Alg(PDRA), Alg(In-M), and Alg($p$-Acc)  for Example \ref{DiscConsEx}}
			\label{compareDisctab}
			\begin{tabular*}{\textwidth}{@{\extracolsep{\fill}}lcccccc@{\extracolsep{\fill}}}
				\hline
				&&Case 1&&&Case 2&\\
				\hline%
				Algorithm & CPU(s) & iter(n) & $Er(n)$ &CPU(s) & iter(n) & $Er(n)$\\
				\hline
				Alg(PDRA) & $0.029219$ & $75$ & $7.4085*10^{-11}$&      $0.029682$ & $99$ & $9.7691*10^{-11}$\\
				Alg(In-M) & $0.015613$ & $46$ & $9.7055*10^{-11}$ & $0.022472$ & $66$ & $7.3605*10^{-11}$\\
				Alg($p$-Acc)&  $0.012893$& $19$ & $6.3342*10^{-11}$ &  $0.012026$& $31$ & $3.6084*10^{-11}$\\
				\hline
			\end{tabular*}
			
		\end{minipage}
	\end{center}
\end{table}

\begin{figure}[htbp]
	\centering
	\begin{subfigure}[b]{0.48\textwidth}
		\includegraphics[width=\textwidth]{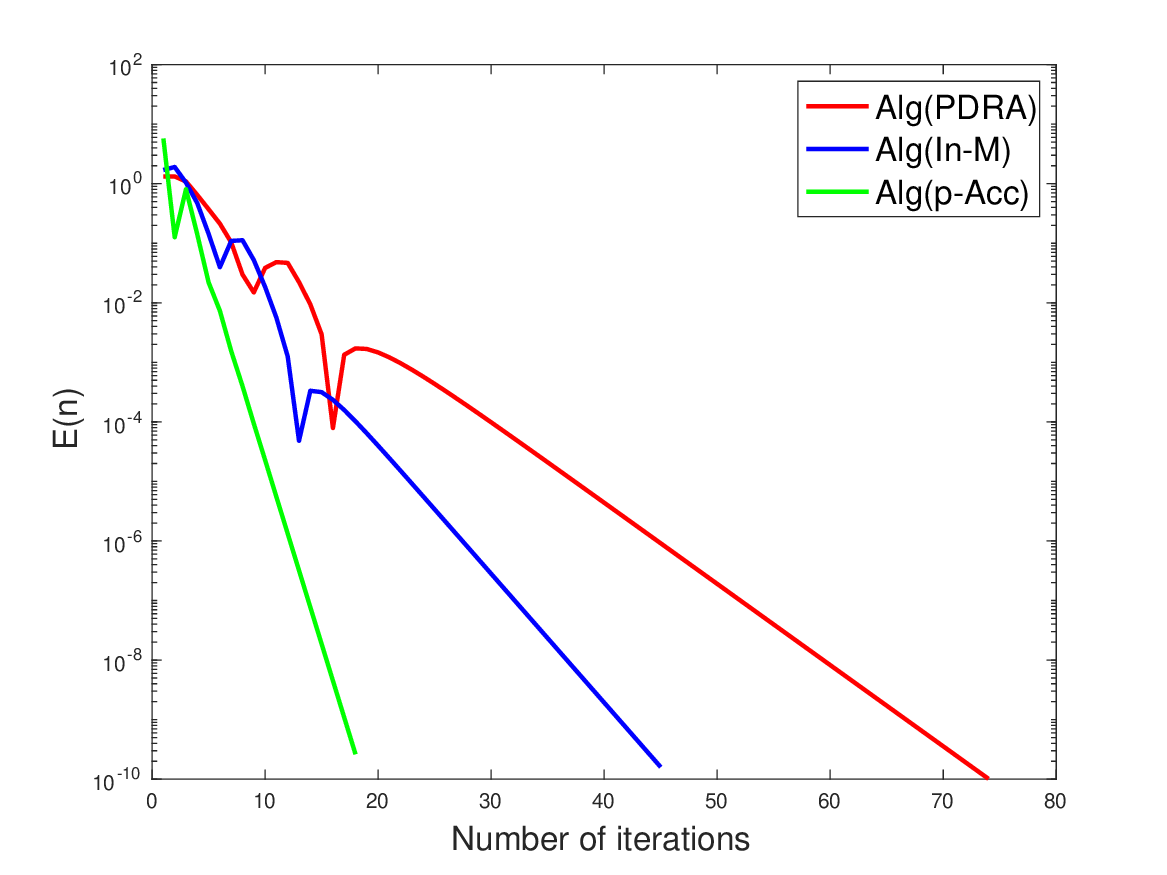}
		\caption{Case 1}
		\label{disc_case1}
	\end{subfigure}
	\hfill
	\begin{subfigure}[b]{0.48\textwidth}
		\includegraphics[width=\textwidth]{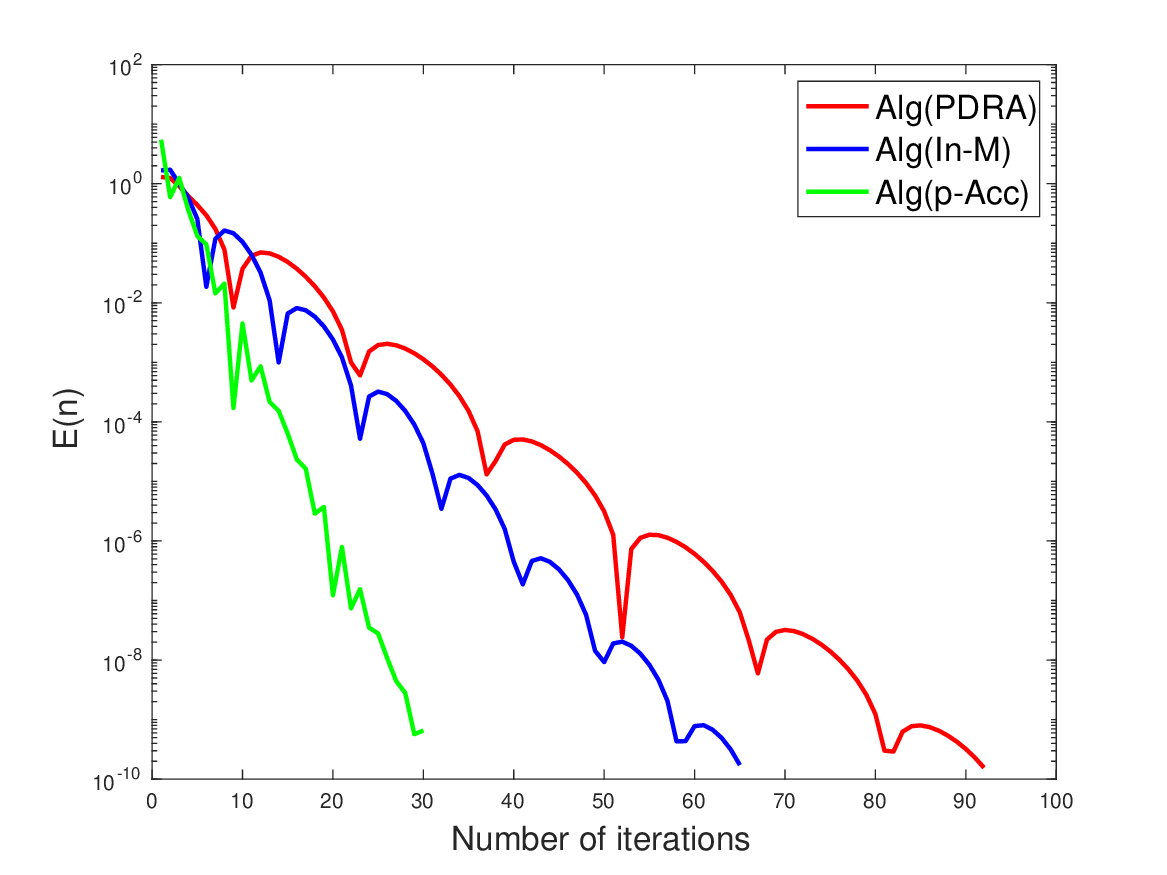}
		\caption{Case 2}
		\label{fig:disc_case2}
	\end{subfigure}
	\caption{Comparison of Alg(PDRA), Alg(In-M), and Alg($p$-Acc) based on the relationship between the number of iterations and $Er(n)$ for Example~\ref{DiscConsEx}.}
	\label{compareDisc}
\end{figure}

\begin{Example} \label{GeneralExmp}
	Consider the generalized Heron problem \eqref{DiluteHeron} for the following two cases:\\
	\textbf{Case 1:} $m=2$, $N=4$.\\
	We take predetermined set $C$ as ball $B_r[c]$ with $c=(35,35)$  and $r=0.4$
	and choose constraints sets to be points
	$a_1=(15,15)$, $a_2=(65,65)$, $a_3 = (10,60)$  and  $a_4=(60,10)$;
	see Figure \ref{case1heron}.\\
	\textbf{Case 2:} $m=20$, $N=10$.\\
	 We take    predetermined set $C$ as ball $B_r[c]$ with 
	$$c=(41, 29, 61, 39, 121, 79, 99, 51, 81, 59, 111, 39, 79, 31, 61, 19, 121, 61, 81, 39) \text{ and } r=0.4$$
	and choose constraints sets to be points $a_k$, $k=1,2,\ldots,10$, where 
	
	\begin{align*}
		a_1&=(35, 30, 60, 40, 120, 80, 100, 50, 80, 60, 110, 40, 80, 30, 60, 20, 120, 60, 80, 40),\\
		a_2&=(45, 30, 60, 40, 120, 80, 100, 50, 80, 60, 110, 40, 80, 30, 60, 20, 120, 60, 80, 40),\\
		a_3&=(35, 30, 55, 40, 120, 80, 100, 50, 80, 60, 110, 40, 80, 30, 60, 20, 120, 60, 80, 40),\\
		a_4&=(35, 30, 65, 40, 120, 80, 100, 50, 80, 60, 110, 40, 80, 30, 60, 20, 120, 60, 80, 40),\\
		a_5&=(35, 30, 60, 40, 115, 80, 100, 50, 80, 60, 110, 40, 80, 30, 60, 20, 120, 60, 80, 40),\\
		a_6&=(35, 30, 65, 40, 125, 80, 100, 50, 80, 60, 110, 40, 80, 30, 60, 20, 120, 60, 80, 40),\\
		a_7&=(35, 30, 65, 40, 120, 80, 95, 50, 80, 60, 110, 40, 80, 30, 60, 20, 120, 60, 80, 40),\\
		a_8&=(35, 30, 65, 40, 120, 80, 105, 50, 80, 60, 110, 40, 80, 30, 60, 20, 120, 60, 80, 40),\\
		a_9&=(35, 30, 65, 40, 120, 80, 100, 50, 75, 60, 110, 40, 80, 30, 60, 20, 120, 60, 80, 40),\\
		a_{10}&=(35, 30, 65, 40, 120, 80, 100, 50, 85, 60, 110, 40, 80, 30, 60, 20, 120, 60, 80, 40).
	\end{align*}
\end{Example}
	
\begin{figure}
	\centering
	\includegraphics[width=0.6\linewidth]{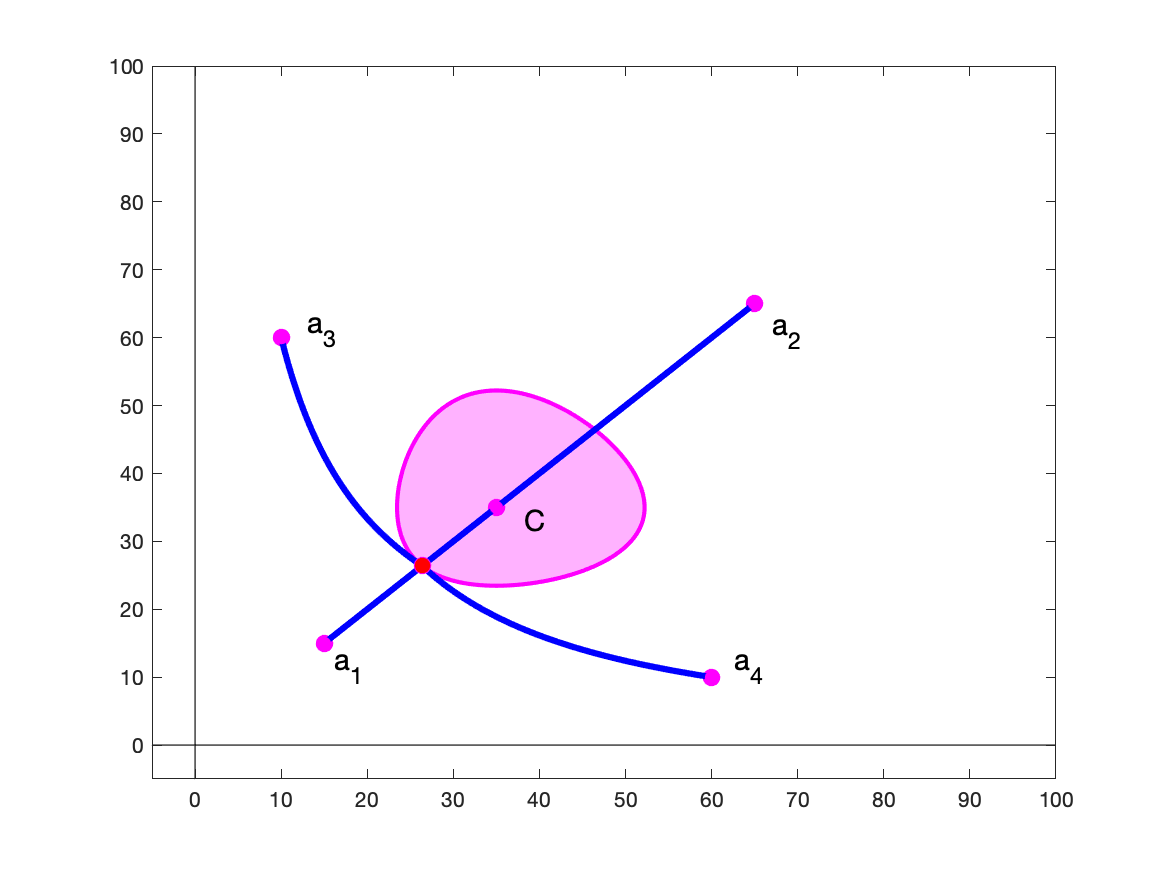}
	\caption{The generalized Heron problem for four points with a ball constraint}
	\label{case1heron}
\end{figure}

We take the initial point $\mathbf{x}_0=ones(1,m*(N+1))$ for all the algorithms and $\mathbf{x}_1=2*ones(1,m*(N+1))$ for Alg(In-M) for both cases (Case 1 and Case 2) in Example \ref{GeneralExmp}.
The numerical performance of algorithms Alg(PDRA), Alg(In-M) and Alg($p$-Acc) are depicted in Table \ref{Entab} for both cases (Case 1 and Case 2). We plot the relationship between the number of iterations and \(Er(n)\) to have a clear visualization of the experiment; see Figure~\ref{compareEn}. 
From Table \ref{Entab} and  Figure \ref{compareEn}, we see that our Algorithm Alg($p$-Acc) requires less computational time and less number of iterations to achieve the desired stopping criterion in both cases.
From Figure \ref{comparefffff}, we can see the convergence behaviour of algorithms Alg(PDRA), Alg(In-M) and Alg($p$-Acc)  to achieving the desired goal as number of iteration increases upto $30$th iteration for Case 2.

\begin{table}
	\begin{center}
		\begin{minipage}{\textwidth}
			\caption{Numerical performance in terms of CPU time and number of iterations for Alg(PDRA), Alg(In-M), and Alg($p$-Acc)  for Example \ref{GeneralExmp}}
			\label{Entab}
			\begin{tabular*}{\textwidth}{@{\extracolsep{\fill}}lcccccc@{\extracolsep{\fill}}}
				\hline
				&&Case 1&&&Case 2&\\
				\hline%
				Algorithm & CPU(s) & iter(n) & $Er(n)$ &CPU(s) & iter(n) & $Er(n)$\\
				\hline
				Alg(PDRA) & $0.029219$ & $113$ & $5.2545*10^{-11}$&      $0.029682$ & $192$ & $9.3761*10^{-11}$\\
				Alg(In-M) & $0.015613$ & $83$ & $4.4691*10^{-11}$ & $0.022472$ & $135$ & $9.0710*10^{-11}$\\
				Alg($p$-Acc)&  $0.012893$& $34$ & $6.0025*10^{-12}$ &  $0.012026$& $62$ & $4.3786*10^{-11}$\\
				\hline
			\end{tabular*}

		\end{minipage}
	\end{center}
\end{table}

\begin{figure}[htbp]
	\centering
	\begin{subfigure}[b]{0.48\textwidth}
		\includegraphics[width=\textwidth]{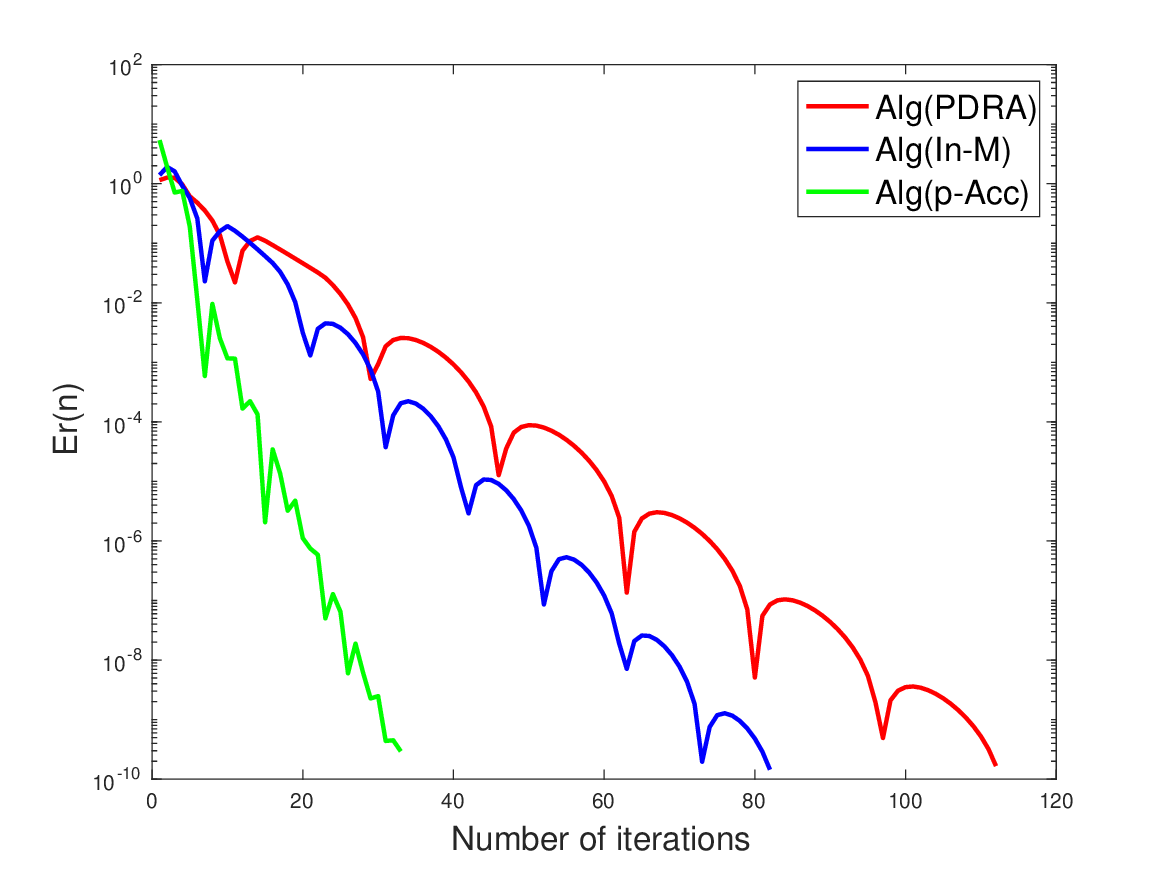}
		\caption{Case 1}
	\end{subfigure}
	\hfill
	\begin{subfigure}[b]{0.48\textwidth}
		\includegraphics[width=\textwidth]{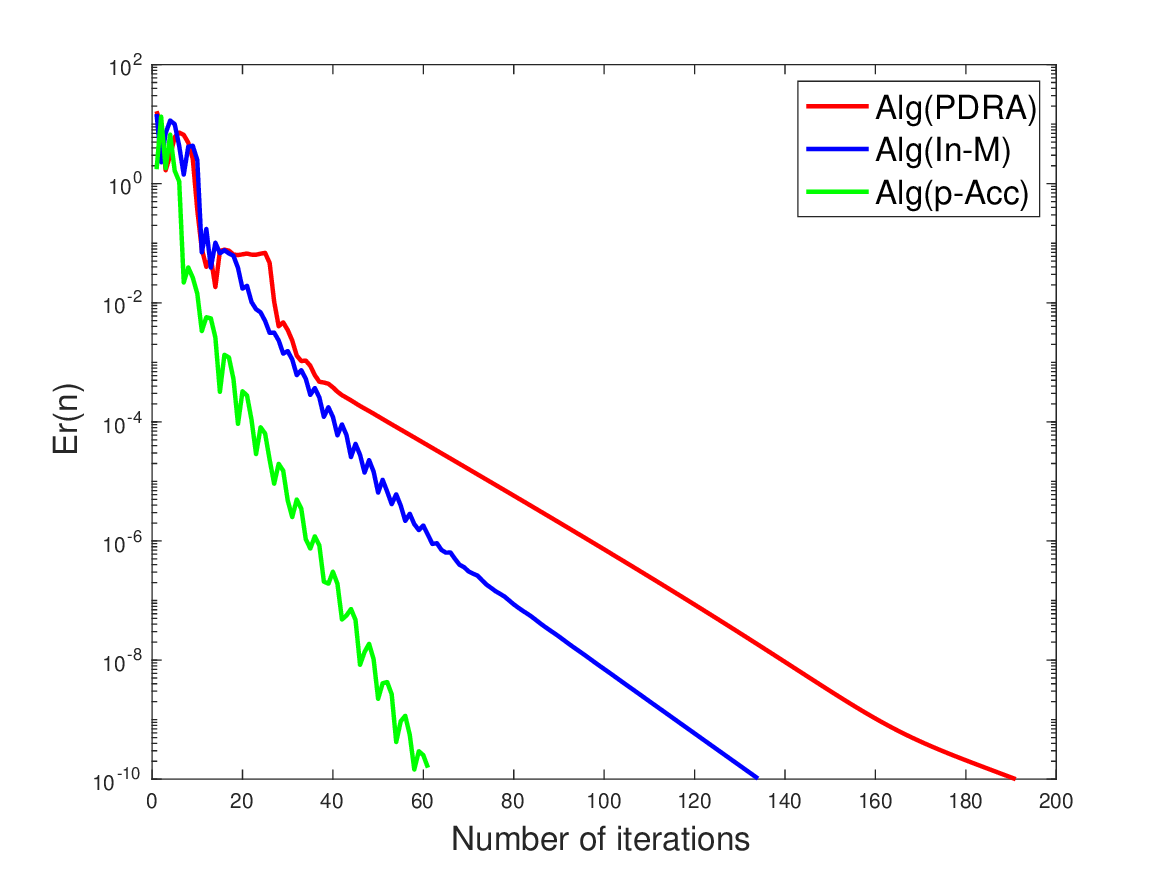}
		\caption{Case 2}
	\end{subfigure}
	\caption{Comparison of Alg(PDRA), Alg(In-M), and Alg($p$-Acc) based on the relationship between the number of iterations and $Er(n)$ for Example~\ref{GeneralExmp}.}
	\label{compareEn}
\end{figure}

\begin{figure}[htbp]
	\centering
	\begin{subfigure}[b]{0.45\textwidth}
		\includegraphics[width=\textwidth]{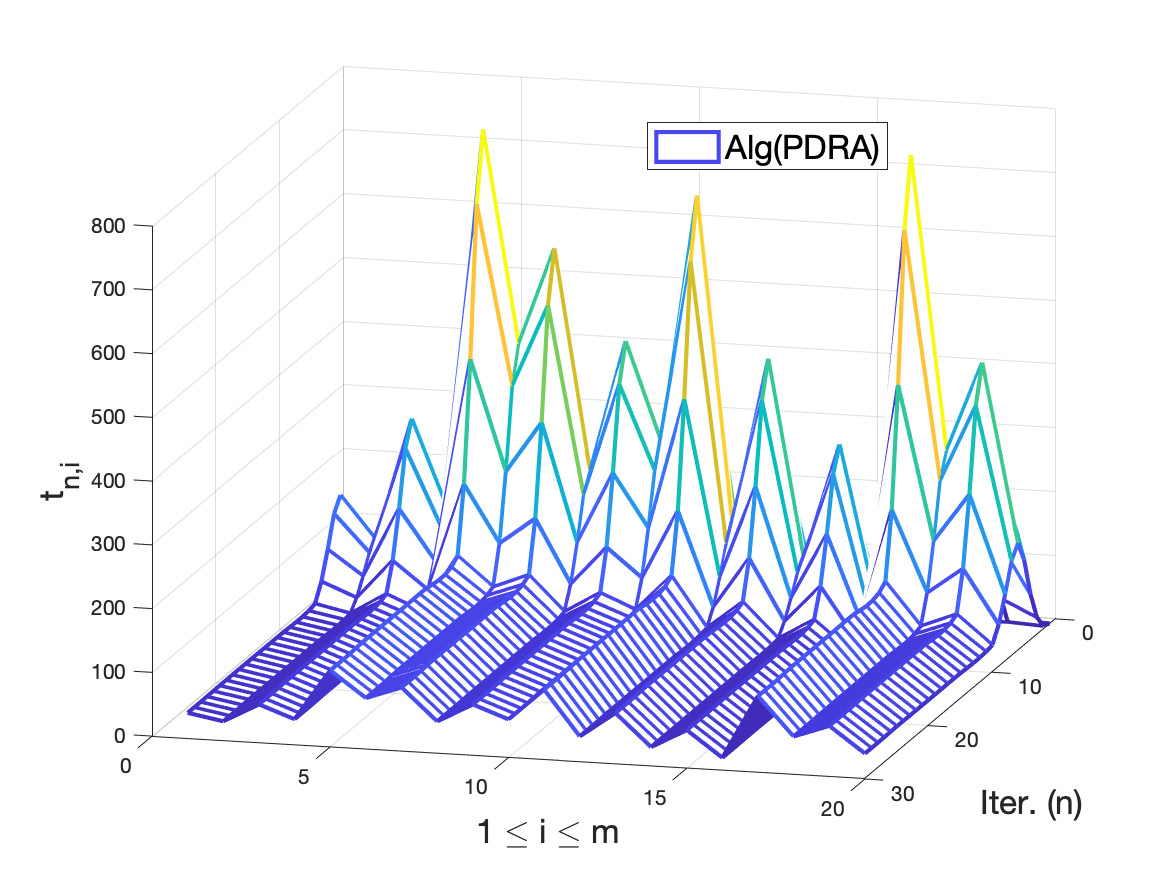}
		\caption{Alg(PDRA)}
	\end{subfigure}
	\hfill
	\begin{subfigure}[b]{0.45\textwidth}
		\includegraphics[width=\textwidth]{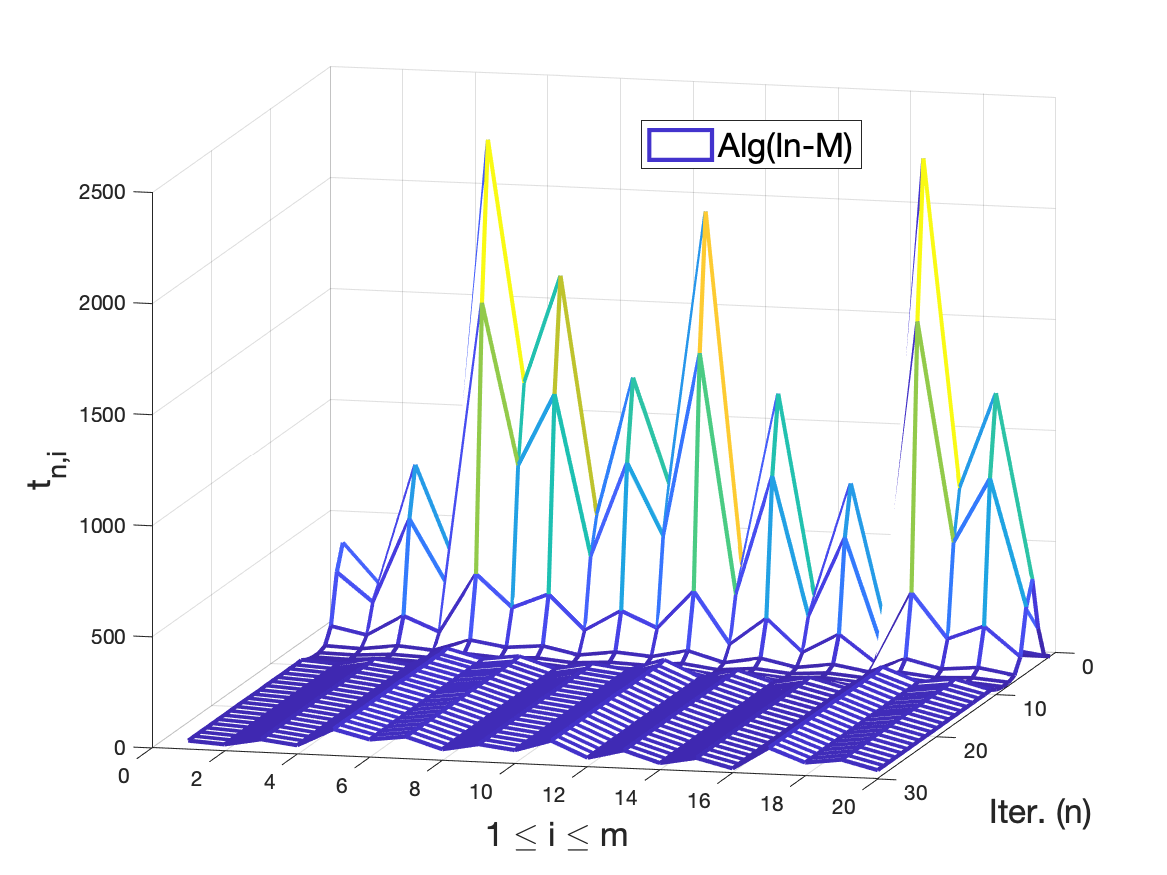}
		\caption{Alg(In-M)}
	\end{subfigure}
	\hfill
	\begin{subfigure}[b]{0.45\textwidth}
		\includegraphics[width=\textwidth]{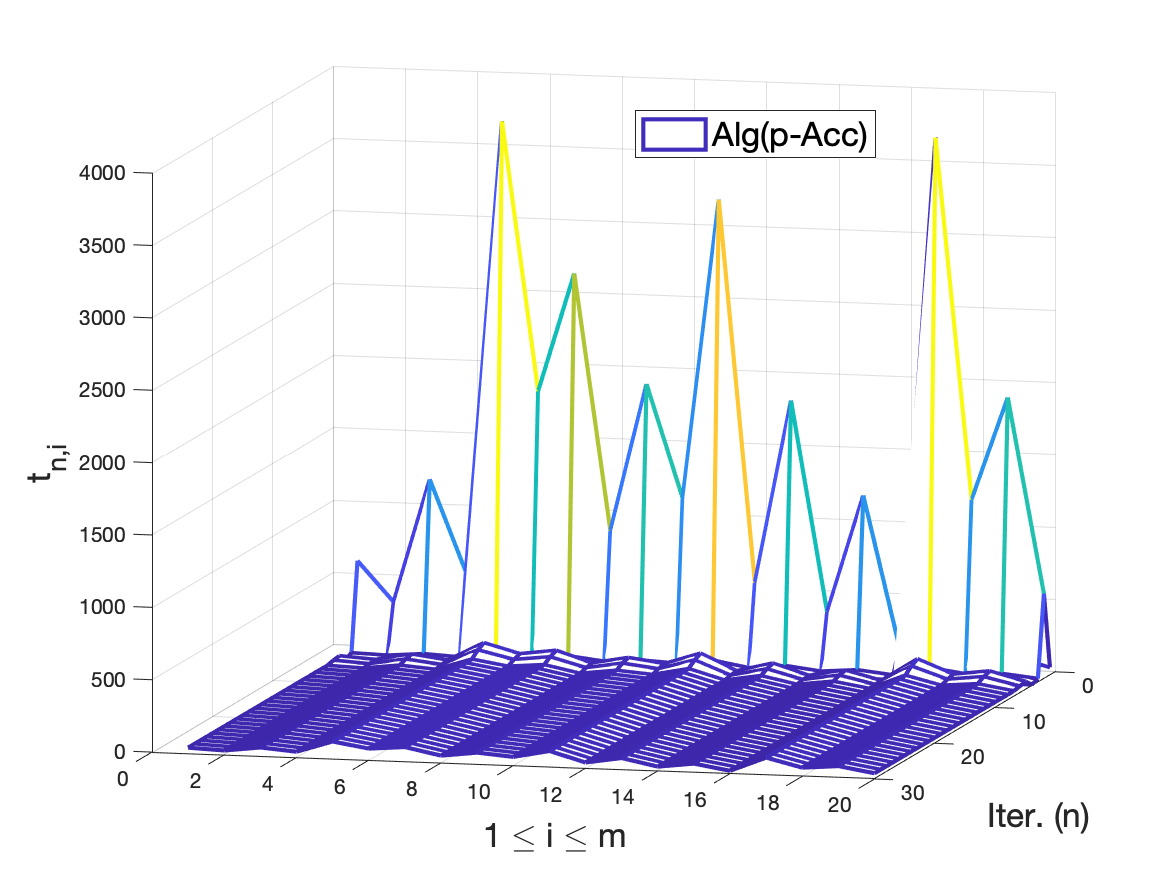}
		\caption{Alg($p$-Acc)}
	\end{subfigure}
	\caption{Graphical progression of Alg(PDRA), Alg(In-M), and Alg($p$-Acc) upto 30 iterations for Case 2 in Example \ref{GeneralExmp}.}
	\label{comparefffff}
\end{figure}

\begin{Example} \label{Ballwithballexmp}
	Consider the generalized Heron problem \eqref{HeronpbmManifold} for the following two cases: In the first case we consider balls with ball constraint in $\M=\R_{++}^2$ and in the second case balls with ball constraint in $\M=\R_{++}^3$ as follows:\\
	%where the solution of the problem is known and we analyze the behaviour Alg(PDRA), Alg(In-M) and Alg($p$-Acc) to achieve the desired solution.\\
	\textbf{Case 1:} $m=2$, $N=4$.\\
	 We take predetermined set $C$ as ball $B_r[c]$ with $c=(35,35)$  and  $r=0.4$
	and choose constraints sets to be balls $C_k=B_r[a_k]$, $k=1,2,3,4$ with the center
	$a_1=(15,15)$, $a_2=(65,65)$, $a_3 = (10,60)$  and  $a_4=(60,10)$ and $r=0.4$; see Figure \ref{ballcase2}.\\
	\textbf{Case 2:} $m=3$, $N=4$.\\
	 We take set $C$ as ball $B_r[c]$ with $c=(35,35,35)$  and $r=0.4$
	and choose constraints sets to be balls $C_k=B_r[a_k]$, $k=1,2,3,4$ with the center
	$a_1=(15,15,15)$, $a_2=(65,65,65)$, $a_3 = (10,60,10)$  and $a_4=(60,10,60)$ and $r=0.4$; see Figure \ref{ballcase3}.
\end{Example}

\begin{figure}[htbp]
	\centering
	\begin{subfigure}[b]{0.48\textwidth}
		\includegraphics[width=\textwidth]{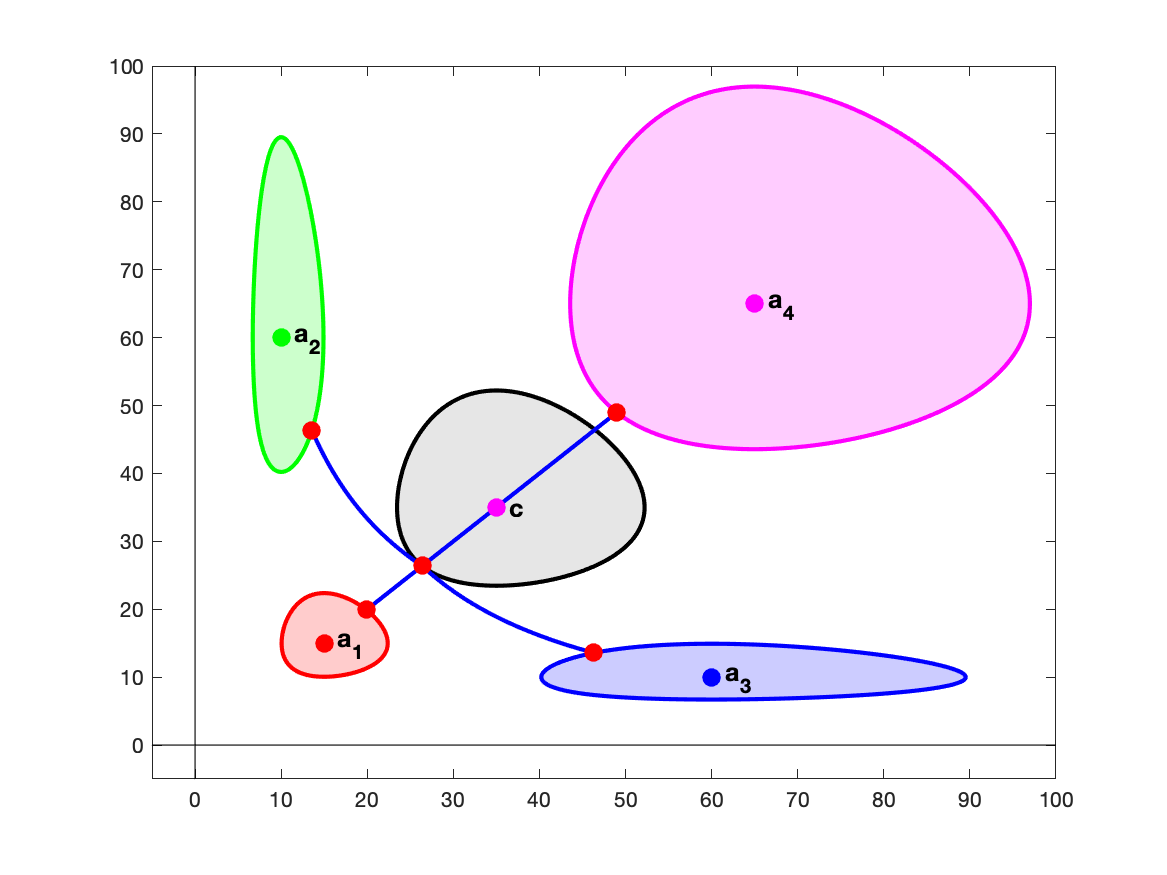}
		\caption{Case 1}
			\label{ballcase2}
	\end{subfigure}
	\hfill
	\begin{subfigure}[b]{0.48\textwidth}
		\includegraphics[width=\textwidth]{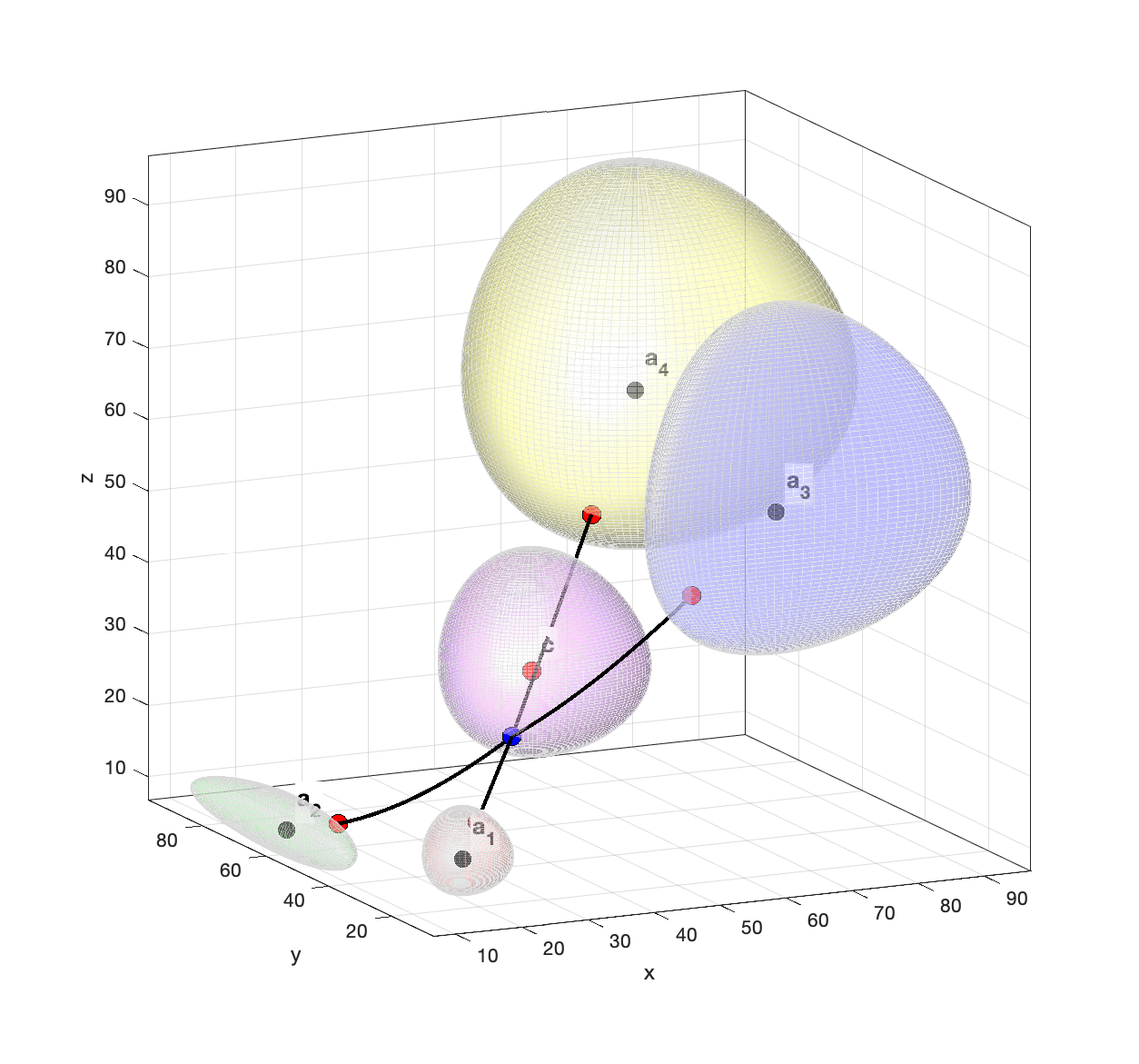}
		\caption{Case 2}
			\label{ballcase3}
	\end{subfigure}
	\caption{Generalized Heron problems for a ball with  ball constraints.}
	\label{discCons33}
\end{figure}

We take the initial point $\mathbf{x}_0=ones(1,m*(N+1))$ for all the algorithms and $\mathbf{x}_1=2*ones(1,m*(N+1))$ for Alg(In-M) for both cases (Case 1 and Case 2) in Example \ref{Ballwithballexmp}.
The numerical performance of algorithms Alg(PDRA), Alg(In-M) and Alg($p$-Acc) are depicted in Table \ref{Entab2} for both cases (Case 1 and Case 2). We plot the relationship between the number of iterations and \(Er(n)\) to have a clear visualization of the experiment; see Figure \ref{compareEn3}. From Table \ref{Entab2} and Figure \ref{compareEn3}, we see that our Algorithm Alg($p$-Acc) takes less runtime and less number of iterations to achieve the desired stopping criterion in both cases.

\begin{table}
	\begin{center}
		\begin{minipage}{\textwidth}
			\caption{Numerical performance in terms of CPU time and number of iterations for Alg(PDRA), Alg(In-M), and Alg($p$-Acc)  for Example \ref{Ballwithballexmp}}
			\label{Entab2}
			\begin{tabular*}{\textwidth}{@{\extracolsep{\fill}}lcccccc@{\extracolsep{\fill}}}
				\hline
				&&Case 1&&&Case 2&\\
				\hline%
				Algorithm & CPU(s) & iter(n) & $Er(n)$ &CPU(s) & iter(n) & $Er(n)$\\
				\hline
				Alg(PDRA) & $0.0190$ & $113$ & $8.9175*10^{-11}$&      $0.0214$ & $137$ & $9.6941*10^{-11}$\\
				Alg(In-M) & $0.0146$ & $83$ & $4.6019*10^{-11}$ & $0.0105$ & $100$ & $7.5025*10^{-11}$\\
				Alg($p$-Acc)&  $0.0074$& $34$ & $8.5375*10^{-11}$ &  $0.0087$& $46$ & $8.6438*10^{-11}$\\
				\hline
			\end{tabular*}

		\end{minipage}
	\end{center}
\end{table}

\begin{figure}[htbp]
	\centering
	\begin{subfigure}[b]{0.48\textwidth}
		\includegraphics[width=\textwidth]{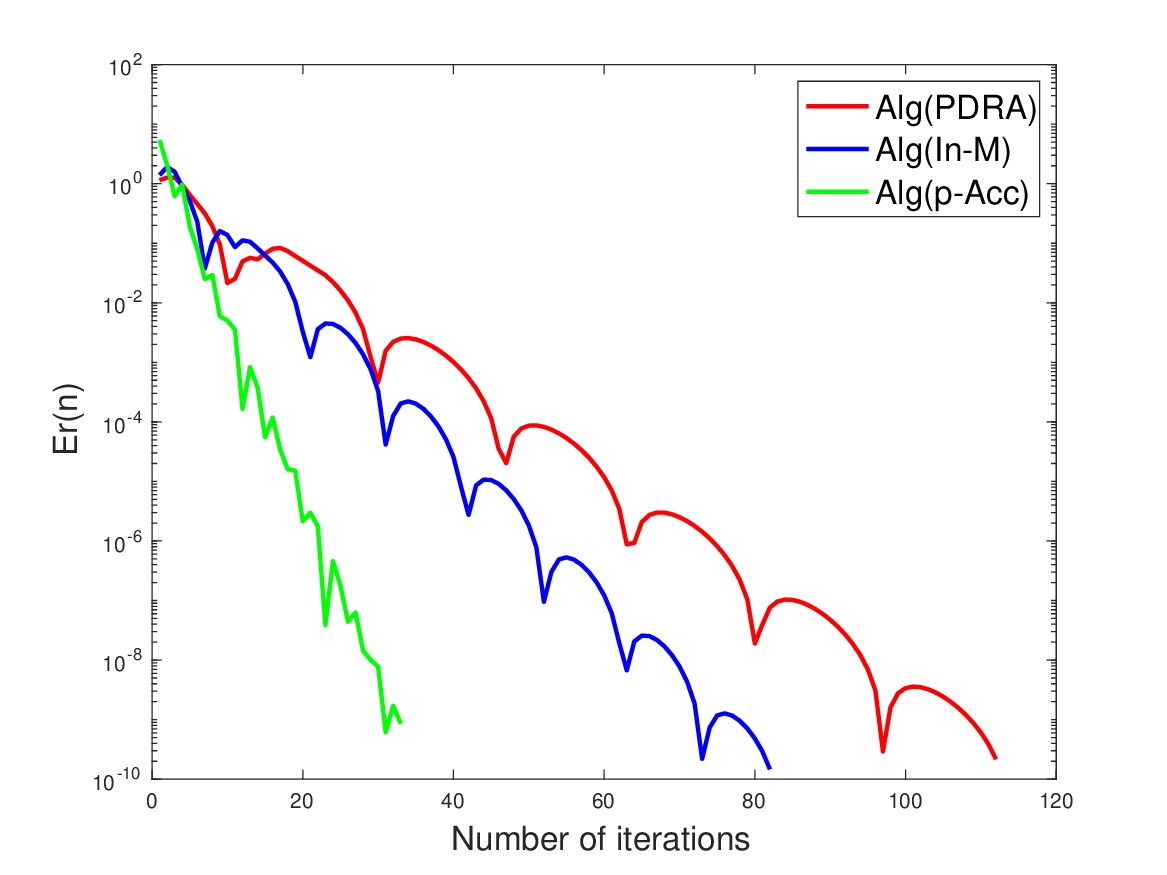}
		\caption{Case 1}
	\end{subfigure}
	\hfill
	\begin{subfigure}[b]{0.48\textwidth}
		\includegraphics[width=\textwidth]{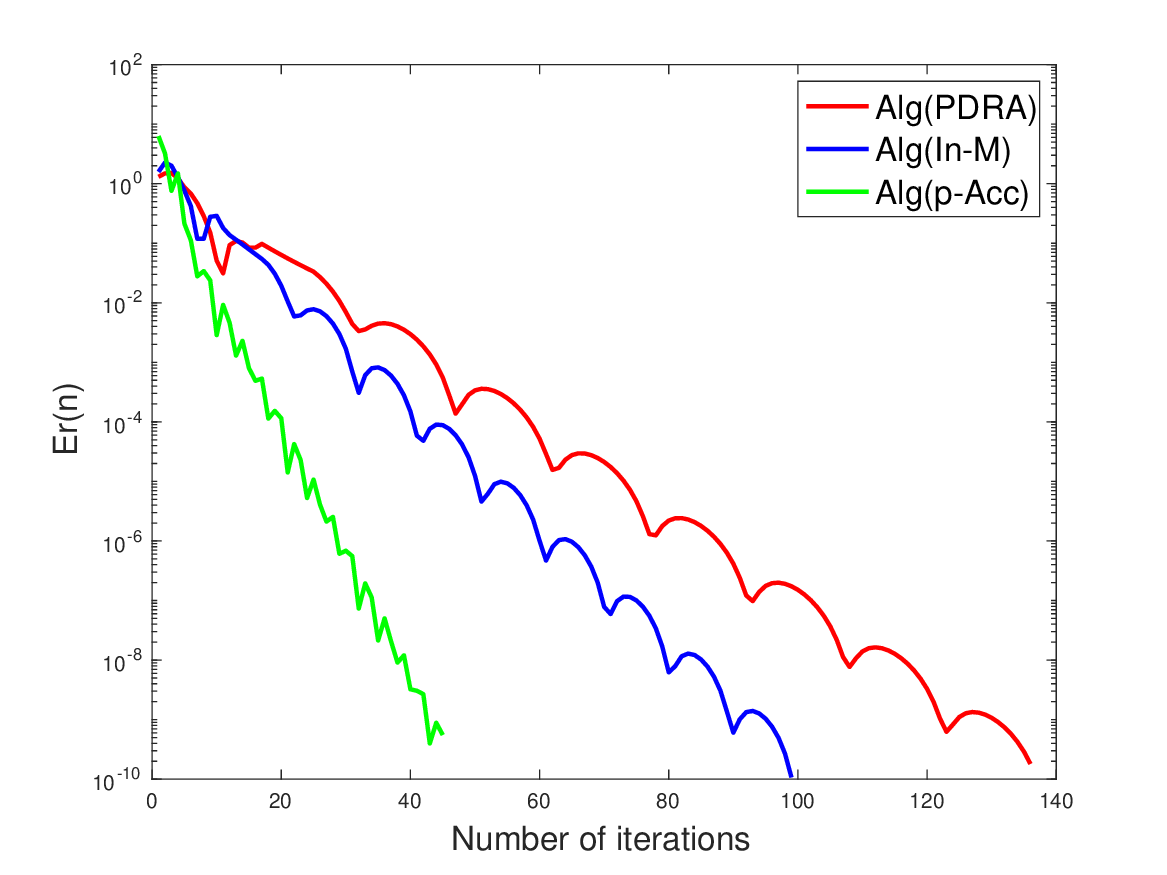}
		\caption{Case 2}
	\end{subfigure}
	\caption{Comparison of Alg(PDRA), Alg(In-M), and Alg($p$-Acc) based on the relationship between the number of iterations and $Er(n)$ for Example~\ref{Ballwithballexmp}.}
	\label{compareEn3}
\end{figure}

\section{Conclusion}\label{sec13}

In this paper, we have proposed new accelerated algorithms--the inertial Mann method and the $p$-accelerated normal S-method for finding fixed points of nonexpansive operators on Hadamard manifolds. These methods were further utilized to minimize the sum of two geodesically convex functions on Hadamard manifolds. For the minimization problem, we have derived the inertial Douglas-Rachford method and the $p$-accelerated normal S-Douglas-Rachford method, which aim to improve the convergence speed of the classical Douglas-Rachford method in the context of Hadamard manifolds. We have studied the convergence analysis of these methods and analyzed their asymptotic behavior under mild conditions on the iteration parameters.

Additionally, we introduced parallel versions of these algorithms--the inertial parallel Douglas-Rachford method and the $p$-accelerated normal S-parallel Douglas-Rachford method for finding the solution of  minimization problems involving sum of more than two functions. We applied these parallel methods to find the solution of the generalized Heron problem in the context of Hadamard manifolds. We conducted numerical experiments which showed the efficiency of the proposed methods.

\bmhead{Acknowledgements}

We thank Ronny Bergmann for his insightful discussions on the parallel Douglas-Rachford methods in the context of Hadamard manifolds, and for his suggestions regarding the numerical experiments. The authors are grateful to the anonymous referees for careful reading of the manuscript and for various suggestions that helped to improve the paper.

\section*{Declarations}

\begin{itemize}
\item Funding: Not applicable.
\item Conflict of interest/Competing interests: The authors declare that they have no competing interests.
\item Ethics approval and consent to participate: Not applicable.
\item Consent for publication: Not applicable.
\item Data availability: Not applicable.
\item Materials availability: Not applicable.
\item Code availability: The code used for the experiments is available from the corresponding author upon reasonable request.
\item Author contribution: All authors read and approved the final manuscript.
\end{itemize}

\begin{appendices}

\section{Proof of Proposition \ref{Newpr2}}\label{Ap1} 

First, we recall some basic results on Hadamard manifolds.

\begin{Lemma}[\cite{Martin1}]  \label{Properties}
	Let $\M$ be a Hadamard manifold and let $x,y,z \in \M$. Then, we have the following:
	\begin{enumerate}
		\item[(i)] Let $\varsigma_y$  be the angle at the vertex $y$. Then
		$\langle \exp_y^{-1}x , \exp_y^{-1}z \rangle = d(x,y) d (y,z)
		\cos\varsigma_{y}$.
		\item[(ii)]  $d^2(x,y) +d^2 (y,z)- 2 \langle \exp_y^{-1}x
		, \exp_y^{-1}z \rangle \leq d^2(z, x)$.
		%\item[(iii)] $d^2(x,y) \leq \langle \exp_x^{-1}z,\exp_{x}^{-1} y \rangle + \langle \exp_{y}^{-1} z,\exp_{y}^{-1} x\rangle$.
		%	\item[(iv)] $\|\exp_{x}^{-1}y\|^2 = \langle \exp_{x}^{-1}y, \exp_{x}^{-1} y \rangle =d^2(x,y)$.
	\end{enumerate}
\end{Lemma}

\begin{Lemma}   [\cite{Martin1}]  \label{L3} 
	Let $\triangle(x,y,z)$ be a geodesic
	triangle in a Hadamard manifold $\M$ and $\triangle(x^{\prime},y^{\prime },z^{\prime })$ be its comparison triangle.
	\begin{enumerate}
		\item[\textrm{(a)}] Let $\varsigma_x, \varsigma_y, \varsigma_z$ (resp. $\varsigma_x', \varsigma_y', \varsigma_z'$) be the angles of $\triangle(x,y,z)$ (resp.
		$\triangle(x^{\prime },y^{\prime },z^{\prime })$) at the vertices $x,y,z$
		(resp. $x^{\prime },y^{\prime },z^{\prime }$). Then $\varsigma_x^{\prime }\geq \varsigma_x$ , $\varsigma_y^{\prime }\geq \varsigma_y$  and
		$\varsigma_z^{\prime }\geq \varsigma_z$.
		
		\item[\textrm{(b)}] Let $w$ be a point in the geodesic joining $x$ to $y$ and $w^{\prime }$ its comparison point in the interval $[x^{\prime },y^{\prime }]$. Suppose that $d(w,x)= \|w^{\prime }-x^{\prime }\|$ and $d(w,y)=\|w^{\prime}-y^{\prime }\|$. Then $d(w,z)\ \leq \|w^{\prime }-z^{\prime }\|.$
	\end{enumerate}
\end{Lemma}

\begin{Lemma}[ \cite{KhammaAOT2022}]   \label{PTle}
	Let $\M$ be an Hadamard manifold and $x,y,z\in \M$. Then 
	$$\|\exp_x^{-1}z-P_{x,y}\exp_y^{-1}z\|\leq d(x,y).$$
\end{Lemma}
\ \\
	\textit{Proof of Proposition \ref{Newpr2}.} Let $u=\exp_x(-\theta \exp_x^{-1}y)$.	Consider the geodesic triangle $\triangle(u,x,z)$. Then, there exists the corresponding comparison triangle $\triangle(\bar{u}, \bar{x},\bar{z})\subset \R^2$ such that 
	$$d(u,x)=\|\bar{u}-\bar{x}\|, d(u,z)=\|\bar{u}-\bar{z}\|\text{ and } d(x,z)=\|\bar{x}-\bar{z}\|.$$
	Observe that 
	\begin{align}
		d^2(u,z)&=\|\bar{u}-\bar{z}\|^2=\|\bar{u}-\bar{x}+\bar{x}-\bar{z}\|^2=\|\bar{u}-\bar{x}\|^2+\|\bar{x}-\bar{z}\|^2+2\langle \bar{u}-\bar{x},\bar{x}-\bar{z}\rangle\nonumber\\
		& = \|\bar{u}-\bar{x}\|^2+\|\bar{x}-\bar{z}\|^2+2\langle \bar{u}-\bar{x},\bar{x}-\bar{z}\rangle+2\|\bar{x}-\bar{z}\|^2-2\|\bar{x}-\bar{z}\|^2 \nonumber\\
		& = d^2(u,x)+d^2(x,z)+2\langle \bar{u}-\bar{z},\bar{x}-\bar{z}\rangle-2d^2(x,z). \label{PKinreq3'}
	\end{align} 
	Suppose that the angles at the vertices $z$ and $\bar{z}$ are denoted by $\varsigma$ and $\varsigma'$ respectively. Then, by Lemma $\ref{L3}$, we have $\gamma \leq \gamma'$. From Lemma \ref{Properties}(i), we get
	\begin{align}
		\langle \bar{u}-\bar{z},\bar{x}-\bar{z}\rangle&=\|\bar{u}-\bar{z}\|\|\bar{x}-\bar{z}\|\cos(\varsigma')= d(u,z)d(x,z)\cos(\varsigma')\nonumber\\
		& \leq d(u,z)d(x,z)\cos(\varsigma)=\langle \exp_{z}^{-1}u,\exp_z^{-1}x \rangle.  \label{PKinreq4'}
	\end{align}
	From $\eqref{PKinreq3'}$ and $\eqref{PKinreq4'}$, we have 
	\begin{equation}\label{PKinreq5'}
		d^2(u,z)\leq d^2(u,x)+d^2(x,z)+2\langle  \exp_z^{-1}u, \exp_z^{-1}x\rangle-2d^2(x,z).
	\end{equation}
	Note that  $\exp_{x}^{-1}u=-\theta\exp_{x}^{-1}y$. Hence $$d(x,u)=\|\exp_{x}^{-1}u\|=\|\exp_{x}^{-1}(\exp_{x}(-\theta \exp_{x}^{-1}y))\|=\theta\|\exp_{x}^{-1}y\|=\theta d(x,y),$$ 
	where $\theta\in [0,1)$. It follows from $\eqref{PKinreq5'}$ that 
	\begin{align}\label{aboineq}
		d^2(u,z)&\leq \theta^2d^2(x,y)+d^2(x,z)+2\langle  \exp_z^{-1}u, \exp_z^{-1}x\rangle-2d^2(x,z).
	\end{align}
	Taking into consideration Remark $\ref{re101}$ and Lemma $\ref{PTle}$, from \eqref{aboineq} we obtain
	\begin{align}
		d^2(u,z)&\leq \theta^2 d^2(x,y)+d^2(x,z)-2d^2(x,z) \nonumber\\
		&\quad+2\langle  \exp_z^{-1}u-P_{z,x}\exp_{x}^{-1}u+P_{z,x}\exp_{x}^{-1}u, \exp_z^{-1}x\rangle\nonumber\\
		&=\theta^2 d^2(x,y)+d^2(x,z)-2d^2(x,z) \nonumber\\
		&\quad+2\langle  \exp_z^{-1}u-P_{z,x}\exp_{x}^{-1}u, \exp_z^{-1}x\rangle+2\langle P_{z,x}\exp_{x}^{-1}u,\exp_z^{-1}x\rangle\nonumber\\
		& \leq \theta^2 d^2(x,y)+d^2(x,z)-2d^2(x,z)\nonumber\\
		&\quad+2\|\exp_{z}^{-1}u-P_{z,x}\exp_{x}^{-1}u\|\|\exp_{z}^{-1}x\|-2\langle \exp_{x}^{-1}u,\exp_{x}^{-1}z \rangle\nonumber\\
		& \leq \theta^2 d^2(x,y)+d^2(x,z)-2d^2(x,z)+2d^2(z,x)-2\langle \exp_{x}^{-1}u,\exp_{x}^{-1}z \rangle\nonumber\\
		&=\theta^2 d^2(x,y)+d^2(x,z)-2\langle \exp_{x}^{-1}u,\exp_{x}^{-1}z \rangle. \label{PKinreq6'}
	\end{align}
	From the fact that $\exp_{x}^{-1}u=-\theta\exp_{x}^{-1}y$,  Remark $\ref{re101}$, Lemma $\ref{PTle}$ and  $\eqref{PKinreq6'}$, we have 
	\begin{align}
		&d^2(u,z)\nonumber\\
		&\leq \theta^2 d^2(x,y)+d^2(x,z)+2\theta \langle \exp_{x}^{-1}y,\exp_{x}^{-1}z \rangle \nonumber\\
		&=\theta^2 d^2(x,y)+d^2(x,z)+2\theta \langle  \exp_{x}^{-1}y-P_{x,z}\exp_{z}^{-1}y+P_{x,z}\exp_{z}^{-1}y, \exp_{x}^{-1}z\rangle \nonumber\\
		&=\theta^2 d^2(x,y)+d^2(x,z)+2\theta \langle \exp_{x}^{-1}y-P_{x,z}\exp_z^{-1}y,\exp_{x}^{-1}z \rangle\nonumber\\
		&\quad-2\theta \langle \exp_{z}^{-1}y,\exp_{z}^{-1}x \rangle\nonumber\\
		& \leq \theta^2  d^2(x,y)+d^2(x,z)+2\theta \|\exp_{x}^{-1}y-P_{x,z}\exp_{z}^{-1}y\|\|\exp_{x}^{-1}z\|\nonumber\\
		&\quad- 2\theta \langle \exp_{z}^{-1}y,\exp_{z}^{-1}x \rangle\nonumber\\
		& \leq \theta^2 d^2(x,y)+d^2(x,z)+2\theta d^2(x,z)- 2\theta \langle \exp_{z}^{-1}y,\exp_{z}^{-1}x \rangle.\label{PKinreq7'}
	\end{align}
	From Lemma \ref{Properties}(ii), we have
	\begin{equation}
		- 2\langle \exp_{z}^{-1}y,\exp_{z}^{-1}x \rangle \leq d^2(x,y)-d^2(y,z)-d^2(x,z).\label{PKinreq8'}
	\end{equation}
	By combining $\eqref{PKinreq7'}$ and $\eqref{PKinreq8'}$, we get
	\begin{align}
		d^2(u,z) &\leq \theta^2 d^2(x,y)+d^2(x,z)+2\theta d^2(x,z)+\theta d^2(x,y)-\theta d^2(y,z)-\theta d^2(x,z)\nonumber\\
		&=(1+\theta)d^2(x,z)-\theta d^2(y,z)+\theta(1+\theta) d^2(x,y).\nonumber
	\end{align}

\section{Proof of Proposition \ref{Newpr3}} \label{Ap3}

For the proof of Proposition~\ref{Newpr3}, we first recall \cite[Proposition 4.1]{Leon2013} for sectional curvature $k=0$.

\begin{Proposition}[\cite{Leon2013}]   \label{zeroProp}
	Let $\M$ be a Hadamard manifold with zero sectional curvature and let
	$x,y,x',y',a,b \in \M$ satisfy
	\(
	x'=\exp_a\bigl(-\exp_a^{-1}x\bigr)
	\quad \text{and} \quad
	y'=\exp_b\bigl(-\exp_b^{-1}y\bigr).
	\)
	Suppose that
	$d(a,b)\leq d\bigl(\gamma(x,a;t),\gamma(y,b;t)\bigr)$
	for all  $t \in [0,1]$.
%\begin{equation} \label{Zeroeq1}
%	d(a,b)\leq d\bigl(\gamma(x,a;t),\gamma(y,b;t)\bigr)
%	\quad \text{for all } t \in [0,1].
%\end{equation}
	Then
	\(
	d(x',y')\leq d(x,y).
	\)
\end{Proposition}
\ \\
\textit{Proof of Proposition \ref{Newpr3}.} By the definition of the reflection, we have
	\[
	R_{\lambda f}(x)
	= \exp_{\prox_{\lambda f}(x)}\!\left(-\exp_{\prox_{\lambda f}(x)}^{-1}(x)\right),
	\quad
	R_{\lambda f}(y)
	= \exp_{\prox_{\lambda f}(y)}\!\left(-\exp_{\prox_{\lambda f}(y)}^{-1}(y)\right).
	\]
	Since $\prox_{\lambda f}$ is a firmly nonexpansive operator, it follows from \eqref{FNEeq} that
\[d(\operatorname{Prox}_{\!\lambda f}(x),\operatorname{Prox}_{\!\lambda f}(y)) \leq d(\gamma(x,\operatorname{Prox}_{\!\lambda f}(x);t),\gamma(y,\operatorname{Prox}_{\!\lambda f}(y);t)) \text{ for all } t \in [0,1].\]
	Therefore, by Proposition~\ref{zeroProp}, we obtain
	\(
	d\bigl(R_{\lambda f}(x), R_{\lambda f}(y)\bigr)
	\leq d(x,y).
	\)

\end{appendices}

\bibliography{reference}% common bib file

@book {Docarmo,
	AUTHOR = {do Carmo, Manfredo Perdig{a}o},
	TITLE = {Riemannian geometry},
	SERIES = {Mathematics: Theory \& Applications},
	EDITION = {Portuguese},
	PUBLISHER = {Birkh\"{a}user Boston, Inc.},
	ADDRESS={Boston, MA},
	YEAR = {1992},
	PAGES = {xiv+300},
	ISBN = {0-8176-3490-8},
	MRCLASS = {53-01},
	MRNUMBER = {1138207},
	MRREVIEWER = {Bang-yen\ Chen},
	DOI = {10.1007/978-1-4757-2201-7},
}

@article {Martin1,
	AUTHOR = {Li, Chong and L\'{o}pez, Genaro and Mart\'{\i}n-M\'{a}rquez,
	Victoria},
	TITLE = {Iterative algorithms for nonexpansive mappings on {H}adamard
	manifolds},
	JOURNAL = {Taiwanese J. Math.},
	FJOURNAL = {Taiwanese Journal of Mathematics},
	VOLUME = {14},
	YEAR = {2010},
	NUMBER = {2},
	PAGES = {541--559},
	ISSN = {1027-5487,2224-6851},
	MRCLASS = {47H10 (47H09 47J25 53C22)},
	MRNUMBER = {2655786},
	MRREVIEWER = {Simeon\ Reich},
	DOI = {10.11650/twjm/1500405806},
}

@article {Bergmann2016,
	AUTHOR = {Bergmann, Ronny and Persch, Johannes and Steidl, Gabriele},
	TITLE = {A parallel {D}ouglas-{R}achford algorithm for minimizing
	{ROF}-like functionals on images with values in symmetric
	{H}adamard manifolds},
	JOURNAL = {SIAM J. Imaging Sci.},
	FJOURNAL = {SIAM Journal on Imaging Sciences},
	VOLUME = {9},
	YEAR = {2016},
	NUMBER = {3},
	PAGES = {901--937},
	ISSN = {1936-4954},
	MRCLASS = {94A08 (49Q99 65K10 68U10)},
	MRNUMBER = {3519549},
	MRREVIEWER = {Bin\ Han},
	DOI = {10.1137/15M1052858},
}

@article {Martin2,
	AUTHOR = {Li, Chong and L\'{o}pez, Genaro and Mart\'{\i}n-M\'{a}rquez,
	Victoria},
	TITLE = {Monotone vector fields and the proximal point algorithm on
	{H}adamard manifolds},
	JOURNAL = {J. Lond. Math. Soc. (2)},
	FJOURNAL = {Journal of the London Mathematical Society. Second Series},
	VOLUME = {79},
	YEAR = {2009},
	NUMBER = {3},
	PAGES = {663--683},
	ISSN = {0024-6107,1469-7750},
	MRCLASS = {47J25 (47H05 49J40 53C20 90C32)},
	MRNUMBER = {2506692},
	MRREVIEWER = {S\'{a}ndor\ Z.\ N\'{e}meth},
	DOI = {10.1112/jlms/jdn087},
}

@article {Martin3,
	AUTHOR = {Li, Chong and L\'{o}pez, Genaro and Mart\'{\i}n-M\'{a}rquez,
	Victoria and Wang, Jin-Hua},
	TITLE = {Resolvents of set-valued monotone vector fields in {H}adamard
	manifolds},
	JOURNAL = {Set-Valued Var. Anal.},
	FJOURNAL = {Set-Valued and Variational Analysis. Theory and Applications},
	VOLUME = {19},
	YEAR = {2011},
	NUMBER = {3},
	PAGES = {361--383},
	ISSN = {1877-0533,1877-0541},
	MRCLASS = {47H04 (45H05 47A10 53C20)},
	MRNUMBER = {2824431},
	MRREVIEWER = {S\'{a}ndor\ Z.\ N\'{e}meth},
	DOI = {10.1007/s11228-010-0169-1},
}

@book {Sakai,
	AUTHOR = {Sakai, Takashi},
	TITLE = {Riemannian geometry},
	SERIES = {Translations of Mathematical Monographs},
	VOLUME = {149},
	NOTE = {Translated from the 1992 Japanese original by the author},
	PUBLISHER = {American Mathematical Society},
	ADDRESS={Providence, RI},
	YEAR = {1996},
	PAGES = {xiv+358},
	ISBN = {0-8218-0284-4},
	MRCLASS = {53-01 (53-02)},
	MRNUMBER = {1390760},
	MRREVIEWER = {Conrad\ Plaut},
	DOI = {10.1090/mmono/149},
}

@book {Udriste,
	AUTHOR = {Udri\c{s}te, Constantin},
	TITLE = {Convex functions and optimization methods on {R}iemannian
	manifolds},
	SERIES = {Mathematics and its Applications},
	VOLUME = {297},
	PUBLISHER = {Kluwer Academic Publishers Group},
	ADDRESS={Dordrecht},
	YEAR = {1994},
	PAGES = {xviii+348},
	ISBN = {0-7923-3002-1},
	MRCLASS = {49K27 (52A41 58C05 65K10 90C48)},
	MRNUMBER = {1326607},
	MRREVIEWER = {Robert\ L.\ Foote},
	DOI = {10.1007/978-94-015-8390-9},
}

@article {Fereira2005,
	AUTHOR = {Ferreira, O. P. and P\'{e}rez, L. R. Lucambio and N\'{e}meth,
	S. Z.},
	TITLE = {Singularities of monotone vector fields and an
	extragradient-type algorithm},
	JOURNAL = {J. Global Optim.},
	FJOURNAL = {Journal of Global Optimization. An International Journal
	Dealing with Theoretical and Computational Aspects of Seeking
	Global Optima and Their Applications in Science, Management
	and Engineering},
	VOLUME = {31},
	YEAR = {2005},
	NUMBER = {1},
	PAGES = {133--151},
	ISSN = {0925-5001,1573-2916},
	MRCLASS = {58C05 (52A55 90C26)},
	MRNUMBER = {2141129},
	MRREVIEWER = {P.\ P.\ Zabre\u{\i}ko},
	DOI = {10.1007/s10898-003-3780-y},
}

@article {SahuNFAO,
	AUTHOR = {Sahu, D. R. and Ansari, Q. H. and Yao, J. C.},
	TITLE = {Convergence of inexact {M}ann iterations generated by nearly
	nonexpansive sequences and applications},
	JOURNAL = {Numer. Funct. Anal. Optim.},
	FJOURNAL = {Numerical Functional Analysis and Optimization. An
	International Journal},
	VOLUME = {37},
	YEAR = {2016},
	NUMBER = {10},
	PAGES = {1312--1338},
	ISSN = {0163-0563,1532-2467},
	MRCLASS = {47H09 (47H10)},
	MRNUMBER = {3553009},
	DOI = {10.1080/01630563.2016.1206566},
}

@article {Ferreira2002,
	AUTHOR = {Ferreira, O. P. and Oliveira, P. R.},
	TITLE = {Proximal point algorithm on {R}iemannian manifolds},
	JOURNAL = {Optimization},
	FJOURNAL = {Optimization. A Journal of Mathematical Programming and
	Operations Research},
	VOLUME = {51},
	YEAR = {2002},
	NUMBER = {2},
	PAGES = {257--270},
	ISSN = {0233-1934,1029-4945},
	MRCLASS = {49M30 (90C26 90C48)},
	MRNUMBER = {1928039},
	MRREVIEWER = {R.\ Tichatschke},
	DOI = {10.1080/02331930290019413},
}

@book {BacakNAA2014,
	AUTHOR = {Ba\v{c}\'{a}k, Miroslav},
	TITLE = {Convex analysis and optimization in {H}adamard spaces},
	SERIES = {De Gruyter Series in Nonlinear Analysis and Applications},
	VOLUME = {22},
	PUBLISHER = {De Gruyter},
	ADDRESS = {Berlin},
	YEAR = {2014},
	PAGES = {viii--185},
	ISBN = {978-3-11-036103-2; 978-3-11-036162-9},
	MRCLASS = {49-02 (49K27 60B99 60J10 90C25 90C48 92D15)},
	MRNUMBER = {3241330},
	MRREVIEWER = {Simeon\ Reich},
	DOI = {10.1515/9783110361629},
}

@article {KhammaAOT2022,
	AUTHOR = {Khammahawong, Konrawut and Chaipunya, Parin and Kumam, Poom},
	TITLE = {Iterative algorithms for monotone variational inequality and
	fixed point problems on {H}adamard manifolds},
	JOURNAL = {Adv. Oper. Theory},
	FJOURNAL = {Advances in Operator Theory},
	VOLUME = {7},
	YEAR = {2022},
	NUMBER = {},
	PAGES = {Paper No. 43},
	ISSN = {2662-2009,2538-225X},
	MRCLASS = {47J25 (47H05 58A05 58C30)},
	MRNUMBER = {4455181},
	DOI = {10.1007/s43036-022-00207-z},
}

@article {Alvarez2004,
	AUTHOR = {Alvarez, Felipe},
	TITLE = {Weak convergence of a relaxed and inertial hybrid
	projection-proximal point algorithm for maximal monotone
	operators in {H}ilbert space},
	JOURNAL = {SIAM J. Optim.},
	FJOURNAL = {SIAM Journal on Optimization},
	VOLUME = {14},
	YEAR = {2003},
	NUMBER = {3},
	PAGES = {773--782},
	ISSN = {1052-6234,1095-7189},
	MRCLASS = {49J40 (47H05 47J25 65K10 90C25 90C31 90C48)},
	MRNUMBER = {2085942},
	MRREVIEWER = {M.\ Beatrice\ Lignola},
	DOI = {10.1137/S1052623403427859},
}

@article {Leon2013,
	AUTHOR = {Fern\'{a}ndez-Le\'{o}n, Aurora and Nicolae, Adriana},
	TITLE = {Averaged alternating reflections in geodesic spaces},
	JOURNAL = {J. Math. Anal. Appl.},
	FJOURNAL = {Journal of Mathematical Analysis and Applications},
	VOLUME = {402},
	YEAR = {2013},
	NUMBER = {2},
	PAGES = {558--566},
	ISSN = {0022-247X,1096-0813},
	MRCLASS = {47J25 (47H10)},
	MRNUMBER = {3029170},
	DOI = {10.1016/j.jmaa.2013.01.060},
}

@article {Acedo2007,
	AUTHOR = {Acedo, Genaro Lopez and Xu, Hong Kun},
	TITLE = {Iterative methods for strict pseudo-contractions in {H}ilbert
	spaces},
	JOURNAL = {Nonlinear Anal.},
	FJOURNAL = {Nonlinear Analysis. Theory, Methods \& Applications. An
	International Multidisciplinary Journal},
	VOLUME = {67},
	YEAR = {2007},
	NUMBER = {7},
	PAGES = {2258--2271},
	ISSN = {0362-546X,1873-5215},
	MRCLASS = {47H10 (47H09 47J25)},
	MRNUMBER = {2331876},
	MRREVIEWER = {Ramendra\ Krishna\ Bose},
	DOI = {10.1016/j.na.2006.08.036},
}

@article{Alvarez2001,
	AUTHOR = {Alvarez, Felipe and Attouch, Hedy},
	TITLE = {An inertial proximal method for maximal monotone operators via discretization of a nonlinear oscillator with damping},
	NOTE = {Wellposedness in optimization and related topics (Gargnano,
	1999)},
	JOURNAL = {Set-Valued Anal.},
	FJOURNAL = {Set-Valued Analysis. An International Journal Devoted to the
	Theory of Multifunctions and its Applications},
	VOLUME = {9},
	YEAR = {2001},
	NUMBER = {1-2},
	PAGES = {3--11},
	ISSN = {0927-6947,1572-932X},
	MRCLASS = {65K10 (34A60 34G25 47N10 49M25)},
	MRNUMBER = {1845931},
	DOI = {10.1023/A:1011253113155},
}

@article {Bauschke1996,
	AUTHOR = {Bauschke, Heinz H. and Borwein, Jonathan M.},
	TITLE = {On projection algorithms for solving convex feasibility
	problems},
	JOURNAL = {SIAM Rev.},
	FJOURNAL = {SIAM Review. A Publication of the Society for Industrial and
	Applied Mathematics},
	VOLUME = {38},
	YEAR = {1996},
	NUMBER = {3},
	PAGES = {367--426},
	ISSN = {0036-1445},
	MRCLASS = {90C25 (47N10 49M45 65J10)},
	MRNUMBER = {1409591},
	MRREVIEWER = {F.\ Deutsch},
	DOI = {10.1137/S0036144593251710},
}

@article {Bot2016,
	AUTHOR = {Bo\c{t}, Radu Ioan and Csetnek, Ern\"{o} Robert},
	TITLE = {An inertial alternating direction method of multipliers},
	JOURNAL = {Minimax Theory Appl.},
	FJOURNAL = {Minimax Theory and its Applications},
	VOLUME = {1},
	YEAR = {2016},
	NUMBER = {1},
	PAGES = {29--49},
	ISSN = {2199-1413,2199-1421},
	MRCLASS = {47H05 (65K05 90C25)},
	MRNUMBER = {3477895},
}

@article {Bot2015,
	AUTHOR = {Bo\c{t}, Radu Ioan and Csetnek, Ern\"{o} Robert and Hendrich,
	Christopher},
	TITLE = {Inertial {D}ouglas-{R}achford splitting for monotone inclusion
	problems},
	JOURNAL = {Appl. Math. Comput.},
	FJOURNAL = {Applied Mathematics and Computation},
	VOLUME = {256},
	YEAR = {2015},
	PAGES = {472--487},
	ISSN = {0096-3003,1873-5649},
	MRCLASS = {47J25 (65J15 90C25)},
	MRNUMBER = {3316085},
	DOI = {10.1016/j.amc.2015.01.017},
}

@article {Byrne2004,
	AUTHOR = {Byrne, Charles},
	TITLE = {A unified treatment of some iterative algorithms in signal
	processing and image reconstruction},
	JOURNAL = {Inverse Problems},
	FJOURNAL = {Inverse Problems. An International Journal on the Theory and
	Practice of Inverse Problems, Inverse Methods and Computerized
	Inversion of Data},
	VOLUME = {20},
	YEAR = {2004},
	NUMBER = {1},
	PAGES = {103--120},
	ISSN = {0266-5611,1361-6420},
	MRCLASS = {47H09 (47H10 47J25 65J15 94A08 94A12)},
	MRNUMBER = {2044608},
	MRREVIEWER = {Gilbert\ Crombez},
	DOI = {10.1088/0266-5611/20/1/006},
}

@article {Chen2014,
	AUTHOR = {Chen, Caihua and Ma, Shiqian and Yang, Junfeng},
	TITLE = {A general inertial proximal point algorithm for mixed
	variational inequality problem},
	JOURNAL = {SIAM J. Optim.},
	FJOURNAL = {SIAM Journal on Optimization},
	VOLUME = {25},
	YEAR = {2015},
	NUMBER = {4},
	PAGES = {2120--2142},
	ISSN = {1052-6234,1095-7189},
	MRCLASS = {65K15 (49J40 90C25 90C33)},
	MRNUMBER = {3413597},
	MRREVIEWER = {Boualem\ Alleche},
	DOI = {10.1137/140980910},
}

@article {Chidume2001,
	AUTHOR = {Chidume, C. E. and Mutangadura, S. A.},
	TITLE = {An example of the {M}ann iteration method for {L}ipschitz
	pseudocontractions},
	JOURNAL = {Proc. Amer. Math. Soc.},
	FJOURNAL = {Proceedings of the American Mathematical Society},
	VOLUME = {129},
	YEAR = {2001},
	NUMBER = {8},
	PAGES = {2359--2363},
	ISSN = {0002-9939,1088-6826},
	MRCLASS = {47H10 (47J25 65J15)},
	MRNUMBER = {1823919},
	MRREVIEWER = {Hong\ Kun\ Xu},
	DOI = {10.1090/S0002-9939-01-06009-9},
}

@article {Dotson1970,
	AUTHOR = {Dotson, Jr., W. G.},
	TITLE = {On the {M}ann iterative process},
	JOURNAL = {Trans. Amer. Math. Soc.},
	FJOURNAL = {Transactions of the American Mathematical Society},
	VOLUME = {149},
	YEAR = {1970},
	PAGES = {65--73},
	ISSN = {0002-9947,1088-6850},
	MRCLASS = {47.80 (65.00)},
	MRNUMBER = {257828},
	MRREVIEWER = {M.\ Z.\ Nashed},
	DOI = {10.2307/1995659},
}

@article {Ishikawa1974,
	AUTHOR = {Ishikawa, Shiro},
	TITLE = {Fixed points by a new iteration method},
	JOURNAL = {Proc. Amer. Math. Soc.},
	FJOURNAL = {Proceedings of the American Mathematical Society},
	VOLUME = {44},
	YEAR = {1974},
	PAGES = {147--150},
	ISSN = {0002-9939,1088-6826},
	MRCLASS = {47H10},
	MRNUMBER = {336469},
	MRREVIEWER = {Gerald\ W.\ Johnson},
	DOI = {10.2307/2039245},
}

@article {Eckstein1992,
	AUTHOR = {Eckstein, Jonathan and Bertsekas, Dimitri P.},
	TITLE = {On the {D}ouglas-{R}achford splitting method and the proximal
	point algorithm for maximal monotone operators},
	JOURNAL = {Math. Programming},
	FJOURNAL = {Mathematical Programming},
	VOLUME = {55},
	YEAR = {1992},
	NUMBER = {3},
	PAGES = {293--318},
	ISSN = {0025-5610,1436-4646},
	MRCLASS = {90C25 (65K05)},
	MRNUMBER = {1168183},
	MRREVIEWER = {F.\ Giannessi},
	DOI = {10.1007/BF01581204},
}

@article {Combettes2008Inv,
	AUTHOR = {Combettes, Patrick L. and Pesquet, Jean-Christophe},
	TITLE = {A proximal decomposition method for solving convex variational
	inverse problems},
	JOURNAL = {Inverse Problems},
	FJOURNAL = {Inverse Problems. An International Journal on the Theory and
	Practice of Inverse Problems, Inverse Methods and Computerized
	Inversion of Data},
	VOLUME = {24},
	YEAR = {2008},
	NUMBER = {6},
	PAGES = {065014, 27},
	ISSN = {0266-5611,1361-6420},
	MRCLASS = {90C25 (65K10 90C48 94A08)},
	MRNUMBER = {2456961},
	MRREVIEWER = {Franklin\ A.\ Mendivil},
	DOI = {10.1088/0266-5611/24/6/065014},
}

@article {Lions1967,
	AUTHOR = {Lions, J.-L. and Stampacchia, G.},
	TITLE = {Variational inequalities},
	JOURNAL = {Comm. Pure Appl. Math.},
	FJOURNAL = {Communications on Pure and Applied Mathematics},
	VOLUME = {20},
	YEAR = {1967},
	PAGES = {493--519},
	ISSN = {0010-3640,1097-0312},
	MRCLASS = {47.90 (35.00)},
	MRNUMBER = {216344},
	MRREVIEWER = {T.\ R.\ Jenkins},
	DOI = {10.1002/cpa.3160200302},
}

@article {Lorenz2015,
	AUTHOR = {Lorenz, Dirk A. and Pock, Thomas},
	TITLE = {An inertial forward-backward algorithm for monotone
	inclusions},
	JOURNAL = {J. Math. Imaging Vision},
	FJOURNAL = {Journal of Mathematical Imaging and Vision},
	VOLUME = {51},
	YEAR = {2015},
	NUMBER = {2},
	PAGES = {311--325},
	ISSN = {0924-9907,1573-7683},
	MRCLASS = {94A08},
	MRNUMBER = {3314536},
	DOI = {10.1007/s10851-014-0523-2},
}

@article {Mainge2008,
	AUTHOR = {Maing\'{e}, Paul-Emile},
	TITLE = {Convergence theorems for inertial {KM}-type algorithms},
	JOURNAL = {J. Comput. Appl. Math.},
	FJOURNAL = {Journal of Computational and Applied Mathematics},
	VOLUME = {219},
	YEAR = {2008},
	NUMBER = {1},
	PAGES = {223--236},
	ISSN = {0377-0427,1879-1778},
	MRCLASS = {47H10 (47J25 65J15)},
	MRNUMBER = {2437708},
	MRREVIEWER = {Alexander\ J.\ Zaslavski},
	DOI = {10.1016/j.cam.2007.07.021},
}

@article {Mann1953,
	AUTHOR = {Mann, W. Robert},
	TITLE = {Mean value methods in iteration},
	JOURNAL = {Proc. Amer. Math. Soc.},
	FJOURNAL = {Proceedings of the American Mathematical Society},
	VOLUME = {4},
	YEAR = {1953},
	PAGES = {506--510},
	ISSN = {0002-9939,1088-6826},
	MRCLASS = {46.0X},
	MRNUMBER = {54846},
	MRREVIEWER = {E.\ H.\ Rothe},
	DOI = {10.2307/2032162},
}

@book {Mercier1980,
	AUTHOR = {Mercier, B.},
	TITLE = {In\'{e}quations variationnelles de la m\'{e}canique},
	SERIES = {Publications Math\'{e}matiques d'Orsay 80 [Mathematical
	Publications of Orsay 80]},
	VOLUME = {1},
	PUBLISHER = {Universit\'{e} de Paris-Sud, D\'{e}partement de
	Math\'{e}matiques},
	ADDRESS={Orsay},
	YEAR = {1980},
	PAGES = {98},
	MRCLASS = {49A29 (35Q20 65P05 73E05)},
	MRNUMBER = {570823},
	MRREVIEWER = {N.\ Gass},
}

@article {Bauschke2002,
	AUTHOR = {Bauschke, Heinz H. and Combettes, Patrick L. and Luke, D.
	Russell},
	TITLE = {Phase retrieval, error reduction algorithm, and {F}ienup
	variants: a view from convex optimization},
	JOURNAL = {J. Opt. Soc. Amer. A},
	FJOURNAL = {Journal of the Optical Society of America A. Optics, Image
	Science, and Vision},
	VOLUME = {19},
	YEAR = {2002},
	NUMBER = {7},
	PAGES = {1334--1345},
	ISSN = {1084-7529,1520-8532},
	MRCLASS = {94A12},
	MRNUMBER = {1914365},
	MRREVIEWER = {Michael\ M.\ Dediu},
	DOI = {10.1364/JOSAA.19.001334},
}

@book {Borwein2011,
	TITLE = {Fixed-point algorithms for inverse problems in science and
	engineering},
	SERIES = {Springer Optimization and Its Applications},
	VOLUME = {49},
	EDITOR = {Bauschke, Heinz H. and Burachik, Regina S. and Combettes,
	Patrick L. and Elser, Veit and Luke, D. Russell and Wolkowicz,
	Henry},
	PUBLISHER = {Springer},
	ADDRESS = {New York},
	YEAR = {2011},
	PAGES = {xii--402},
	ISBN = {978-1-4419-9568-1},
	MRCLASS = {49-06 (47-06 49N45 65-06 90-06)},
	MRNUMBER = {2858828},
	DOI = {10.1007/978-1-4419-9569-8},
}

@article {Ochs2014,
	AUTHOR = {Ochs, Peter and Chen, Yunjin and Brox, Thomas and Pock,
	Thomas},
	TITLE = {i{P}iano: inertial proximal algorithm for nonconvex
	optimization},
	JOURNAL = {SIAM J. Imaging Sci.},
	FJOURNAL = {SIAM Journal on Imaging Sciences},
	VOLUME = {7},
	YEAR = {2014},
	NUMBER = {2},
	PAGES = {1388--1419},
	ISSN = {1936-4954},
	MRCLASS = {94A08},
	MRNUMBER = {3218822},
	DOI = {10.1137/130942954},
}

@article {Polyak1964,
	AUTHOR = {Poljak, B. T.},
	TITLE = {Some methods of speeding up the convergence of iterative
	methods},
	JOURNAL = {\v{Z}. Vy\v{c}isl. Mat i Mat. Fiz.},
	FJOURNAL = {\v{Z}urnal Vy\v{c}islitel\cprime no\u{\i} Matematiki i
	Matemati\v{c}esko\u{\i} Fiziki},
	VOLUME = {4},
	YEAR = {1964},
	PAGES = {791--803},
	ISSN = {0044-4669},
	MRCLASS = {65.10},
	MRNUMBER = {169403},
	MRREVIEWER = {Walter\ Gautschi},
}

@article {Sahu2011,
	AUTHOR = {Sahu, D. R.},
	TITLE = {Applications of the {S}-iteration process to constrained
	minimization problems and split feasibility problems},
	JOURNAL = {Fixed Point Theory},
	FJOURNAL = {Fixed Point Theory. An International Journal on Fixed Point
	Theory, Computation and Applications},
	VOLUME = {12},
	YEAR = {2011},
	NUMBER = {1},
	PAGES = {187--204},
	ISSN = {1583-5022,2066-9208},
	MRCLASS = {47J25 (47H10 47J20)},
	MRNUMBER = {2797080},
}

@article {Takahashi2008,
	AUTHOR = {Takahashi, Wataru and Takeuchi, Yukio and Kubota, Rieko},
	TITLE = {Strong convergence theorems by hybrid methods for families of
	nonexpansive mappings in {H}ilbert spaces},
	JOURNAL = {J. Math. Anal. Appl.},
	FJOURNAL = {Journal of Mathematical Analysis and Applications},
	VOLUME = {341},
	YEAR = {2008},
	NUMBER = {1},
	PAGES = {276--286},
	ISSN = {0022-247X,1096-0813},
	MRCLASS = {47H10 (47H09 47J25)},
	MRNUMBER = {2394082},
	MRREVIEWER = {Chika\ Moore},
	DOI = {10.1016/j.jmaa.2007.09.062},
}

@article {Jost1995,
	AUTHOR = {Jost, J\"{u}rgen},
	TITLE = {Convex functionals and generalized harmonic maps into spaces
	of nonpositive curvature},
	JOURNAL = {Comment. Math. Helv.},
	FJOURNAL = {Commentarii Mathematici Helvetici},
	VOLUME = {70},
	YEAR = {1995},
	NUMBER = {4},
	PAGES = {659--673},
	ISSN = {0010-2571,1420-8946},
	MRCLASS = {58E20 (53C23)},
	MRNUMBER = {1360608},
	MRREVIEWER = {Kang\ Zuo},
	DOI = {10.1007/BF02566027},
}

@article {Douglas1956,
	AUTHOR = {Douglas, Jim and Rachford, H. H.},
	TITLE = {On the numerical solution of heat conduction problems in two
	and three space variables},
	JOURNAL = {Trans. Amer. Math. Soc.},
	FJOURNAL = {Transactions of the American Mathematical Society},
	VOLUME = {82},
	YEAR = {1956},
	PAGES = {421--439},
	ISSN = {0002-9947,1088-6850},
	MRCLASS = {65.3X},
	MRNUMBER = {84194},
	MRREVIEWER = {Charles\ Saltzer},
	DOI = {10.2307/1993056},
}

@article {Passty1979,
	AUTHOR = {Passty, Gregory B.},
	TITLE = {Ergodic convergence to a zero of the sum of monotone operators
	in {H}ilbert space},
	JOURNAL = {J. Math. Anal. Appl.},
	FJOURNAL = {Journal of Mathematical Analysis and Applications},
	VOLUME = {72},
	YEAR = {1979},
	NUMBER = {2},
	PAGES = {383--390},
	ISSN = {0022-247X},
	MRCLASS = {47H05 (49A29)},
	MRNUMBER = {559375},
	MRREVIEWER = {R.\ A.\ Plastock},
	DOI = {10.1016/0022-247X(79)90234-8},
}

@article {DaCruz,
	AUTHOR = {Da Cruz Neto, J. X. and Ferreira, O. P. and P\'{e}rez, L. R.
	Lucambio and N\'{e}meth, S. Z.},
	TITLE = {Convex- and monotone-transformable mathematical programming
	problems and a proximal-like point method},
	JOURNAL = {J. Global Optim.},
	FJOURNAL = {Journal of Global Optimization. An International Journal
	Dealing with Theoretical and Computational Aspects of Seeking
	Global Optima and Their Applications in Science, Management
	and Engineering},
	VOLUME = {35},
	YEAR = {2006},
	NUMBER = {1},
	PAGES = {53--69},
	ISSN = {0925-5001,1573-2916},
	MRCLASS = {90C30 (90C52)},
	MRNUMBER = {2232276},
	MRREVIEWER = {Claudia\ A.\ Sagastiz\'{a}bal},
	DOI = {10.1007/s10898-005-6741-9},
}

@article {Rosenbrock,
	AUTHOR = {Rosenbrock, H. H.},
	TITLE = {An automatic method for finding the greatest or least value of
	a function},
	JOURNAL = {Comput. J.},
	FJOURNAL = {The Computer Journal},
	VOLUME = {3},
	YEAR = {1960/61},
	PAGES = {175--184},
	ISSN = {0010-4620},
	MRCLASS = {65.10},
	MRNUMBER = {136042},
	MRREVIEWER = {B.\ A.\ Chartres},
	DOI = {10.1093/comjnl/3.3.175},
}

@article{Dixit2020,
	title={Application of a new accelerated algorithm to regression problems},
	author={Dixit, Avinash and Sahu, Daya Ram and Singh, Amit Kumar and Som, Tanmoy},
	journal={Soft Computing},
	volume={24},
	pages={1539--1552},
	year={2020},
	publisher={Springer},
	doi={10.1007/s00500-019-03984-7}
}

@article{Sahusoft2020,
	title={Applications of accelerated computational methods for quasi-nonexpansive operators to optimization problems},
	author={Sahu, Daya Ram},
	journal={Soft Computing},
	volume={24},
	number={23},
	pages={17887--17911},
	year={2020},
	publisher={Springer},
	doi={10.1007/s00500-020-05038-9}
}

@article {Mordukhovich2012,
	AUTHOR = {Mordukhovich, Boris S. and Nam, Nguyen Mau and Salinas, Jr.,
	Juan},
	TITLE = {Solving a generalized {H}eron problem by means of convex
	analysis},
	JOURNAL = {Amer. Math. Monthly},
	FJOURNAL = {American Mathematical Monthly},
	VOLUME = {119},
	YEAR = {2012},
	NUMBER = {2},
	PAGES = {87--99},
	ISSN = {0002-9890,1930-0972},
	MRCLASS = {49J52 (90C25)},
	MRNUMBER = {2892422},
	DOI = {10.4169/amer.math.monthly.119.02.087},
}

@article {Dixit2022,
	AUTHOR = {Dixit, Avinash and Sahu, D. R. and Gautam, Pankaj and Som, T.},
	TITLE = {Convergence analysis of two-step inertial {D}ouglas-{R}achford
	algorithm and application},
	JOURNAL = {J. Appl. Math. Comput.},
	FJOURNAL = {Journal of Applied Mathematics and Computing},
	VOLUME = {68},
	YEAR = {2022},
	NUMBER = {2},
	PAGES = {953--977},
	ISSN = {1598-5865,1865-2085},
	MRCLASS = {47J22 (47J25 49J40 49J53)},
	MRNUMBER = {4396104},
	MRREVIEWER = {G.\ N.\ Ogwo},
	DOI = {10.1007/s12190-021-01554-5},
}

@article {Agarwal2007,
	AUTHOR = {Agarwal, R. P. and O'Regan, Donal and Sahu, D. R.},
	TITLE = {Iterative construction of fixed points of nearly
	asymptotically nonexpansive mappings},
	JOURNAL = {J. Nonlinear Convex Anal.},
	FJOURNAL = {Journal of Nonlinear and Convex Analysis. An International
	Journal},
	VOLUME = {8},
	YEAR = {2007},
	NUMBER = {1},
	PAGES = {61--79},
	ISSN = {1345-4773,1880-5221},
	MRCLASS = {47H10},
	MRNUMBER = {2314666},
	MRREVIEWER = {Jes\'us\ Garc\'ia-Falset},
}

@article {ShikherCNSNS2024,
	AUTHOR = {Sahu, D. R. and Pitea, Ariana and Sharma, Shikher and Singh,
	Amit Kumar},
	TITLE = {Applications of a variable anchoring iterative method to
	equation and inclusion problems on {H}adamard manifolds},
	JOURNAL = {Commun. Nonlinear Sci. Numer. Simul.},
	FJOURNAL = {Communications in Nonlinear Science and Numerical Simulation},
	VOLUME = {138},
	YEAR = {2024},
	PAGES = {Paper No. 108192},
	ISSN = {1007-5704,1878-7274},
	MRCLASS = {47J05 (47H09 49J40 58C30)},
	MRNUMBER = {4773522},
	MRREVIEWER = {Bhavana\ Deshpande},
	DOI = {10.1016/j.cnsns.2024.108192},
}

@article{ShikherOpt2024,
	author = {D.R. Sahu and Shikher Sharma and J. C. Yao and Xiaopeng Zhao},
	title = {Accelerated iterative splitting methods on {H}adamard manifolds},
	journal = {Optimization},
	volume = {0},
	number = {0},
	pages = {1--38},
	year = {2024},
	publisher = {Taylor \& Francis},
	doi = {10.1080/02331934.2024.2399807}
	}

@article {BergmannJOTA2024,
	AUTHOR = {Bergmann, Ronny and Ferreira, Orizon P. and Santos,
	Elianderson M. and Souza, Joao Carlos O.},
	TITLE = {The difference of convex algorithm on {H}adamard manifolds},
	JOURNAL = {J. Optim. Theory Appl.},
	FJOURNAL = {Journal of Optimization Theory and Applications},
	VOLUME = {201},
	YEAR = {2024},
	NUMBER = {1},
	PAGES = {221--251},
	ISSN = {0022-3239,1573-2878},
	MRCLASS = {90C30 (49N15 53B21 90C26)},
	MRNUMBER = {4734031},
	MRREVIEWER = {Junfeng\ Chen},
	DOI = {10.1007/s10957-024-02392-8},
}

@article {Hieu2018COA,
	AUTHOR = {Thao, Nguyen Hieu},
	TITLE = {A convergent relaxation of the {D}ouglas-{R}achford algorithm},
	JOURNAL = {Comput. Optim. Appl.},
	FJOURNAL = {Computational Optimization and Applications. An International
	Journal},
	VOLUME = {70},
	YEAR = {2018},
	NUMBER = {3},
	PAGES = {841--863},
	ISSN = {0926-6003,1573-2894},
	MRCLASS = {47J25 (49K40 49M05 49M27 65K05 65K10 90C26)},
	MRNUMBER = {3808698},
	MRREVIEWER = {Scott\ B.\ Lindstrom},
	DOI = {10.1007/s10589-018-9989-y},
}

@article {Samir2023,
	AUTHOR = {Adly, Samir and Attouch, Hedy and Le, Manh Hung},
	TITLE = {First order inertial optimization algorithms with threshold
	effects associated with dry friction},
	JOURNAL = {Comput. Optim. Appl.},
	FJOURNAL = {Computational Optimization and Applications. An International
	Journal},
	VOLUME = {86},
	YEAR = {2023},
	NUMBER = {3},
	PAGES = {801--843},
	ISSN = {0926-6003,1573-2894},
	MRCLASS = {34A60 (34D05 34G25 70F40 90C25 90C48)},
	MRNUMBER = {4666867},
	MRREVIEWER = {Shiqian\ Ma},
	DOI = {10.1007/s10589-023-00509-9},
}

@article {Attouch2019,
	AUTHOR = {Attouch, Hedy and Cabot, Alexandre},
	TITLE = {Convergence of a relaxed inertial forward-backward algorithm
	for structured monotone inclusions},
	JOURNAL = {Appl. Math. Optim.},
	FJOURNAL = {Applied Mathematics and Optimization},
	VOLUME = {80},
	YEAR = {2019},
	NUMBER = {3},
	PAGES = {547--598},
	ISSN = {0095-4616,1432-0606},
	MRCLASS = {49M37 (65K05 65K10 90C25 90C48)},
	MRNUMBER = {4026592},
	MRREVIEWER = {B\brevgrv ang\ C\^ong\ V\~u},
	DOI = {10.1007/s00245-019-09584-z},
}

@article {Dong2015comments,
	AUTHOR = {Dong, Yunda},
	TITLE = {Comments on ``{T}he proximal point algorithm revisited''
	[{MR}3193802]},
	JOURNAL = {J. Optim. Theory Appl.},
	FJOURNAL = {Journal of Optimization Theory and Applications},
	VOLUME = {166},
	YEAR = {2015},
	NUMBER = {1},
	PAGES = {343--349},
	ISSN = {0022-3239,1573-2878},
	MRCLASS = {47H05 (47J25 90C25 90C48)},
	MRNUMBER = {3366118},
	DOI = {10.1007/s10957-014-0685-5},
}

@article {Barshad2023,
	AUTHOR = {Barshad, Kay and Gibali, Aviv and Reich, Simeon},
	TITLE = {Unrestricted {D}ouglas-{R}achford algorithms for solving
	convex feasibility problems in {H}ilbert space},
	JOURNAL = {Optim. Methods Softw.},
	FJOURNAL = {Optimization Methods \& Software},
	VOLUME = {38},
	YEAR = {2023},
	NUMBER = {4},
	PAGES = {655--667},
	ISSN = {1055-6788,1029-4937},
	MRCLASS = {47J25 (47H09 47H10 90C25 90C48)},
	MRNUMBER = {4616448},
	MRREVIEWER = {Matthew\ K.\ Tam},
	DOI = {10.1080/10556788.2022.2157003},
}

@article {ReichFT2023,
	AUTHOR = {Reich, Simeon and Tuyen, Truong Minh},
	TITLE = {The generalized {F}ermat-{T}orricelli problem in {H}ilbert
	spaces},
	JOURNAL = {J. Optim. Theory Appl.},
	FJOURNAL = {Journal of Optimization Theory and Applications},
	VOLUME = {196},
	YEAR = {2023},
	NUMBER = {1},
	PAGES = {78--97},
	ISSN = {0022-3239,1573-2878},
	MRCLASS = {47H05 (47H09 47N10 49J53 90C25)},
	MRNUMBER = {4530150},
	MRREVIEWER = {Oluwatosin\ Temitope\ Mewomo},
	DOI = {10.1007/s10957-022-02113-z},
}

\end{document}